\documentclass{amsart}
\usepackage{tikz}
\usepackage{hyperref}   
\usepackage{amssymb,xcolor}
\usepackage{stmaryrd} 

\usepackage{tcolorbox}
\usepackage{mathtools}
\mathtoolsset{showonlyrefs}
\usepackage[shortlabels]{enumitem} 

\theoremstyle{plain}
\newtheorem{theorem}{Theorem}[section]
\theoremstyle{plain}

\theoremstyle{plain}
\newtheorem{lemma}[theorem]{Lemma}
\theoremstyle{remark}
\newtheorem{remark}[theorem]{Remark}
\theoremstyle{plain}
\newtheorem{proposition}[theorem]{Proposition}
\newtheorem{corollary}[theorem]{Corollary}
\numberwithin{equation}{section}
\theoremstyle{plain}

\newtcbox{\redbox}{colback=red!5!white,colframe=red!75!black, tcbox raise base}

\usepackage{import}
\usepackage{xifthen}
\usepackage{pdfpages}
\usepackage{transparent}
\usepackage{subcaption}

\begin{document}
\title[Scalar curvature comparison]{Scalar curvature comparison of rotationally symmetric sets}

\author{Xiaoxiang Chai}
\address{Department of Mathematics, POSTECH, Pohang, Gyeongbuk, South Korea}
\email{xxchai@kias.re.kr, xxchai@postech.ac.kr}

\author{Gaoming Wang}
\address{Department of Mathematics, Cornell University, Ithaca, NY, 14853, and Yau Mathematical Sciences Center, Tsinghua University, Beijing, China}
\email{gmwang@mail.tsinghua.edu.cn}

\begin{abstract}
  Let $(M, g)$ be a compact 3-manifold with nonnegative scalar curvature $R_g
  \geq 0$. The boundary $\partial M$ is diffeomorphic to the boundary of
  a rotationally symmetric and weakly convex body $\bar{M}$ in $\mathbb{R}^3$.
  We call $(\bar{M}, \delta)$ a model or a reference. Let $H_{\partial M}$ and
  $\bar{H}_{\partial M}$ be respectively the mean curvatures of $\partial M$
  in $(M, g)$ and $\partial M$ in $(\bar{M}, \delta)$, $\sigma$ and
  $\bar{\sigma}$ be the induced metric from $g$ and $\delta$. We show that for
  some classes of $\partial M$, if $H_{\partial M} \geq \bar{H}_{\partial
  M}$, $\sigma \geq \bar{\sigma}$ and the dihedral angles at the nonsmooth
  part of $\partial M$ are no greater than the model, then $M$ is flat. We
  also generalize this result to the hyperbolic case and some spaces with
  $\mathbb{S}^1$-symmetry. Our approach is inspired by Gromov and uses stable capillary surface with varying prescribed contact angles.
\end{abstract}



{\maketitle}

\section{Introduction}\label{introduction}




Let $(M^3, g)$ be a compact Riemannian 3-manifold with
nonnegative scalar curvature $R_g$. Assume its boundary $\partial M$ is diffeomorphic to some embedded surface in $\mathbb{R}^3$. Using this diffeomorphism, we can identify $\partial M$ with its image in $\mathbb{R}^3$ and we can also identify a tubular neighborhood of $\partial M$ with a tubular neighborhood of the image. Now let 
   $H_{\partial M}$ denote the mean curvature of $\partial M$ in $M$ computed with respect to $g$
   and
   $\bar{H}_{\partial M}$ denote the mean curvature  of $\partial M$ in $M$
   computed with respect to the Euclidean metric $\delta$, and let $\sigma$ and $\bar{\sigma}$ be the induced metric of $\partial M$ from $g$
and $\delta$, Gromov (see {\cite[Section 5.8.1]{gromov-four-2021}}; 
\text{{\itshape{spin}}}-\text{{\itshape{extremality}}} \text{{\itshape{of}}}
\text{{\itshape{doubly}}} \text{{\itshape{punctured}}}
\text{{\itshape{balls}}}) showed the following (not rigorously) for the
standard unit ball via stable capillary surface with varying prescribed contact angles (in his terminology, stable $\mu$-bubble).
\begin{theorem}
  \label{gromov}Assume that the boundary $\partial M$ of $(M^3, g)$ is diffeomorphic to a standard 2-sphere, and satisfies the comparisons of the mean curvature $H_{\partial M} \geq \bar{H}_{\partial M}$ and
of the induced metrics $\sigma \geq \bar{\sigma}$,
  then $(M, g)$ is flat.
\end{theorem}

For the sake of exposition, Theorem \ref{gromov} is actually only a simplified version of what Gromov have considered. On the other hand, Gromov have used regularity of stable capillary surface up to dimesion 8, which is actually known to hold in dimension 3 (see \cite{de-philippis-regularity-2015}). This explains the three dimensional restriction in Theorem \ref{gromov}. We also note a more general version \cite[Theorem 3.2]{wang-gromovs-2023} for strictly convex smooth Euclidean domains using spinorial approach.

Theorem \ref{gromov} is a natural boundary analog of Llarull's theorem \cite{llarull-sharp-1998} regarding the scalar curvature rigidity of the unit sphere. See some recent developments \cite{hu-rigidity-2022-arxiv}, \cite{hirsch-rigid-2022}. Also, the case when $\sigma=\bar{\sigma}$ in Theorem \ref{gromov} was proved by \cite{miao-positive-2002}, \cite{shi-positive-2002} via a reduction to Shoen-Yau's positive mass theorem \cite{schoen-proof-1979}. Miao and Shi-Tam allowed more general $\partial M$. See also analogs of their results in the hyperbolic space
{\cite{bonini-positive-2008,shi-rigidity-2007}}.

Results involving the scalar curvature rigidity of compact manifolds in close relation with Theorem \ref{gromov} dated back to the Schoen-Yau \cite{schoen-existence-1979} and Gromov-Lawson \cite{gromov-positive-1983} resolution of the Geroch conjecture. The Geroch conjecture states that an $n$-torus does not admit nonnegative scalar curvature metrics except the flat one (see also \cite{stern-scalar-2022-1}). Lohkamp \cite{lohkamp-scalar-1999} observed a closed relationship between the Geroch conjecture and the positive mass theorem. The relationship carries over to the convex polytopes and the scalar curvature rigidity of polytopes is called the Gromov dihedral rigidity conjecture \cite{gromov-dirac-2014}. The conjecture was originally formulated to find a weak notion of scalar curvature. Towards the Gromov dihedral rigidity conjecture, we have seen a lot of progress recently using various methods. For instance, stable minimal surface \cite{li-polyhedron-2020,li-dihedral-2020-1,li-dihedral-2020,chai-dihedral-2022-arxiv}, harmonic function method \cite{chai-scalar-2022,tsang-dihedral-2021-arxiv} and spinors \cite{wang-gromovs-2022-arxiv,brendle-scalar-2023,brendle-rigidity-2023,wang-dihedral-2023,wang-gromovs-2023}.

We view Theorem \ref{gromov} as a scalar curvature rigidity result of a compact
3-manifold bounded by a standard 2-sphere, and it is natural to ask the same
questions for 3-manifolds bounded by other surfaces. We consider compact 3-manifolds
which are bounded by a weakly convex surface that is also rotationally
symmetric with respect to the $x^3$-coordinate axis in $\mathbb{R}^3$. Without
loss of generality, we assume that the surface $\partial M$ lies between the
two coordinate planes
\[ P_{\pm} = \{x \in \mathbb{R}^3 : x^3 = t_\pm \} \]
and $\partial M \cap P_{\pm}$ are nonempty. Here $t_-<t_+$. For convenience, we take $t_+=0$.
We fix the points $p_{\pm} = (0,
0, t_\pm )$ and we call $p_+$ ($p_-$) north (south) pole. We call $\partial M
\cap \{x^3 = s\}$ a boundary $s$-level set.

We have three cases depending on the geometry of $\partial M \cap P_{\pm}$:

1. The set $\partial M \cap P_{\pm}$ is a disk;

2. The set $\partial M \cap P_{\pm}$ contains only $p_{\pm}$ and $\partial M$
is conical at $p_{\pm}$;

3. The set $\partial M \cap P_{\pm}$ contains only $p_{\pm}$ and $\partial M$
is at least $C^1$ at $p_{\pm}$.

If $\partial M$ is smooth at $p_\pm$, we say $\partial M$ is \textit{spherical} at $p_{\pm}$ if there exist a smooth function $\phi(t)$ near $0$ with $\phi(0)=0$ and $\phi'(0)>0$ such that $\partial M$ can be written as a graph of function $\phi(x_1^2+x_2^2)$ near $p_\pm $. This obviously includes the important case when $M$ is a standard $3$-sphere.

\begin{figure}[ht]
    \centering
	\begingroup
	\def\svgwidth{0.8\columnwidth}
\begingroup%
  \makeatletter%
  \providecommand\color[2][]{%
    \errmessage{(Inkscape) Color is used for the text in Inkscape, but the package 'color.sty' is not loaded}%
    \renewcommand\color[2][]{}%
  }%
  \providecommand\transparent[1]{%
    \errmessage{(Inkscape) Transparency is used (non-zero) for the text in Inkscape, but the package 'transparent.sty' is not loaded}%
    \renewcommand\transparent[1]{}%
  }%
  \providecommand\rotatebox[2]{#2}%
  \newcommand*\fsize{\dimexpr\f@size pt\relax}%
  \newcommand*\lineheight[1]{\fontsize{\fsize}{#1\fsize}\selectfont}%
  \ifx\svgwidth\undefined%
    \setlength{\unitlength}{584.51500119bp}%
    \ifx\svgscale\undefined%
      \relax%
    \else%
      \setlength{\unitlength}{\unitlength * \real{\svgscale}}%
    \fi%
  \else%
    \setlength{\unitlength}{\svgwidth}%
  \fi%
  \global\let\svgwidth\undefined%
  \global\let\svgscale\undefined%
  \makeatother%
  \begin{picture}(1,0.38234235)%
    \lineheight{1}%
    \setlength\tabcolsep{0pt}%
    \put(0,0){\includegraphics[width=\unitlength,page=1]{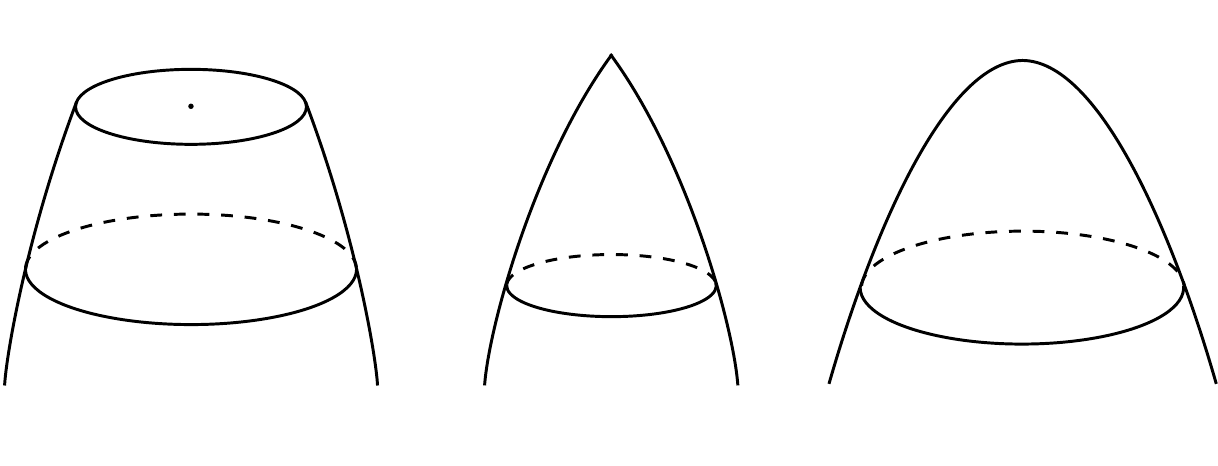}}%
    \put(0.14671637,0.34338091){\color[rgb]{0,0,0}\makebox(0,0)[lt]{\lineheight{0}\smash{\begin{tabular}[t]{l}$p_+$\end{tabular}}}}%
    \put(0.49059121,0.3536459){\color[rgb]{0,0,0}\makebox(0,0)[lt]{\lineheight{0}\smash{\begin{tabular}[t]{l}$p_+$\end{tabular}}}}%
    \put(0.82676736,0.3536459){\color[rgb]{0,0,0}\makebox(0,0)[lt]{\lineheight{0}\smash{\begin{tabular}[t]{l}$p_+$\end{tabular}}}}%
    \put(0.10601586,0.01062753){\color[rgb]{0,0,0}\makebox(0,0)[lt]{\lineheight{0}\smash{\begin{tabular}[t]{l}Case 1\end{tabular}}}}%
    \put(0.46785429,0.01062753){\color[rgb]{0,0,0}\makebox(0,0)[lt]{\lineheight{0}\smash{\begin{tabular}[t]{l}Case 2\end{tabular}}}}%
    \put(0.80916278,0.01062753){\color[rgb]{0,0,0}\makebox(0,0)[lt]{\lineheight{0}\smash{\begin{tabular}[t]{l}Case 3\end{tabular}}}}%
    \put(0,0){\includegraphics[width=\unitlength,page=2]{various-types-of-p.pdf}}%
  \end{picture}%
\endgroup%

	\endgroup

    \caption{Various types of $p_+$.}
    \label{fig:various-types-of-p}
\end{figure}
It might happen that at $p_+$ and $p_-$ have different geometries.

We establish the following scalar curvature rigidity.

\begin{theorem}
  \label{main}Let $(M^3, g)$ be a compact 3-manifold with nonnegative scalar
  curvature such that its boundary $\partial M$ is diffeomorphic to a weakly
  convex rotationally symmetric surface in $\mathbb{R}^3$. The boundary
  $\partial M$ bounds a region $\bar{M}$ (which we call a model or a
  reference) in $\mathbb{R}^3$, let the induced metric of the flat metric be
  $\bar{\sigma}$ and the induced metric of $g$ on $\partial M$ be $\sigma$. We
  assume that $\sigma \geq \bar{\sigma}$ and $H_{\partial M} \geq
  \bar{H}_{\partial M}$ on $\partial M \cap \{x \in \mathbb{R}^3 : t_- < x^3 <t_+\}$ and assume $\partial M$ satisfies one of the following three conditions near $p_\pm$.
  \begin{enumerate}
    \item \label{case slab}$\partial M \cap P_{\pm}$ is a disk. For this case, we further
    assume that $H_{\partial M} \geq 0$ at $\partial M \cap P_{\pm}$ and
    the dihedral angles forming by $P_{\pm}$ and $\partial M\backslash (P_+
    \cup P_-)$ are no greater than the Euclidean reference.
    
    \item \label{case conical}$\partial M$ is conical at $p_{\pm}$.
	\item \label{case spherical} $\partial M$ is spherical at $p_\pm $.
  \end{enumerate}
  Then $(M, g)$ is flat.
\end{theorem}

The model when $\bar{M}$ is a standard unit ball (Theorem \ref{gromov}) is a special case of our third case.
Another interesting model that motivates our main theorem is the right circular cone whose Gauss-Bonnet type inequalities were discussed by Gromov (see {\cite[Section 5.9]{gromov-four-2021}}). We have the following scalar curvature rigidity.
\begin{corollary}
Assume that the boundary $\partial M$ of a compact 3-manifold $(M^3, g)$ with nonnegative scalar curvature is diffeomorphic to a right circular cone in $\mathbb{R}^3$, and satisfies the comparisons of the mean curvature $H_{\partial M} \geq \bar{H}_{\partial M}$, the dihedral angles forming by base face and side face are no greater than the model in the Euclidean space, and
of the induced metrics $\sigma \geq \bar{\sigma}$,
  then $(M, g)$ is flat.    
\end{corollary}

Gromov proposed the use of
a minimal capillary surface to Theorem \ref{gromov}.
We use Gromov's argument and adapt it to
handle various geometries of $\partial M \cap P_{\pm}$. When $\partial M \cap
P_{\pm}$ are disks, we know by a maximum principle they are natural barriers for the existence
of minimal capillary surfaces.  

As noted before, the case that $\partial M$ is spherical at either $p_{\pm}$ includes Theorem \ref{gromov} as a special case. However, Gromov did not have a strategy on how to construct a nontrivial minimizing surface  (to the functional \eqref{action}). The case when $\partial M$ at $p_\pm$ are conical at both $p_\pm$ has the same issue. It turns out that the presence of a spherical or a conical point at $\pm$ poses essential difficulty. By constructing a constant mean curvature foliation near the
$p_{\pm}$ if $\partial M$ is conical or spherical and showing that any leaf of the foliation would serve as a barrier, we are able to reduce to the case when $\partial M\cap P_\pm$ are both disks. This is the major constribution of our work and the construction has the potential to apply to other problems.
In particular, this leads to a rigorous proof of Theorem \ref{gromov} of Gromov which provides an approach different from \cite[Theorem 3.2]{wang-gromovs-2023}.

Using similar procedures, we are also able to generalize to the hyperbolic case using the upper half
space model and $\mathbb{S}^1$-invariant space forms with nonpositive curvature.
We remark that Gromov also discussed the case when $\bar{M}$ is a geodesic
ball in the hyperbolic space (see {\cite[Section 3.5]{gromov-four-2021}};
\text{{\itshape{on}}} \text{{\itshape{non-spin}}} \text{{\itshape{manifolds}}}
\text{{\itshape{and}}} \text{{\itshape{on}}} $\sigma < 0$).

\

The article is organized as follows:

In Section \ref{sec variational problem}, we introduce the variational problem
whose minimiser is vital to our proof. In Section \ref{slab case}, we give the
proof of Case \ref{case slab} of Theorem \ref{main} by studying the minimiser
of the variational problem introduced in Section \ref{sec variational
problem}. In Section \ref{conical case}, we give a proof of the Case \ref{case
conical} of Theorem \ref{main} by constructing a constant mean curvature
foliation near the conical point. In Section \ref{sec:spherical}, we prove the Case \ref{case spherical} of Theorem \ref{main} by similar construction of a constant mean curvature foliation near the spherical point. In Section \ref{generalizations}, we discuss
generalizations to the hyperbolic space (Theorem \ref{hyperbolic}) and spaces
with a circle factor (Theorem \ref{circle factor}). We also discuss unsolved
cases.

\section*{{\textbf{Acknowledgements}}}

Part of this work was carried out when X. Chai was a member of Korea Institute for Advanced Study under the individual grant MG074402. Currently, X. Chai is partially supported by National Research Foundation of Korea grant No. 2022R1C1C1013511.
G. Wang would like to thank to Prof. Xin Zhou for his support and encouragement.

\section{The variational problem}\label{sec variational problem}

In this section, we introduce the functional \eqref{action} and study its variational properties.
\subsection{Notations}We set up some notations, in particular, various labeling
of the normal vectors. Let $E \subset M$ such that $\partial E \cap M$ is a regular surface with
boundary, we name it $\Sigma$. We set
\begin{itemize}
  \item
  $N$: unit normal vector of $\Sigma$ in $M$ pointing inside $E$;
  
  \item $\nu$: unit normal vector of $\partial \Sigma$ in $\Sigma$ pointing
  outside of $\Sigma$;
  
  \item $\eta$: unit normal vector of $\partial \Sigma$ in $\partial M$
  pointing outside of $\partial E \cap \partial M$;
  
  \item $X$: unit normal vector of $\partial M$ in $M$ pointing outside of
  $M$
  \item $\gamma$: the contact angle formed by $\Sigma$ and $\partial M$, and the magnitude of the angle is given by $\cos \gamma = \langle X, N \rangle$.
  \item $\left<X,Y\right>=\left<X,Y\right>_g$: the inner product of vector $X$ and $Y$ under metric $g$.
  \item $X\cdot Y=\left<X,Y\right>_\delta$: the inner product of vector $X$ and $Y$ under Euclidean metric $\delta$.
\end{itemize}
See Figure \ref{fig:naming-of-various-vectors}.
We use a bar on every quantity to denote that the quantity is computed with respect to the flat metric except $\bar{\gamma}$ which is defined by \eqref{angle def}.

\subsection{The functional and its stability}

Let $\bar{\gamma} \in [0, \pi]$ be the contact angle formed by the coordinate
planes $\{x^3 = s\}$ ($t_- < s < t_+$) with the boundary $\partial M$ under the flat metric,
that is,
\begin{equation}
  \cos \bar{\gamma} = \langle \bar{X}, \tfrac{\partial}{\partial x^3}
  \rangle_{\delta} . \label{angle def}
\end{equation}
Obviously, $\cos \bar{\gamma}$ is a function on $\partial M$ to $\mathbb{R}$.

\begin{figure}[ht]
    \centering
	\begingroup
	\def\svgwidth{0.5\columnwidth}
	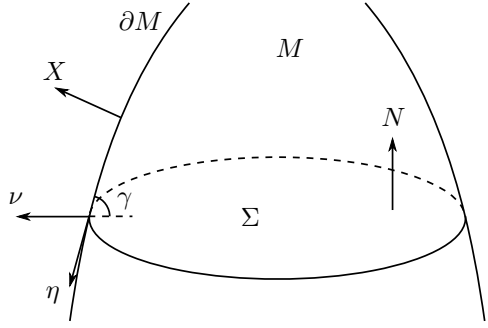
	\endgroup

    \caption{Naming of various vectors}
    \label{fig:naming-of-various-vectors}
\end{figure}

Consider the following functional
\begin{equation}
  I (E) = | \partial E \cap \ensuremath{\operatorname{int}}M| -
  \int_{\partial E \cap \partial M} \cos \bar{\gamma} \label{action}
\end{equation}
and the variational problem
\begin{equation}
  \mathcal{I}= \inf \{I (E) : E \in \mathcal{E}\}, \label{infimum}
\end{equation}
where $\mathcal{E}$ is the collection of contractible open subsets $E'$ such
that $\partial M \cap P_+ \subset E'$ and $\partial M \cap P_- \subset E'$. For an interesting construction of capillary surface with prescribed mean curvature via the min-max approach, see \cite{zhou-existence-2020,li-min-max-capillary-2021-arxiv}. If
the solution to \eqref{infimum} is regular, then $\Sigma := \partial E \cap
\ensuremath{\operatorname{int}}M$ is minimal and meets $\partial M$ with the
prescribed angles $\bar{\gamma}$.

Let $\Sigma$ be a surface with boundary $\partial \Sigma \subset \partial M$
away from $P_{\pm}$. Then $\partial \Sigma$ separates $\partial M$ into two
components and the component closer to $P_+$ is just $E$. We can reformulate
the functional \eqref{action} in terms of $\Sigma$. We define
\begin{equation}
  F (\Sigma) : = I (E) = | \Sigma | - \int_{\Omega} \cos \bar{\gamma} ,
  \label{action in terms of surface}  
\end{equation}
where $\Omega = \partial E \cap \partial M$. 
Let $\phi_t$ be a family of diffeomorphisms $\phi_t : \Sigma \to M$ such that
$\phi_t (\partial \Sigma) \subset \partial M$ and $\phi_0 (\Sigma) = \Sigma$.
Let $\Sigma_t$ be $\phi_t (\Sigma)$ and $E_t$ be the corresponding component
separated by $\Sigma_t$. Let $Y$ be the vector $\tfrac{\partial
\phi_t}{\partial t}$. Define $\mathcal{A} (t)$ by
\begin{equation}
  \mathcal{A} (t) = F (\Sigma_t) .
\end{equation}
Letting $f = \langle Y, N \rangle$, by first variation formula,
\begin{equation}
  \mathcal{A}' (0) = \int_{\Sigma} H f + \int_{\partial \Sigma} \langle Y, \nu
  - \eta \cos \bar{\gamma} \rangle . \label{first variation}
\end{equation}
We know that $\Sigma$ is a critical point of \eqref{action in terms of
surface} if and only if $H \equiv 0$ on $\Sigma$ and $\nu - \eta \cos
\bar{\gamma}$ is normal to $\partial M$ (that is, the angle formed by the
vectors $N$ and $X$ is $\bar{\gamma}$). We call $\Sigma$ a
\text{{\itshape{minimal}}} \text{{\itshape{capillary}}} surface if $\Sigma$ is
a critical point of \eqref{action in terms of surface}.

Assume that $\Sigma$ is minimal capillary, we have the second variation
formula
\begin{equation}
  \mathcal{A}'' (0) = Q (f, f) := - \int_{\Sigma} (f \Delta f + (|A|^2
  +\ensuremath{\operatorname{Ric}}(N)) f^2) + \int_{\partial \Sigma} f
  (\tfrac{\partial f}{\partial \nu} - q f), \label{stability}
\end{equation}
where $q$ is defined to be
\begin{equation}
  q = \tfrac{1}{\sin \bar{\gamma}} A_{\partial M} (\eta, \eta) - \cot
  \bar{\gamma} A (\nu, \nu) + \tfrac{1}{\sin^2 \bar{\gamma}} \partial_{\eta}
  \cos \bar{\gamma} . \label{q}
\end{equation}
We call $\Sigma$ \text{{\itshape{stable}}} \text{{\itshape{minimal}}}
\text{{\itshape{capillary}}} if $\mathcal{A}'' (0) \geq 0$ for all $f \in
C^{\infty} (\Sigma)$. Define
\begin{equation}
  L f = - \Delta f - (\ensuremath{\operatorname{Ric}}(N) + |A|^2) f,
\end{equation}
and
\begin{equation}
  B f = \tfrac{\partial f}{\partial \nu} - q f.
\end{equation}
We call $L$ the \text{{\itshape{stability}}} (or Jacobi)
\text{{\itshape{operator}}} and $B$ the \text{{\itshape{boundary}}}
\text{{\itshape{stability}}} \text{{\itshape{operator}}}.

\begin{proof}[Proof of \eqref{q}]
  Let $Y$ be a vector field tangent to $\partial M$. The key part of the computation really reduces to the calculations of the variations of the mean curvature $\delta_Y H$ and the variations of the angle difference $\delta_Y (\langle X, N \rangle - \cos \bar{\gamma})$. The first variation of the mean curvature is given by
  \begin{equation}
    \delta_Y H = \nabla_{Y^{\top}} H - \Delta  f -
    (\ensuremath{\operatorname{Ric}}(N) + |A|^2) f, 
  \end{equation}
  where $Y^\top$ is the component of $Y$ tangential to $M$.
   Let $Y^{\partial M}$ be the projection of $Y$ to $\partial M$. We have
  \[ \delta_Y (\langle X, N \rangle - \cos \bar{\gamma}) = - \sin \bar{\gamma}
     \tfrac{\partial f}{\partial \nu} + \sin \bar{\gamma} q f +
     \nabla_{Y^{\partial M}} (\langle X, N \rangle - \cos \bar{\gamma}) . \]
  Indeed, we have $Y-Y^{\partial M} = - \tfrac{f}{\sin
  \bar{\gamma}} \bar{\nu}$. The term
  \begin{equation}
    \delta_Y \langle X, N \rangle = - \sin \bar{\gamma} \tfrac{\partial
    f}{\partial \nu} + [A_{\partial M} (\eta, \eta) - \cos \bar{\gamma} A
    (\nu, \nu)] f + \nabla_{Y^{\partial M}} \langle X, N \rangle \label{angle first
    variation}
  \end{equation}
  is already computed in {\cite[Appendix]{ros-stability-1997}}, we only have
  to compute $\delta_Y \cos \bar{\gamma}$. 
  We have
  \begin{equation}
    \delta_Y \cos \bar{\gamma} = - \tfrac{f}{\sin \bar{\gamma}}
    \partial_{\eta} \cos \bar{\gamma} + \nabla_{Y^{\partial M}} \cos \bar{\gamma} .
    \label{background angle first variation}
  \end{equation}
    So we have obtained the expression of our $q$.
\end{proof}

\subsection{Rewrite of the second variation}
We rewrite the second variation. Using Schoen-Yau's rewrite (essentially Gauss equation, see
{\cite{schoen-proof-1979}}),
\begin{equation}
  |A|^2 +\ensuremath{\operatorname{Ric}} (\nu) = \tfrac{1}{2} R_g - K +
  \tfrac{1}{2} |A|^2 + \tfrac{1}{2} H^2 \label{sy}
\end{equation}
where $K$ is the Gauss curvature of $\Sigma$. We also have a boundary analog of this rewrite as follows. See 
\cite[Lemma 3.1]{ros-stability-1997} and {\cite[(4.13)]{li-polyhedron-2020}}.

\begin{lemma}
	\label{lem:IIandMean}
  Along the boundary $\partial \Sigma$, we have that
  \begin{equation}
    \tfrac{1}{\sin \bar{\gamma}} A_{\partial M} (\eta, \eta) - \cot
    \bar{\gamma} A (\nu, \nu) = - H \cot \bar{\gamma} + \tfrac{H_{\partial
    M}}{\sin \bar{\gamma}} - \kappa \label{boundary sy},
  \end{equation}
  where $\kappa$ is the geodesic curvature of $\partial \Sigma$ in $\Sigma$.
\end{lemma}

\begin{proof}
  We show by direct calculation. Since $\Sigma$ is of mean curvature $H$, so
  $A (\nu, \nu) = H - A (T, T)$. So
\begin{align}
& - \cos \bar{\gamma} A (\nu, \nu) + \kappa \sin \bar{\gamma} \\
= & - H \cos \bar{\gamma} + \cos \bar{\gamma} A (T, T) + \langle \nabla_T
\nu, T \rangle \sin \bar{\gamma} \\
= & - H \cos \bar{\gamma} - \langle \nabla_T T, \sin \bar{\gamma} \nu +
\cos \bar{\gamma} N \rangle \\
= & - H \cos \bar{\gamma} - \langle \nabla_T T, X \rangle \\
= & - H \cos \bar{\gamma} + A_{\partial M} (T, T) .
\end{align}
  Since $T$ and $\eta$ form an orthonormal basis of $\partial M$, we have that
  \begin{equation}
    A_{\partial M} (T, T) = H_{\partial M} - A_{\partial M} (\eta, \eta)
  \end{equation}
  and hence
  \begin{equation}
    - \cos \bar{\gamma} A (\nu, \nu) + \kappa \sin \bar{\gamma} = - H \cos
    \bar{\gamma} + H_{\partial M} - A_{\partial M} (\eta, \eta) .
  \end{equation}
  Dividing both sides by $\sin \bar{\gamma}$ finishes the proof of the lemma.
\end{proof}

\section{Weakly convex sets between coordinate planes}\label{slab case}

In this section, we deal with Case \ref{case slab} of Theorem \ref{main}, that is, both $\partial M\cap P_\pm$ are disks. This
is the easiest case since the boundary $\partial M \cap P_{\pm}$ provide
natural barriers for the minimal capillary surface. Provided the existence of a
minimiser to \eqref{action}, it is therefore quite direct to obtain the rigidity
via constructing a CMC capillary foliation near the minimiser.

\subsection{Minimiser}

Since $H_{\partial M} \geq 0$ at $\partial M \cap P_{\pm}$ and the
dihedral angles forming by $P_{\pm}$ and $\partial M\backslash (P_+ \cup P_-)$
are no greater than the Euclidean reference, there exists a minimiser $E$ to
the problem \eqref{action}. By the classical interior maximum principle and
{\cite[Proposition 2.2]{li-polyhedron-2020}}, $\partial E$ either is either
$P_+$, $P_-$ or lies away from $P_{\pm}$.

Let $\Sigma = \partial E \cap \ensuremath{\operatorname{int}}M$, we know that
$\Sigma$ is a stable minimal capillary surface. By {\cite[Theorem
1.3]{de-philippis-regularity-2015}}, $\Sigma$ is free of singularities when
the dimension of $M$ is 3. Taking $f$ to be 1 in \eqref{stability}, and using
\eqref{sy} and \eqref{boundary sy}, we have that
\begin{equation}
  \int_{\Sigma} K + \int_{\partial \Sigma} \kappa \geq \int_{\partial
  \Sigma} \tfrac{H_{\partial M}}{\sin \bar{\gamma}} + \tfrac{1}{\sin^2
  \bar{\gamma}} \partial_{\eta} \cos \bar{\gamma} + \tfrac{1}{2} \int_{\Sigma}
  R_g + |A|^2 . \label{taking constant}
\end{equation}
Using the bounds $R_g + |A|^2 \geq 0$, $H_{\partial M} \geq
\bar{H}_{\partial M}$ and the Gauss-Bonnet theorem,
\begin{equation}
  2 \pi \chi (\Sigma) \geq \int_{\partial \Sigma} \left(
  \tfrac{H_{\partial M}}{\sin \bar{\gamma}} + \tfrac{1}{\sin^2 \bar{\gamma}}
  \partial_{\eta} \cos \bar{\gamma} \right) \mathrm{d} \lambda . \label{after
  gauss-bonnet}
\end{equation}
In the above, we have made the line element $\mathrm{d} \lambda$ under the
metric $g$ explicit, because we have to deal with $\mathrm{d}
\bar{\lambda}$ under the flat metric as well.

\subsection{Comparison of $\sigma$ and $\bar{\sigma}$}

We show the left-hand side of \eqref{after gauss-bonnet} has a favorable lower
bound.

\begin{lemma}
  \label{lemma lower bound over separating curve}If $\partial \Sigma$ is any
  closed curve separating $\partial M \cap P_{\pm}$, then
  \begin{equation}
    \int_{\partial \Sigma} \left( \frac{H_{\partial M}}{\sin \bar{\gamma}} -
    \frac{1}{\sin \bar{\gamma}}  \frac{\partial \bar{\gamma}}{\partial \eta}
    \right) \mathrm{d} \lambda \geq 2 \pi . \label{lower bound for
    separating curve}
  \end{equation}
\end{lemma}

\begin{proof}
  First, since $\bar{\gamma} \in (0, \pi)$, we have
  \[ \tfrac{H_{\partial M}}{\sin \bar{\gamma}} \geq
     \tfrac{\bar{H}_{\partial M}}{\sin \bar{\gamma}} . \]
  Let $\ell_s$ be the curve $\partial M \cap \{x^3 = s\}$ and $\bar{\eta}$
  unit normal of $\ell_s$ in $\partial M$ pointing in the direction of $-
  \tfrac{\partial}{\partial x^3}$. Since $M$ is rotationally symmetric,
  $\ell_s$ is a circle and $\bar{\gamma}$ is constant on $\ell_s$, so if
  $\tfrac{\partial \bar{\gamma}}{\partial \bar{\eta}} = 0$, we obviously have
  $\tfrac{\partial \bar{\gamma}}{\partial \eta} = \tfrac{\partial
  \bar{\gamma}}{\partial \bar{\eta}} = 0$. If $\tfrac{\partial
  \bar{\gamma}}{\partial \bar{\eta}} \neq 0$, then by weak convexity, we have
  that $\tfrac{\partial \bar{\gamma}}{\partial \bar{\eta}} > 0$ and the
  gradient of $\bar{\gamma}$ under $\bar{\sigma}$ is parallel to $\bar{\eta}$.
  Since
  \[ | \eta |_{\bar{\sigma}} \leq | \eta |_{\sigma} = 1, \]
  so
  \begin{equation}
    \tfrac{\partial \bar{\gamma}}{\partial \bar{\eta}} \geq
    \tfrac{\partial \bar{\gamma}}{\partial \eta} \label{derivative of angle comparison}
  \end{equation}
  and
  \begin{equation}
    \int_{\partial \Sigma} \left( \frac{H_{\partial M}}{\sin \bar{\gamma}} -
    \frac{1}{\sin \bar{\gamma}}  \frac{\partial \bar{\gamma}}{\partial \eta}
    \right) \mathrm{d} \lambda \geq \int_{\partial \Sigma} \left(
    \frac{\bar{H}_{\partial M}}{\sin \bar{\gamma}} - \frac{1}{\sin
    \bar{\gamma}}  \frac{\partial \bar{\gamma}}{\partial \bar{\eta}} \right)
    \mathrm{d} \lambda . \label{flatten integrand}
  \end{equation}
  Now we work with the flat metric. Let $e_1$ be the unit tangent vector field
  of $\ell_s$, $e_2$ the unit normal vector field of $\ell_s$ pointing outside
  of the region bounded by $\ell_s$ in the plane $\{x^3 = s\}$ and $e_3$ be
  the third positive coordinate direction. Now we can calculate explicitly
  $\bar{X} = \cos \bar{\gamma} e_3 + \sin \bar{\gamma} e_2$, normal vector
  field of $\partial B$, $\bar{\eta} = - \sin \bar{\gamma} e_3 + \cos
  \bar{\gamma} e_2$. We use $D$ to denote the covariant derivative in
  $\mathbb{R}^3$. So
\begin{align}
\bar{H}_{\partial M} = & \langle D_{e_1} e_1, - \bar{X} \rangle + \langle
D_{\bar{\eta}}  \bar{\eta}, - \bar{X} \rangle \\
= & \sin \bar{\gamma} \langle D_{e_1} e_1, - e_2 \rangle + \cos
\bar{\gamma} \langle D_{e_1} e_1, - e_3 \rangle + \langle D_{\bar{\eta}}
\bar{\eta}, - \bar{X} \rangle \\
= & \sin \bar{\gamma} \langle D_{e_1} e_1, - e_2 \rangle + \langle
D_{\bar{\eta}}  \bar{\eta}, - \bar{X} \rangle,
\end{align}
  where $\langle D_{e_1} e_1, - e_3 \rangle$ vanishes because $\ell_s$ is
  planar. Meanwhile, by \eqref{angle def},
  
  \begin{align*}
    - \langle D_{\bar{\eta}}  \bar{\eta}, \bar{X} \rangle - \frac{\partial
    \bar{\gamma}}{\partial \bar{\eta}} = & \langle D_{\bar{\eta}}  \bar{X},
    \bar{\eta} \rangle - \frac{\partial \bar{\gamma}}{\partial \bar{\eta}}\\
    = & \langle D_{\bar{\eta}}  \bar{X}, \bar{\eta} \rangle + \frac{1}{\sin
    \bar{\gamma}}  \frac{\partial}{\partial \bar{\eta}} \langle \bar{X}, e_3
    \rangle\\
    = & \frac{1}{\sin \bar{\gamma}}  \langle D_{\bar{\eta}}  \bar{X}, - \sin^2
    \bar{\gamma} e_3 + \sin \bar{\gamma} \cos \bar{\gamma} e_2 + e_3
    \rangle\\
    = & \frac{1}{\sin \bar{\gamma}}  \langle D_{\bar{\eta}}  \bar{X}, \cos
    \bar{\gamma}  \bar{X} \rangle\\
    = & \cot \bar{\gamma} \langle D_{\bar{\eta}}  \bar{X}, \bar{X} \rangle\\
    = & 0.
  \end{align*}
  
  Hence, we have
  \begin{equation}
    \frac{\bar{H}_{\partial M}}{\sin \bar{\gamma}} - \frac{1}{\sin
    \bar{\gamma}}  \frac{\partial \bar{\gamma}}{\partial \bar{\eta}} = \langle
    D_{e_1} e_1, - e_2 \rangle . \label{euclidean pointwise relation}
  \end{equation}
  The right-hand side is the geodesic curvature of the circle $\ell_s$ on the
  plane $\{x^3 =s\}$ and so it is positive. Considering $\mathrm{d} \lambda
  \geq \mathrm{d} \bar{\lambda}$ in \eqref{flatten integrand}, we have
\begin{align}
& \int_{\partial \Sigma} \left( \frac{H_{\partial M}}{\sin \bar{\gamma}}
- \frac{1}{\sin \bar{\gamma}}  \frac{\partial \bar{\gamma}}{\partial \eta}
\right) \mathrm{d} \lambda \\
\geq & \int_{\partial \Sigma} \langle D_{e_1} e_1, - e_2
\rangle_{\delta} \mathrm{d} \lambda \\
\geq & \int_{\partial \Sigma} \langle D_{e_1} e_1, - e_2
\rangle_{\delta} \mathrm{d} \bar{\lambda} .
\end{align}
  We conclude our proof by invoking the following lemma.
\end{proof}

\begin{remark}\label{angle comparison}
  We have a different proof of \eqref{lower bound for separating curve} in
  Section \ref{generalizations} where we show the hyperbolic analog \eqref{hyperbolic decomposition}.
\end{remark}

\begin{lemma}
  \label{lem_curve_inte}For any rectifiable curve $\ell$ separating $p_{\pm}$
  ($P_{\pm}$ if we are in Case \ref{case slab} of Theorem \ref{main})
  \[ \int_{\ell} \langle D_{e_1} e_1, - e_2 \rangle_{\delta} \mathrm{d}
     \bar{\lambda} \geq 2 \pi . \label{lower bound of integral of
     geodesic curvature} \]
  and equality holds if and only if $\ell$ is at the same height of $\partial
  M$.
\end{lemma}

\begin{proof}
  It suffices to prove for such $\ell$ which is $C^1$. Since every height of
  $\partial M$ is a circle, so we parameterize $\partial M$ by
  \[ (r (s) \cos \theta, r (s) \sin \theta, s) \in \mathbb{R}^3, \text{ } s
     \in [- 1, 1], \theta \in \mathbb{S}^1 . \]
  Then $\langle D_{e_1} e_1, - e_2 \rangle = \tfrac{1}{r (s)}$ if the third
  coordinate is $s$. We assume that the curve $\ell$ is parameterized by $t
  \in [0, t_0]$ and we write $\ell$ as
  \[ (r (s (t)) \cos \theta (t), r (s (t)) \sin \theta (t), s (t)) . \]
  The tangent vector
  \[ \partial_t = \tfrac{\mathrm{d} s}{\mathrm{d} t} (r' \cos \theta, r' \sin
     \theta, 1) + \tfrac{\mathrm{d} \theta}{\mathrm{d} t} (- r \sin \theta, r
     \cos \theta, 0), \]
  with length
  \[ | \partial_t |^2 = | \tfrac{\mathrm{d} s}{\mathrm{d} t} |^2 (1 + (r')^2)
     + (\tfrac{\mathrm{d} \theta}{\mathrm{d} t})^2 r^2 . \]
  The line element is $\mathrm{d} \bar{\lambda} \geq | \tfrac{\mathrm{d}
  \theta}{\mathrm{d} t} | r$. Then
  \[ \int_{\ell} \langle D_{e_1} e_1, - e_2 \rangle_{\delta} \geq
     \int_0^{t_0} | \tfrac{\mathrm{d} \theta}{\mathrm{d} t} | \mathrm{d} t
     \geq \int_0^{t_0} \tfrac{\mathrm{d} \theta}{\mathrm{d} t} \mathrm{d}
     t = \int_{\mathbb{S}^1} \mathrm{d} \theta = 2 \pi . \]
  By tracing back the proof, we see that equality holds if and only if
  $\ell$ is at the same height of $\partial M$.
\end{proof}

\subsection{Infinitesimally rigid surface}

Since $\Sigma$ has at least one boundary component, so $\chi (\Sigma)
\leq 1$ and by Lemma \ref{lemma lower bound over separating curve},
\begin{align}
2 \pi & \geq 2 \pi \chi (\Sigma) \\
& \geq \int_{\partial \Sigma} \left( \frac{H_{\partial M}}{\sin
\bar{\gamma}} - \frac{1}{\sin \bar{\gamma}}  \frac{\partial
\bar{\gamma}}{\partial \eta} \right) \mathrm{d} \lambda \\
& \geq \int_{\partial \Sigma} \left( \frac{\bar{H}_{\partial M}}{\sin
\bar{\gamma}} - \frac{1}{\sin \bar{\gamma}}  \frac{\partial
\bar{\gamma}}{\partial \bar{\eta}} \right) \mathrm{d} \bar{\lambda}
\\
& \geq 2 \pi .
\end{align}
All inequalities are equalities, by tracing the proof we find that
\begin{equation}
  R_g = |A| = 0 \text{ in } \Sigma \label{infinitesimal interior}
\end{equation}
and
\begin{equation}
  H_{\partial M} = \bar{H}_{\partial M}, \text{ } \langle X, N \rangle = \cos
  \bar{\gamma}, \sigma = \bar{\sigma} \text{ along } \partial \Sigma
  \label{infinitesimal boundary}
\end{equation}
for some constant angle $\bar{\gamma} \in (0, \pi)$. Because $\Sigma$ is
stable, then the eigenvalue problem
\begin{equation}
  \left\{\begin{array}{ll}
    L f & = \mu f \text{ in } \Sigma\\
    B f & = 0 \text{ on } \partial \Sigma
  \end{array}\right. \label{eigen problem}
\end{equation}
has a nonnegative first eigenvalue $\mu_1 \geq 0$. The analysis now is
similar to {\cite{fischer-colbrie-structure-1980}}. Letting $f \equiv 1$ in
\eqref{stability}, and using both \eqref{sy}, \eqref{boundary sy},
\eqref{infinitesimal interior} and \eqref{infinitesimal boundary} we have
\begin{equation}
  Q (1, 1) = \int_{\Sigma} K + \int_{\partial \Sigma} \kappa - \int_{\partial
  \Sigma} (\tfrac{H_{\partial M}}{\sin \bar{\gamma}} - \tfrac{1}{\sin
  \bar{\gamma}} \tfrac{\partial \bar{\gamma}}{\partial \eta}) = 0.
  \label{constant in bilinear form}
\end{equation}
So $\mu_1 = 0$ and the constant 1 is its corresponding eigenfunction. By
\eqref{infinitesimal interior} and \eqref{sy}, the stability operator reduces
to $L = - \Delta + K$; by \eqref{boundary sy} and \eqref{infinitesimal
boundary}, the boundary stability operator reduces to $B = \partial_{\nu} -
(\tfrac{H_{\partial M}}{\sin \bar{\gamma}} - \tfrac{1}{\sin \bar{\gamma}}
\tfrac{\partial \bar{\gamma}}{\partial \eta} - \kappa)$. So
\begin{equation}
  K = 0 \text{ in } \Sigma, \text{ } \kappa = \tfrac{H_{\partial M}}{\sin
  \bar{\gamma}} - \tfrac{1}{\sin \bar{\gamma}} \tfrac{\partial
  \bar{\gamma}}{\partial \eta} . \label{vanishing gauss and constant geodesic}
\end{equation}
We call $\Sigma$ satisfying \eqref{infinitesimal interior},
\eqref{infinitesimal boundary} and \eqref{vanishing gauss and constant
geodesic} an \text{{\itshape{infinitesimally}}} \text{{\itshape{rigid}}}
\text{{\itshape{surface}}}.

\begin{remark}
  It is direct to see that
  \begin{equation}
    \kappa = \langle D_{e_1} e_1, - e_2 \rangle_{\delta}
  \end{equation}
  and $\Sigma$ has to be a flat disk of radius $\langle D_{e_1} e_1, - e_2
  \rangle_{\delta}^{- 1}$.
\end{remark}

\subsection{CMC capillary foliation}\label{construction of CMC foliation} We now construct a CMC foliation with prescribed contact angles near an infinitesimally rigid surface. See for instance the works {\cite{ye-foliation-1991}},
{\cite{bray-rigidity-2010}} and {\cite{ambrozio-rigidity-2015}} on
constructing CMC foliations.

Let $\phi (x, t)$ be a local
flow of a vector field $Y$ which is tangent to $\partial M$ and transverse to
$\Sigma$ and that $\langle Y, N \rangle = 1$. In the following theorem, we
only require that the scalar curvature of $(M, g)$ and the mean curvature of
$\partial M$ are bounded below.

\begin{theorem}
  \label{foliation near infinitesimal}Suppose $(M, g)$ is a three-manifold
  with boundary, if $\Sigma$ is an infinitesimally rigid surface, then there
  exists $\varepsilon > 0$ and a function $w (x, t)$ on $\Sigma \times (-
  \varepsilon, \varepsilon)$ such that for each $t \in (- \varepsilon,
  \varepsilon)$, the surface
  \begin{equation}
    \Sigma_t = \{\phi (x, w (x, t)) : x \in \Sigma\}
  \end{equation}
  is a constant mean curvature surface intersecting $\partial M$ with
  prescribed angle $\bar{\gamma}$. Moreover, for every $x \in \Sigma$ and
  every $t \in (- \varepsilon, \varepsilon)$,
  \begin{equation}
    w (x, 0) = 0, \text{ } \int_{\Sigma} (w (x, t) - t) = 0 \text{ and }
    \tfrac{\partial}{\partial t} w (x, t) |_{t = 0} = 1.
  \end{equation}
\end{theorem}

\begin{proof}
  Given a function in the H{\"o}lder space $C^{2, \alpha} (\Sigma)$ ($0 <
  \alpha < 1$), we consider
  \[ \Sigma_u = \{\phi (x, u (x)) : x \in \Sigma\}, \]
  which is a properly embedded surface if the norm of $u$ is small enough. We
  use the subscript $u$ to denote the quantities associated with $\Sigma_u$.
  
  Consider the space
  \[ \mathcal{Y}= \{u \in C^{2, \alpha} (\Sigma) \cap C^{1,\alpha}(\bar{\Sigma}): \int_{\Sigma} u = 0\} \]
  and
  \[ \mathcal{Z}= \{u \in C^{0, \alpha} (\Sigma) : \int_{\Sigma} u = 0\} . \]
  Given small $\delta > 0$ and $\varepsilon > 0$, we define the map
  \[ \Phi : (- \varepsilon, \varepsilon) \times B (0, \delta) \to \mathcal{Z}
     \times C^{1, \alpha} (\partial \Sigma) \]
  given by
  \begin{equation}
    \Phi (t, u) = \left( H_{t + u} - \tfrac{1}{| \Sigma |} \int_{\Sigma} H_{t
    + u}, \langle X_{t + u}, N_{t + u} \rangle - \cos \bar{\gamma}_{t + u}
    \right) . \label{phi}
  \end{equation}
  Here $\bar{\gamma}_{t+u}$ is the function $\bar{\gamma}$ restricted on $\partial \Sigma_{t+u}$. Evidently, for a $(t,u)$ such that $\Phi(t,u)=(0,0)$ the surface $\Sigma_{t+u}$ is of constant mean curvature and meets $\partial M$ with contact angles $\bar{\gamma}$.
  For each $v \in \Sigma$, the map
  \[ f : (x, s) \in \Sigma \times (- \varepsilon, \varepsilon) \to \phi (x, s
     v (x)) \in M \]
  gives a variation with
  \[ \tfrac{\partial f}{\partial s} |_{s = 0} = \tfrac{\partial}{\partial s}
     \phi (x, s v (x)) |_{s = 0} = v N. \]
  Since $\Sigma$ is infinitesimally rigid, we obtain that using \eqref{angle
  first variation} and \eqref{background angle first variation},
  \[ D \Phi_{(0, 0)} (0, v) = \tfrac{\mathrm{d}}{\mathrm{d} s} \Phi (0, s v)
     |_{s = 0} = \left( - \Delta v + \tfrac{1}{| \Sigma |} \int_{\partial
     \Sigma} \Delta v, - \sin \bar{\gamma} \tfrac{\partial v}{\partial \nu}
     \right) . \]
  It follows from the elliptic theory for the Laplace operator with Neumann
  type boundary conditions that $D \Phi (0, 0)$ is an isomorphism when
  restricted to $0 \times \mathcal{Y}$.
  
  Now we apply the implicit function theorem: For some smaller $\varepsilon$,
  there exists a function $u (t) \in B (0, \delta) \subset \mathcal{X}$, $t
  \in (- \varepsilon, \varepsilon)$ such that $u (0) = 0$ and $\Phi (t, u (t))
  = \Phi (0, 0) = (0, 0)$ for every $t$. In other words, the surfaces
  \[ \Sigma_{t + u (t)} = \{\phi (x, t + u (t)) : x \in \Sigma\} \]
  are constant mean curvature surfaces with prescribed angles $\bar{\gamma}$.
  
  Let $w (x, t) = t + u (t) (x)$ where $(x, t) \in \Sigma \times (-
  \varepsilon, \varepsilon)$. By definition, $w (x, 0) = 0$ for every $x \in
  \Sigma$ and $w (\cdot, t) - t = u (t) \in B (0, \delta) \subset \mathcal{X}$
  for every $t \in (- \varepsilon, \varepsilon)$. Observe that the map $s
  \mapsto \phi (x, w (x, s))$ gives a variation of $\Sigma$ with variational
  vector field given by
  \[ \tfrac{\partial \phi}{\partial t} \tfrac{\partial w}{\partial s} |_{s =
     0} = \tfrac{\partial w}{\partial s} |_{s = 0} Y. \]
  Since for every $t$ we have that
  \[ 0 = \Phi (t, u (t)) = \left( H_{w (\cdot, t)} - \tfrac{1}{| \Sigma |}
     \int_{\Sigma} H_{w (\cdot, t)}, \langle X_{t + u}, N_{t + u} \rangle -
     \cos \bar{\gamma}_{t + u} \right), \]
  by taking the derivative at $t = 0$ we conclude that
  \[ \langle \tfrac{\partial w}{\partial t} |_{t = 0} Y, N \rangle =
     \tfrac{\partial w}{\partial t} |_{t = 0} \]
  satisfies the homogeneous Neumann problem. Therefore, it is constant on
  $\Sigma$. Since
  \[ \int_{\Sigma} (w (x, t) - t) = \int_{\Sigma} u (x, t) = 0 \]
  for every $t$, by taking derivatives at $t = 0$ again, we conclude that
  \[ \int_{\Sigma} \tfrac{\partial w}{\partial t} |_{t = 0} = | \Sigma | . \]
  Hence, $\tfrac{\partial w}{\partial t} |_{t = 0} = 1$. Taking $\varepsilon$
  small, we see that $\phi (x, w (x, t))$ parameterize a foliation near
  $\Sigma$.
\end{proof}

\begin{theorem}
  \label{H ode}There exists a continuous function $\Psi (t)$ such that
  \[ \tfrac{\mathrm{d}}{\mathrm{d} t} \left( \exp \left(- \int_0^{t} \Psi
     (\tau) \mathrm{d} \tau\right) H \right) \leq 0. \]
\end{theorem}

\begin{proof}
  Let $\psi : \Sigma \times I \to M$ parameterize the foliation, $Y =
  \tfrac{\partial \psi}{\partial t}$, $v_{t} = \langle Y, N_{t}
  \rangle$. Then
  \begin{equation}
    - \tfrac{\mathrm{d}}{\mathrm{d} t} H (t) = \Delta_{t} v_{t} +
    (\ensuremath{\operatorname{Ric}}(N_{t}) + |A_{t} |^2) v_{t}
    \text{ in } \Sigma_{t}, \label{s variation}
  \end{equation}
  and
  \[ \tfrac{\partial v_{t}}{\partial t} = [- \cot \bar{\gamma}  A_{t}
     (\nu_{t}, \nu_{t}) + \tfrac{1}{\sin \bar{\gamma}} A_{\partial M}
     (\eta_{t}, \eta_{t}) + \tfrac{1}{\sin^2 \bar{\gamma}}
     \nabla_{\eta_{t}} \cos \bar{\gamma}] v_{t} . \]
  By shrinking the interval if needed, we assume that $v_{t} > 0$ for
  $t \in I$. By multiplying of \eqref{s variation} and integrating on
  $\Sigma_{t}$, we deduce that
\begin{align}
& - H' (s) \int_{\Sigma_{t}} \tfrac{1}{v_{t}} \\
= & \int_{\Sigma_{t}} \tfrac{\Delta_{t} v_{t}}{v_{t}} +
(\ensuremath{\operatorname{Ric}}(N_{t}) + |A_{t} |^2) \\
= & \int_{\partial \Sigma_{t}} \tfrac{1}{v_{t}} \tfrac{\partial
v_{t}}{\partial \nu_{t}} + \tfrac{1}{2} \int_{\Sigma_{t}} (R_g +
|A_{t} |^2 + H (t)^2) - \int_{\Sigma_{t}} K_{\Sigma_{t}}
\\
\geq & \int_{\partial \Sigma_{t}} \left[- \cot \bar{\gamma}  A_{t}
(\nu_{t}, \nu_{t}) + \tfrac{1}{\sin \bar{\gamma}} A_{\partial M}
(\eta_{t}, \eta_{t}) + \tfrac{1}{\sin^2 \bar{\gamma}}
\nabla_{\eta_{t}} \cos \bar{\gamma}\right] - \int_{\Sigma_{t}}
K_{\Sigma_{t}} \\
\geq & - \left[ \int_{\partial \Sigma_{t}} \kappa_{\partial
\Sigma_{t}} + \int_{\Sigma_{t}} K_{\Sigma_{t}} \right] -
\int_{\partial \Sigma_{t}} H (t) \cot \bar{\gamma} + \int_{\partial
\Sigma_{t}} \tfrac{H_{\partial M}}{\sin \bar{\gamma}} +
\tfrac{1}{\sin^2 \bar{\gamma}} \nabla_{\eta_{t}} \cos \bar{\gamma}
\end{align}
  where in the last line we have used the following version of \eqref{boundary
  sy}
  \[ \kappa_{\partial \Sigma_{t}} - \cot \bar{\gamma} A (\nu_{t},
     \nu_{t}) + \tfrac{1}{\sin \bar{\gamma}} A_{\partial M} (\eta_{t},
     \eta_{t}) = - H (t) \cot \bar{\gamma} + \tfrac{1}{\sin
     \bar{\gamma}} H_{\partial M} . \]
  By the Gauss-Bonnet theorem and by Lemma \ref{lemma lower bound over
  separating curve}, we have
  \[ - H' (t) \int_{\Sigma_{t}} \tfrac{1}{v_{t}} \geq - H (t)
     \int_{\partial \Sigma_{t}} \cot \bar{\gamma} . \]
  Let
  \[ \Psi (t) = \left( \int_{\Sigma_{t}} \tfrac{1}{v_{t}} \right)^{-
     1} \int_{\partial \Sigma_{t}} \cot \bar{\gamma}, \]
  then note that we have assume that $v_{t} > 0$ near $t = 0$, so $H
  (t)$ satisfies the ordinary differential inequality
  \begin{equation}
    H' - \Psi (t) H \leq 0. \label{original ode}
  \end{equation}
  We see then
  \[ \tfrac{\mathrm{d}}{\mathrm{d} t} \left( \exp \left( - \int_0^{t}
     \Psi (\tau) \mathrm{d} \tau \right) H \right) \leq 0. \]
\end{proof}
\begin{remark}
  It is only required that the mean curvature of each $\Sigma_{t}$ is a constant and $\gamma=\bar{\gamma}$ along $\partial\Sigma_{t}$. This
  theorem would be used also in later sections.
  \end{remark}
\subsection{From local foliation to rigidity}

Let $\Sigma_{t}$ be the constant mean curvature surfaces forming prescribed
contact angles $\bar{\gamma}$ with $\partial M$.

\begin{proposition}
  Every $\Sigma_{t}$ constructed in Theorem \ref{foliation near
  infinitesimal} is infinitesimally rigid.
\end{proposition}

\begin{proof}
  We abuse the notation and let
  \[ F (t) = | \Sigma_{t} | - \int_{\partial \Omega_{t}} \cos
     \bar{\gamma} . \]
  By the first variation formula \eqref{first variation},
  \[ F (t_2) - F (t_1) = \int_{t_1}^{t_2} \mathrm{d} t
     \int_{\Sigma_{t}} H (t) v_{t} . \]
  By Theorem \ref{H ode},
  \[ H (t) \leq 0 \text{ if } t \geq 0 ; \text{ } H (t)
     \geq 0 \text{ if } t \leq 0, \]
  which in turn implies that
  \[ F (t) \leq 0 \text{ if } t \geq 0 ; \text{ } F (t)
     \leq 0 \text{ if } t \leq 0. \]
  However, $\Omega$ is a minimiser to the functional \eqref{action}, hence
  \[ F (t) \equiv F (0) . \]
  It then follows every $\Sigma_{t}$ is a minimiser, hence infinitesimally
  rigid.
\end{proof}

Now we finish the proof of Case \ref{case slab} of Theorem \ref{main} by using
Theorem \ref{H ode}.

\begin{proof}[Proof of Case \ref{case slab} of Theorem \ref{main}]
  Let $Y = \tfrac{\mathrm{d}}{\mathrm{d} t} \phi (x, w (x, t))$ where $\phi$
  and $w$ are as Theorem \ref{foliation near infinitesimal}, we show first
  that $N_{t}$ is a parallel vector field. Since every $\Sigma_{t}$ is
  infinitesimally rigid, from \eqref{eigen problem} and \eqref{constant in
  bilinear form}, we know that $\langle Y, N_{t} \rangle$ is a constant.
  Let $\partial_i$, $i = 1, 2$ are vector fields induced by local coordinates
  on $\Sigma$ and extended to a neighbtod of $\Sigma$ by $Y$, then $\nabla_{\partial_i} \langle Y, N \rangle = 0$. Note that
  $\Sigma_{t}$ are totally geodesic, so $\nabla_{\partial_i} N \equiv 0$
  and
\begin{align}
& 0 = \nabla_{\partial_i} \langle Y, N \rangle = \langle
\nabla_{\partial_i} Y, N \rangle + \langle Y, \nabla_{\partial_i} N
\rangle = \langle \nabla_{\partial_i} Y, N \rangle,
\end{align}
  and
  \[ 0 = \langle \nabla_{\partial_i} Y, N \rangle = \langle \nabla_Y
     \partial_i, N \rangle = \langle \nabla_{Y^{\bot}} \partial_i, N \rangle .
  \]
  So
  \[ 0 = \langle \nabla_{Y^{\bot}} \partial_i, N \rangle = Y^{\bot} \langle
     \partial_i, N \rangle - \langle \partial_i, \nabla_{Y^{\bot}} N \rangle =
     - \langle \partial_i, \nabla_{Y^{\bot}} N \rangle . \]
  We conclude that $N$ is a parallel vector field and since every
  $\Sigma_{t}$ is flat, then $\cup_{t} \Sigma_{t}$ foliates a subset
  of the Euclidean space $\mathbb{R}^3$ with the flat metric. So
  $\Sigma_{t}$ is a family of parallel disks in $\mathbb{R}^3$ and we can
  parameterize $\partial M \cap (\cup_{t} \partial \Sigma_{t})$ locally
  by
  \[ \vec{x} (t, \theta) = (\psi (t) \cos \theta + a_1 (t), \psi
     (t) \sin \theta + a_2 (t), - t) . \]
  By translation invariance, we can set $a_1 (0) = a_2 (0) = 0$. Then the
  tangent vectors are
  \[ \vec{x}_{t} = (\psi' \cos \theta + a_1', \psi' \sin \theta + a_2', -
     1), \vec{x}_{\theta} = (- \psi \sin \theta, \psi \cos \theta) . \]
  It is easy to see that $\tilde{X} = (\cos \theta, \sin \theta, \psi' + a_1'
  \cos \theta + a_2' \sin \theta)$ is normal to $\partial M \cap (\cup_{t}
  \partial \Sigma_{t})$. Since $\tilde{X}$ forms an angle with $N = (0, 0,
  1)$ independent of $\theta$ by \eqref{infinitesimal boundary}, then
  \[ \tfrac{\partial}{\partial \theta} \tfrac{\langle \tilde{X}, N
     \rangle_{\delta}}{| \tilde{X} |_{\delta}} = 0. \]
  By an easy calculation,
  \[ - a_1' \sin \theta + a_2' \cos \theta = 0 \]
  which requires that $a_1' = a_2' = 0$ and so $a_1 = a_2 = 0$. So $\partial M
  \cap (\cup_{t} \partial \Sigma_{t})$ is rotationally symmetric. Since
  $M$ is connected, we conclude that $M$ is a rotationally symmetric set in
  $\mathbb{R}^3$ with the flat metric.
\end{proof}

\section{Conical pole}\label{conical case}

In this section, we deal with the conical case of Theorem \ref{main}. 
In Subsection \ref{sub:foliation} and Subsection \ref{sub:MeanConial}, we show that if the metric at the vertex is isometric to the model,
then we can construct a non-negative, constant mean curvature foliation with prescribed contact angles near
the vertex.
In Subsection \ref{sub:non_euclidean}, we directly construct a surface with positive mean curvature and larger contact angles than the model when the metric is not equal to $\delta$ at the vertex.
In both cases, we have a barrier for the minimal capillary surface, which again reduces the problem to the previous section.

\subsection{Foliation near the conical point}
\label{sub:foliation}

We assume that $g = \delta$ at $p_+$ . We consider the case when $p_+$ is conical and 
construct CMC foliations near the conical point $p_+$.

Let $\Sigma_{t}$
be the disk given by
\begin{equation}
  \Sigma_{t} = (\psi (t) \hat{x},  - t) \quad\text{where }\psi(0)=0, \psi'(0)>0.
\end{equation}
We also understand $\Sigma_t$ as a map from $D$ to $\Sigma_t$, that is, $\Sigma_t(\hat{x}):=(\psi(t)\hat{x},-t)$ for any $\hat{x} \in D$.
Since $\partial M$ at $p_+$ is conical, so $\psi (t) = \psi' (0) t + O
(t^2)$. Let $D$ be the unit disc, assume that $| \hat{u} (\cdot, t)
|_{C^{2, \alpha} (D)} = O (t^2)$, define
\begin{equation}
  \Sigma_{t, \hat{u}} = (\psi (t + \hat{u}) \hat{x},  - t - \hat{u}),
\end{equation}
where $\hat{x} = (x_1, x_2) \in D$, then
\begin{equation}
  \Sigma_{t, \hat{u}} = \Sigma_{t} + (\psi' (t) x_1, \psi'
  (t) x_2, - 1)  \hat{u}+ O (\hat{u}^2)
\end{equation}
by the Taylor expansion. The normal to $\Sigma_{t}$ is
\begin{equation}
  N_{t} = g^{3 j} e_j / \sqrt{g^{33}},
\end{equation}
where $\{e_i = \frac{\partial}{\partial x_i} \}_{i = 1,2,3}$ is the standard orthonormal basis of $\mathbb{R}^3$ and $g^{33} = g^{33} (\psi x_1, \psi x_{2}, - t)$. 
Note that $\Sigma_{t, \hat{u}}$ is approximately a normal graph of a function over
$\Sigma_{t}$, it is given by
\begin{align}
& U \\
= & \hat{u} \langle (\psi' x_1, \psi' x_2, - 1), g^{3 j} e_j / \sqrt{g^{33}}
\rangle + O (\hat{u}^2) \\
= & - \tfrac{\hat{u}}{\sqrt{g^{3 3}}} + O (\hat{u}^2) \\
= & - \tfrac{\hat{u}}{\sqrt{g^{33}}} + O (t^3) \\
= & - \hat{u} + O (t^3) 
\end{align}
using the Taylor expansion of $\Sigma_{t,\hat{u}}$.


We use the subscript $t,\hat{u}$ on every geometric quantity defined on the surface
$\Sigma_{t,\hat{u}}$, for instance, $X_{t,\hat{u}}$ is the unit outward normal of $\partial M$ in $M$ along $\partial\Sigma_{t,\hat{u}}$. This is similarly applied to the use of the subscript $t$. However, $\bar{\gamma}_{t,\hat{u}}$ represents the value of the prescribed angle $\bar{\gamma}$ at $\partial \Sigma_{t,\hat{u}}$.

It is also easy to see that
\begin{equation}
  \Delta_{t}^D U = - \Delta_{t}^D \hat{u} + O (t^3) .
\end{equation}
Here, $\Delta_t^D$ is the Laplace-Beltrami operator under metric $\frac{1}{\psi^2(t)}\Sigma_t^*(g)$ on $D$. Note that $\Delta_t^D$ converges to the standard Laplace-Beltrami operator $\Delta$ on $D$ under Euclidean metric as $t \to 0$.
We also write $\nu_t^D$ the unit outer normal vector field of $\partial D$ under same metric $\frac{1}{\psi^2(t)}\Sigma_t^*(g)$ and $\nu_t^D$ converges to the unit outer normal vector field $\nu$ along $\partial D$ under Euclidean metric as $t\to 0$.

Consider the space
\[ \mathcal{Y}= \left\{ u \in C^{2, \alpha} (D)\cap C^{1,\alpha}(\bar{D}) : \int_D u = 0 \right\} \]
and
\begin{equation}
  \mathcal{Z}= \left\{ u \in C^{0, \alpha} (D) : \int_D u = 0 \right\} ,
\end{equation}
given small $\delta > 0$ and $\varepsilon > 0$, we define the map
\begin{equation}
 \Psi : (- \varepsilon, \varepsilon) \times B (0, \delta) \to \mathcal{Z}
  \times C^{1, \alpha} (\partial \Sigma)
\end{equation}
by
\begin{equation}
  \Psi (t, u) = \left( H_{t, t^2 u} - \frac{1}{|D|} \int_D H_{t,
  t^2 u}, t^{- 1} (\langle X_{t, t^2 u}, N_{t, t^2 u}
  \rangle - \cos \bar{\gamma}_{t, t^2 u}) \right) \label{Psi}
\end{equation}
for $t \neq 0$. Here the integration on $D$ is calculated with respect to
the flat metric. We extend $\Psi (t, u)$ to $t = 0$ by taking limits, that is,
\begin{equation}
  \Psi (0, u) = \lim_{t \to 0} \Psi (t, u) .
\end{equation}

Using the variational formulas of the mean curvature and the angles, we obtain
\begin{equation}
  H_{t, t^2 u} - H_{t} = \tfrac{t^2}{\psi (t)^2} \Delta_{t}^D
  u + t^2 (\ensuremath{\operatorname{Ric}}(N_{t}) + |A_{t} |^2) u + O
  (t^3), \label{mean curvature taylor expansion}
\end{equation}
and
\begin{align}
& t^{- 1} [\langle X_{t, t^2 u}, N_{t, t^2 u} \rangle -
\langle X_{t}, N_{t} \rangle] \\
= & \tfrac{t \sin \gamma}{\psi (t)} \tfrac{\partial u}{\partial
\nu_{t}^D} - t (- \cos \gamma A (\nu_{t}, \nu_{t}) + A_{\partial
M} (\eta_{t} {,} \eta_{t})) u + O (t^2), \label{angle variation}
\end{align}
and if we choose $g=\delta$ in \eqref{angle variation}, we have (see, e.g. \eqref{std angle difference})
\begin{equation}
  \cos \bar{\gamma}_{t, t^2 u} - \cos \bar{\gamma}_{t}=-t^2 \bar{A}_{\partial M}(\bar{\eta}_t,\bar{\eta}_t)u+O(t^3)= 
  \tfrac{t^2 u}{\sin \bar{\gamma}} \partial_{\eta} \cos \bar{\gamma} + O
  (t^3) . \label{background angle difference near pole}
\end{equation}
Since
\begin{align}
  & \cos \gamma_{t, t^2 u} - \cos \bar{\gamma}_{t, t^2 u} \\
  = & (\cos \gamma_{t, t^2 u} - \cos \gamma_{t}) + (\cos \gamma_{t} -
  \cos \bar{\gamma}_{t}) + (\cos \bar{\gamma}_{t} - \cos
  \bar{\gamma}_{t, t^2 u}),
\end{align}
and from \eqref{background angle difference near pole} and \eqref{angle variation},
we have
\begin{align}
& t^{-1} [\cos \gamma_{t,t^2u}-\cos \bar{\gamma}_{t,t^2u}] \\
= & \tfrac{t \sin \gamma}{\psi (t)} \tfrac{\partial u}{\partial
\nu_{t}^D} + t (\cos \gamma_t A(\nu_t,\nu_t)- A_{\partial
    M} (\eta_{t} {,} \eta_{t})+\bar{A}_{\partial M}(\bar{\eta}_t,\bar{\eta}_t)) u \\
    &+t^{- 1} (\cos \gamma_{t} - \cos \bar{\gamma}_{t})+ O (t^2), \label{angle variation2}
\end{align}
and
\begin{align}
  & t^{- 1}  (\cos \gamma_{t, t^2 u} - \cos \bar{\gamma}_{t, t^2 u}) \\
  = & \tfrac{t \sin \gamma}{\psi (t)} \tfrac{\partial u}{\partial
  \nu_{t}^D} + t (\cos \gamma_{t} A (\nu_{t}, \nu_{t}) -
  A_{\partial M} (\eta_{t}, \eta_{t}) + \tfrac{1}{\sin \bar{\gamma}}
  \partial_{\eta} \cos \bar{\gamma}) u \\
  & + t^{- 1} (\cos \gamma_{t} - \cos \bar{\gamma}_{t}) + O (t^2). \label{angle variation 3} 
\end{align}
\begin{remark}
The term $A_{\partial M}(\eta_t,\eta_t)=O(1)$, although $|A_{\partial M}|_\delta=O(1/t)$. This can be seen via observing that $|\eta_t-\bar{\eta}_t|_\delta=O(t)$ and \cite[Lemma 2.2]{miao-mass-2021}.
Note that we also have $\bar{A}_{\partial M}(\bar{\eta}_t,\bar{\eta}_t)=O(1)$.
\end{remark}

\begin{proposition}
  \label{prop cmc capillary at zero}
  For each $t \in [0,\varepsilon)$ with $\varepsilon$ small enough, we can find $u_t=u(\cdot,t) \in C^{2, \alpha} (D)\cap C^{1,\alpha}(\overline{D})$
  such that $\int_D u(\cdot,t) = 0$ and
  \begin{equation}
    \Psi (t, u(\cdot,t)) = (0, 0) . \label{cmc capillary at zero}
  \end{equation}
  In particular, the surfaces $\Sigma_{t,t^2u_t}$ have constant mean curvature and prescribed angles $\gamma_{t,t^2u_t}=\bar{\gamma}_{t,t^2u_t}$.
\end{proposition}

\begin{proof}

   First, we compute $\Psi(0,u)$ and solve $\Psi(0,u)=0$.
   We note
  \begin{equation}
    \lim_{t \to 0} H_{t, t^2 u} = \lim_{t \to 0} H_{t} +
    \tfrac{1}{\psi' (0)^2} \Delta u
  \end{equation}
  from \eqref{mean curvature taylor expansion}. Since $\lim_{t \to 0}
  H_{t}$ is a constant, we obtain that the first component $\Psi_1 (t,
  u)$ of $\Psi (t, u)$ at $(0, u)$ is
  \begin{equation}
    \Psi_1 (0, u) = \tfrac{1}{\psi' (0)^2} \left( \Delta u - \tfrac{1}{|D|}
    \int_D \Delta u \right) .
  \end{equation}
By \eqref{angle variation 3}, the second component $\Psi_2 (t, u)$ of $\Psi (t, u)$ at $(0, u)$ is
  \begin{equation}
    \Psi_2 (0, u) = \tfrac{\sin \bar{\gamma}_0}{\psi' (0)} \tfrac{\partial
    u}{\partial \nu} + \lim_{t \to 0} \frac{\cos \gamma_{t} - \cos
    \bar{\gamma}_{t}}{t} .
  \end{equation}
  Here, $\sin \bar{\gamma}_0:=\lim_{t\to 0}\sin \bar{\gamma}_t$ and indeed, we can find $\sin \bar{\gamma}_0=1/\sqrt{1+(\psi')^2(0)}$.
  Note that $\lim_{t \to 0} \tfrac{\cos \gamma - \cos \bar{\gamma}}{t}$
  exists for every point of $\partial D$ and is a function on $\partial D$.
  Then $\Psi (0, u) = (0, 0)$ is equivalent to the elliptic boundary value problem
\begin{align}
\Delta u & = \tfrac{1}{|D|} \int_D \Delta u \text{ in } D \\
\tfrac{\partial u}{\partial \nu} & = - \tfrac{\psi' (0)}{\sin
\bar{\gamma}_0} \lim_{t \to 0} \tfrac{\cos \gamma - \cos
\bar{\gamma}}{t} \text{ on } \partial D.
\end{align}
By minimizing the energy
\[
	\mathcal{E}(u):=\int_{ D}  \frac{\left|\nabla u\right|^2}{2} +\int_{ \partial  D} \left( \frac{\psi'(0)}{\sin \bar{\gamma}_0}\lim_{t\rightarrow 0} \frac{\cos \gamma_t -\cos \bar{\gamma}_t}{t} \right) u
\]
on the function space $\{u\in W^{1,2}(D):\int_{ D} u=0\}$, we can find a minimiser $u_0 \in W^{1,2}(D)$ to $\mathcal{E}$. We have that $u_0$ is unique and in $C^{2,\alpha}(D)\cap C^{1,\alpha}(\bar{D})$ by elliptic regularity.
Now we compute
\begin{equation}
  D \Psi_{(0, u_0)} (0, v) = \tfrac{\mathrm{d}}{\mathrm{d} s} |_{s = 0} \Psi
  (0, u_0 + s v) .
\end{equation}
In particular,
\begin{align}
D \Psi |_{(0, u_0)} (0, v)
= & \tfrac{\mathrm{d}}{\mathrm{d} s} |_{s = 0} \Psi (0, u_0 + s v)
\\
= & \left( \tfrac{1}{\psi' (0)^2} \left[ - \Delta v + \tfrac{1}{|D|} \int_D
\Delta v \right], \tfrac{\sin \bar{\gamma}_0}{\psi' (0)} \tfrac{\partial
v}{\partial \nu} \right)
\end{align}
because $u_0$ satisfies \eqref{cmc capillary at zero}. It follows from the
elliptic theory for the Laplace operator with Neumann-type boundary conditions
that $D \Phi |_{(0, u_0)}$ is an isomorphism when restricted to $0 \times
\mathcal{Y}$.

Now we apply the implicit function theorem. For some smaller $\varepsilon>0$,
there exists a function $u (\cdot, t) \in B (0, \delta) \subset
\mathcal{X}$, $t \in (- \varepsilon, \varepsilon)$ such that $u (\cdot, 0)
= u_0$ and
\[ \Phi (t, u (\cdot, t)) = \Phi (0, u_0) = (0, 0) \]
for every $t$. In other words, the surfaces $\Sigma_{t, t^2 u}$ are
constant mean curvature surfaces with prescribed angles $\gamma =
\bar{\gamma}$ at $\partial \Sigma$.
\end{proof}

\subsection{Behavior of the mean curvature of the leaf}

\label{sub:MeanConial}

Let $u = u (\hat{x}, t)$ be as constructed as above, let $\lambda_{t} =
H_{t, t^2 u}$, then $\lambda_{t}$ is a constant. Similar to
{\cite{li-polyhedron-2020}} and {\cite{chai-dihedral-2022-arxiv}}, we obtain
the limiting behavior of $\lambda_{t}$ as $t \to 0$.

\begin{lemma}
	\label{lem:Lambdt}
  For the foliation constructed in the above,
  \begin{equation}
    \lambda_{t} |D_{t} | = \int_{D_{t}} H_{t} + \int_{\partial
    D_{t}} \tfrac{1}{\sin \gamma} (\cos \bar{\gamma}_{t} - \cos
    \gamma_{t}) + O (t^3) . \label{asymptotic mean curvature}
  \end{equation}
\end{lemma}

\begin{proof}
  We integrate \eqref{mean
  curvature taylor expansion} with respect to the metric $\sigma^t:=\frac{1}{\psi^2(t)}\Sigma_t^*(g)$ (we
  omit the area element and line element) on $D$ and using the divergence theorem and
  applying \eqref{angle variation},
\begin{align}
& \lambda_{t} |D|_{\sigma^{t}} - \int_D H_{t} \\
= & \tfrac{t^2}{\psi^2 (t)} \int_D \Delta_{t}^D u + t^2 \int_D
(\ensuremath{\operatorname{Ric}}(N_{t}) + |A_{t} |^2) u + O (t^3)
\\
= & \tfrac{t^2}{\psi^2 (t)} \int_{\partial D} \tfrac{\partial
u}{\partial \nu_{t}^D} + O (t^2) \\
= & \tfrac{t}{\psi (t)} \left[ \int_{\partial D} \tfrac{t}{\sin
\gamma_t} (- \cos \gamma_t A (\nu_{t}, \nu_{t}) + A_{\partial M}
(\eta_{t}, \eta_{t})) u + O (t^2) \right] \\
& + \tfrac{t}{\psi (t)} \int_{\partial D} \frac{\langle X_{t,
t^2 u}, N_{t, t^2 u} \rangle - \langle X_{t}, N_{t}
\rangle}{t \sin \gamma_t} + O (t^2) \\
= & \tfrac{t}{\psi (t)} \int_{\partial D} \frac{\cos \gamma_{t,
t^2 u} - \cos \gamma_{t}}{t \sin \gamma_t} + O (t) .
\end{align}
  It follows from \eqref{background angle difference near pole} that
\begin{align}
& \cos \gamma_{t, t^2 u} - \cos \gamma_{t} \\
= & (\cos \gamma_{t, t^2 u} - \cos \bar{\gamma}_{t, t^2 u}) +
(\cos \bar{\gamma}_{t, t^2 u} - \cos \bar{\gamma}_{t}) + (\cos
\bar{\gamma}_{t} - \cos \gamma_{t}) \\
= & 0 + (\tfrac{t^2 u}{\sin^2 \bar{\gamma}_t} \nabla_{\eta} \cos
\bar{\gamma} + O (t^3)) + (\cos \bar{\gamma}_{t} - \cos
\gamma_{t}) .
\end{align}
  Hence
  \[ \lambda_{t} |D|_{\sigma^{t}} - \int_D H_{t} = \tfrac{1}{\psi
     (t)} \int_{\partial D} \frac{\cos \bar{\gamma}_{t} - \cos
     \gamma_{t}}{\sin \gamma_t} + O (t), \]
  and a rescaling proves the lemma.
\end{proof}

Now we recall a result {\cite[Proposition 2.1]{miao-mass-2021}} interpreted
slightly differently.

\begin{proposition}
	\label{prop_MeanCur}
  Let $M$ be a manifold with two metrics $g$ and $\bar{g}$. \ We use bar to
  denote geometric quantities computed with respect to the metric $\bar{g}$.
  Let $S$ be a hypersurface in $M$, $X$ be a chosen unit normal and at a point
  $p \in S$. Let $H$ and $A$ denote respectively the mean curvature and the
  second fundamental form of $S$ in $M$ and $\sigma$ be the induced metric on $S$. If at some point $p \in S$, the two
  metrics agree $g (p) = \bar{g} (p)$, then near $p$,
\begin{align}
2 (H - \bar{H}) = & (\mathrm{d} \ensuremath{\operatorname{tr}}_{\bar{g}} h
-\ensuremath{\operatorname{div}}_{\bar{g}} h) (\bar{X})
-\ensuremath{\operatorname{div}}_{\bar{\sigma}} W - \langle h, \bar{A}
\rangle_{\bar{\sigma}} \\
& + | \bar{A} |_{\bar{g}} O (|h|^2_{\bar{g}}) + O (| \bar{\nabla}^M
h|_{\bar{g}} |h|_{\bar{g}}), \label{miao piubello}
\end{align}
  where $h = g - \bar{g}$ and $W$ is the dual vector field on $S$ of the
  1-form $h (\cdot, \bar{X})$.
\end{proposition}

\begin{remark}
  It suffices to follow the proof of {\cite[Proposition 2.1]{miao-mass-2021}}, we have to
  bear in mind that $|h|_{\bar{g}}$ is small but $| \bar{\nabla}^M h|$ is not small
  when getting closer to $p$ in our settings.
\end{remark}

Using the above lemma, we are able to show the following.

\begin{lemma}
  \label{limit of mean}For the foliation constructed in the above,
  \begin{equation}
    \lim_{t \to 0} \lambda_{t} \geq 0.
  \end{equation}
\end{lemma}

\begin{proof}
  We integrate both sides of \eqref{miao piubello} on $\Omega_{t}$ (note that here $\bar{g}=\delta$), we
  obtain
  \[ \int_{\partial \Omega_{t}} 2 (H - \bar{H})
     +\ensuremath{\operatorname{div}}_{\bar{\sigma}} W + \langle h, \bar{A}
     \rangle_{\bar{\sigma}} = \int_{\partial \Omega_{t}} (\mathrm{d}
     \ensuremath{\operatorname{tr}}_{\bar{g}} h
     -\ensuremath{\operatorname{div}}_{\bar{g}} h) (\bar{X}) + O (t^3) . \]
  The $O (t^3)$ remainder terms follow from that $|h|_{\bar{g}} = O
  (t)$, $| \nabla^M h|_{\bar{g}} = O (1)$ and $| \bar{A} |_{\bar{g}} = O
  (t^{- 1})$. Applying the divergence theorem on the right,
  \[ \int_{\partial \Omega_{t}} (\mathrm{d}
     \ensuremath{\operatorname{tr}}_{\bar{g}} h
     -\ensuremath{\operatorname{div}}_{\bar{g}} h) (\bar{X}) =
     \int_{\Omega_{t}} \bar{\nabla}^M_i (\bar{\nabla}_i^M
     \ensuremath{\operatorname{tr}}_{\bar{g}} h - \bar{\nabla}^M_j h_{i j}) =
     O (t^3) . \]
  Here $i, j$ denotes the indices of a local $\bar{g}$-orthonormal frame at
  the tangent space of $\Omega_{t}$. Hence
  \[ \int_{\partial \Omega_{t}} 2 (H - \bar{H})
     +\ensuremath{\operatorname{div}}_{\bar{\sigma}} W + \langle h, \bar{A}
     \rangle_{\bar{\sigma}} = O (t^3) . \]
  On $\Sigma_{t}$, $H = - H_{t}$ and $\bar{H} = 0$, $\bar{A} = 0$, we
  obtain that
  \[ 2 \int_{\Sigma_{t}} H_{t} - \int_{\partial \Omega_{t}}
     \ensuremath{\operatorname{div}}_{\bar{\sigma}} W = \int_{\partial
     \Omega_{t} \cap \partial M} 2 (H_{\partial M} - \bar{H}_{\partial M})
     + \langle h, \bar{A}_{\partial M} \rangle_{\bar{\sigma}} + O (t^3) .
  \]
  From the assumptions of Theorem \ref{main} we have
  \[ 2 \int_{\Sigma_{t}} H_{t} - \int_{\Sigma_{t} \cup (\partial
     \Omega_{t} \cap \partial M)}
     \ensuremath{\operatorname{div}}_{\bar{\sigma}} W \geq O (t^3) .
  \]
  It follows from the same lines of {\cite[(3.18)]{miao-mass-2021}} that
\begin{align}
& \int_{\Sigma_{t}} \ensuremath{\operatorname{div}}_{\bar{\sigma}} W +
\int_{\partial \Omega_{t} \cap \partial M}
\ensuremath{\operatorname{div}}_{\bar{\sigma}} W \\
= & \int_{\partial \Sigma_{t}} g (\bar{X}, \bar{\eta}) + \int_{\partial
\Sigma_{t}} g (- \bar{N}, \bar{\nu}) \\
= & - \int_{\partial \Sigma_{t}} \tfrac{2}{\sin \bar{\gamma}_{t}}
(\cos \bar{\gamma}_{t} - \cos \gamma_{t}),
\end{align}
  which in turn implies that
  \[ 2 \int_{\Sigma_{t}} H_{t} + 2 \int_{\partial \Sigma_{t}}
     \tfrac{1}{\sin \bar{\gamma}_{t}} (\cos \bar{\gamma}_{t} - \cos
     \gamma_{t}) \geq O (t^3) . \]
  From \eqref{asymptotic mean curvature}, we arrive
  \begin{equation}
    \lambda_{t} |D_{t} | \geq O (t^3) .
  \end{equation}
  Since $|D_{t} | = \pi \psi (t)^2 + O (t^3)$, so $\lambda_{t}
  \geq O (t)$. By taking limits, we have that $\lim_{t \to 0}
  \lambda_{t} \geq 0$.
\end{proof}

\begin{remark}
  This lemma and \eqref{asymptotic mean curvature} are quite natural in the
  sense that the limit of $\lambda_{t}$ can be regarded as an averaged mean
  curvature near the conical point.
\end{remark}

In fact, we have a stronger result which asserts that every leaf of the foliation has
nonnegative mean curvature.

\begin{lemma}
  \label{nonnegativity of mean curvature of every leaf}For the foliation
  constructed in the above, for each $t \in (0, \varepsilon)$,
  \begin{equation}
    \lambda_{t} \geq 0.
  \end{equation}
\end{lemma}

\begin{proof}
  We rename $\Sigma_{t, t^2 u}$ constructed earlier to $\Sigma_{t}$.
  We use the subscript $t$ on the geometric quantities of $\Sigma_{t}$.
  Note that the chosen unit normal $N_{t}$ points in the direction that
  $t$ decreases, so the ordinary differential inequality \eqref{original
  ode} established in Theorem \ref{H ode} is valid but with a reversed sign
  for the foliations near the conical point, that is,
  \begin{equation}
    \lambda' + \Psi (t) \lambda \geq 0, \text{ }t\geq 0 \label{H ode near conical
    point}
  \end{equation}
  where $\lambda (t) = \lambda_{t}$ and
  \[ \Psi (t) = \left( \int_{\Sigma_{t}} \tfrac{1}{v_{t}} \right)^{-
     1} \int_{\partial \Sigma_{t}} \cot \bar{\gamma}_t . \]
  where $v_t:=-\left<\frac{\mathrm{d}}{\mathrm{d}t}\Sigma_t,N_t\right>$.
  Since $\Sigma_{t}$ is constructed via a higher-order perturbation, we see
  that $v_{t} = 1 + O (t)$. Note that $\cot \bar{\gamma}_t=\psi'(0)+O(t)$, we have
  \[ \int_{\Sigma_{t}} \tfrac{1}{v_{t}} = \pi \psi (t)^2 + O
     (t^3), \text{ } \int_{\partial \Sigma_{t}} \cot \bar{\gamma}_t = 2
     \pi t (\psi' (0))^2 + O (t^2) . \]
  Therefore,
  \[ \Psi (t) = 2 t^{- 1} + C_1 (t) \]
  where $C_1 (t)$ is a continuous function of order $O (1)$. So it follows
  from \eqref{H ode near conical point} that $\lambda_{t}$ satisfies the
  ordinary differential inequality
  \begin{equation}
    \tfrac{\mathrm{d}}{\mathrm{d} t} \left[ \exp \left( \int^{t}_0
    C_1(s) \mathrm{d} s\right) t^2 \lambda (t) \right]
    \geq 0
  \end{equation}
  and the lemma now follows combining with Lemma \ref{limit of mean}.
\end{proof}

\subsection{Non-Euclidean metric at conical points}%
\label{sub:non_euclidean}

In this subsection, we assume $g\neq \delta$ at the conical point $O$.
Together with the assumption in Theorem \ref{main}, we actually assume, at the spherical point, we have
\begin{equation}
    |v|_g\ge |v|,\quad \forall 0\neq v\in T_O\partial M,
\end{equation}
and we can find at least one non-zero vector $v\in T_OM$ such that
$|v|_g\neq |v|$.
Note that it is possible that 
\begin{equation}
    |v|_g<|v| \quad\text{ for some non-zero }v \in T_OM.
    \label{eq:lessLengthAtConical}
\end{equation}

Before the construction of a barrier surface near $O$, we need a preliminary result regarding the comparison of cones in $\mathbb{R}^3$ with the standard Euclidean metric and some constant metric. 

For any $a_1,a_2>0$, we define a cone $C_{a_1a_2}\subset \mathbb{R}^3$ by
\[
	C_{a_1a_2}:=\left\{ (a_1t \hat{x}_1, a_2t \hat{x}_2, -t): t> 0, \hat{x} \in D \right\}.
\]
In short, we write $C_{a_1}:=C_{a_1a_2}$.

Given a plane $P$, we write
\[
	C_{a_1a_2}^P:=\text{ the bounded component of }C_{a_1a_2}\backslash P.
\]
Note that for some $P$, $C_{a_1a_2}^{P}$ may not be well-defined, we are only interested in the case of $P$ such that $C_{a_1a_2}^P$ is well-defined.

We write $\partial _F C^{P}_{a_1a_2}:=\partial C^P_{a_1a_2}\cap \partial C_{a_1a_2}$ as the side face of $C^P_{a_1a_2}$ and $\partial _B C^P_{a_1a_2}:=\partial C^{P}_{a_1a_2}\cap P$ the base face of $C^P_{a_1a_2}$.
Suppose $\mathbb{R}^3$ carries a metric $g$, we define the dihedral angle $\measuredangle _g^{C_{a_1a_2}^P}(p):=$the dihedral angle between $\partial _F C^P_{a_1a_2}$ and $\partial _B C^P_{a_1a_2}$ at point $p$ under metric $g$ for $p \in \partial _FC^P_{a_1a_2}\cap \partial _B C^P_{a_1a_2}$.

Now, let us fix $\bar{a}>0$ and write $\bar{C}=C_{\bar{a}}^{\bar{P}}$ where $\bar{P}=\left\{ x^3=-1 \right\}$, the model cone.
We have the following comparison result.
\begin{proposition}
	\label{prop_coneG0}
	Suppose $g_0$ is a constant metric on $\mathbb{R}^3$ and $\bar{C}$ satisfies the following comparisons: 
	\begin{enumerate}[(a)]
		\item $g_0|_{T_O\partial C_{\bar{a}}}\ge \delta|_{T_O\partial C_{\bar{a}}}$.
			\label{it:G0cone}
		\item $H^{g_0}_{\partial_F \bar{C}}(p)\ge H^\delta_{\partial_F \bar{C}}(p)$ for any $p \in \partial _F\bar{C}$.
			\label{it:G0mean}
	\end{enumerate}
	Then, we can find a plane $P$ such that
	\begin{equation}
		\measuredangle _{g_0}^{C_{\bar{a}}^P}(p)\ge \arctan\left(\frac{1}{\bar{a}}\right)=\measuredangle _\delta^{\bar{C}}(p),
		\label{eq:propAngleG0}
	\end{equation}
	for any $p \in \partial _F C_{\bar{a}}^P\cap \partial _B C^P_{\bar{a}}$. Moreover, if the inequality \eqref{eq:propAngleG0} is strict for every $p\in \partial _F C_{\bar{a}}^P\cap \partial _B C^P_{\bar{a}}$ except when $g_0=\delta$.
\end{proposition}

\begin{proof}
	[Proof of Proposition \ref{prop_coneG0}]
	We choose a linear map $L \in GL(3,\mathbb{R})$ such that $L_*g_0=\delta$ and assume $L(C_{\bar{a}})=C_{a_1a_2}$, $L(\bar{C})=C_{a_1a_2}^{P_0}$ for some plane $P_0$.
	For simplicity, we write $C=C_{a_1a_2}^{P_0}$.

	Let us prove the following claim at first.

	\noindent\textbf{Claim.} $\bar{a}\ge a_1$ and $\bar{a}\ge a_2$.

	We need to use condition \ref{it:G0cone} and \ref{it:G0mean} to prove this claim. See Figure \ref{fig:comparison-between-c1-and-c} for the comparison between two cones.

\begin{figure}[ht]
    \centering
	\begingroup
	\def\svgwidth{0.8\columnwidth}
	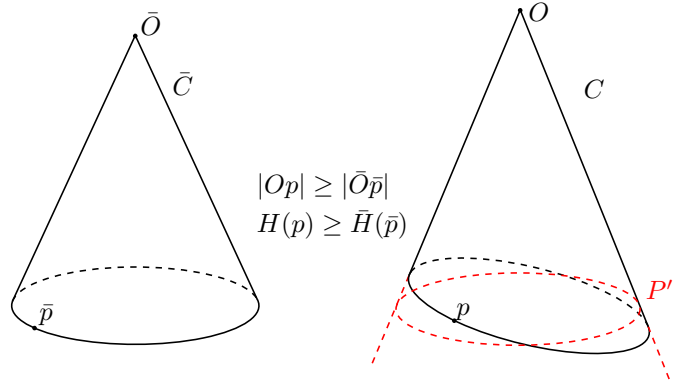
	\endgroup

    \caption{Comparison between $\bar{C}$ and $C$}
    \label{fig:comparison-between-c1-and-c}
\end{figure}

	Based on condition \ref{it:G0cone}, for any $q \in \partial _F C \cap \partial _B C $, we have (recall that $|\bar{q}|=\sqrt{1+\bar{a}^2}$ for any $\bar{q} \in \partial _F \bar{C}\cap \partial _B \bar{C}$),
	\begin{align}
		|q|\ge{} & \sqrt{1+\bar{a}^2}.\label{eq:pfLength}
	\end{align}
	For any $\theta \in [0,2\pi]$, we write $t_\theta$ such that $(a_1t_\theta\cos \theta,a_2t_\theta\sin \theta,-t_\theta) \in \partial _F C \cap \partial _B C$ and denote $q_\theta=(a_1t_\theta\cos\theta,a_2t_\theta\sin\theta,-t_\theta)$.
	The condition \eqref{eq:pfLength} can be rewritten as,
	\begin{equation}
		t^2_\theta(a_1^2\cos ^2\theta+a_2^2\sin ^2\theta+1)\ge 1+\bar{a}^2.
		\label{eq:pfLenComp}
	\end{equation}

	On the other hand, we note the mean curvature of $\partial _F C$ is given by (see \eqref{append:mean} in Appendix \ref{appdsec:meanAngleCone})
	\begin{equation}
		H^\delta_{\partial _F C}(q)=
		\frac{a_1a_2(a_1^2\cos ^2\theta+a_2^2\sin ^2\theta+1)}{t\left( a_1^2a_2^2+a_2^2\cos ^2\theta+a_1^2\sin ^2\theta \right)^{\frac{3}{2}}},
		\label{eq:pfMeanC}
	\end{equation}
	if $q=(a_1t\cos \theta,a_2t\sin \theta,-t)$.
	In particular, if we choose $C=\bar{C}$ in \eqref{eq:pfMeanC}, we have
	\[
		H^\delta_{\partial _F\bar{C}}(\bar{q})=\frac{1}{\bar{a}\sqrt{\bar{a}^2+1}}
	\]
	for any $\bar{q}\in \partial _F\bar{C}\cap \partial _B\bar{C}$.
	Hence, for any $\theta \in [0,2\pi]$, based on condition \ref{it:G0mean}, we have
	\begin{equation}
		\frac{a_1a_2(a_1^2\cos ^2\theta+a_2^2\sin ^2\theta+1)}{t_\theta\left( a_1^2a_2^2+a_2^2\cos ^2\theta+a_1^2\sin ^2\theta \right)^{\frac{3}{2}}}\ge \frac{1}{\bar{a}\sqrt{\bar{a}^2+1}}.
		\label{eq:pfMeanComp}
	\end{equation}

	Now, if we choose $\theta=0$ in \eqref{eq:pfLenComp} and \eqref{eq:pfMeanComp}, we have
	\begin{align*}
		t_\theta \sqrt{a_1^2+1}\ge{}& \sqrt{1+\bar{a}^2},\\
		\frac{a_1}{t_\theta a_2^2\sqrt{ a_1^2+1 }}\ge{}& \frac{1}{\bar{a}\sqrt{\bar{a}^2+1}}.
	\end{align*}
	Hence, we have
	\begin{equation}
		a_1 \bar{a}\ge a_2^2 \frac{t_\theta \sqrt{a_1^2+1}}{\sqrt{\bar{a}^2+1}}\ge a_2^2.
		\label{eq:pfAbound}
	\end{equation}
	Similarly, we have
	\begin{equation}
		a_2 \bar{a}\ge a_1^2.
		\label{eq:pfBbound}
	\end{equation}
	From \eqref{eq:pfAbound} and \eqref{eq:pfBbound}, we get
	\begin{align*}
		a_1\bar{a}^3\ge{}& a_2^2\bar{a}^2\ge a_1^4\implies \bar{a}\ge a_1,\\
		a_2\bar{a}^3\ge{}& a_1^2\bar{a}^2\ge a_2^4\implies \bar{a}\ge a_2.
	\end{align*}
	Hence, we finish the proof of claim.

	Now, we choose $P'=\left\{ x_3=-1 \right\}$. See Figure \ref{fig:comparison-between-c1-and-c} for a location of $P'$.
	Let us show that
    \begin{equation}
        \cos\measuredangle _\delta^{C_{a_1a_2}^{P'}}(q)\le \frac{\bar{a}}{\sqrt{1+\bar{a}^2}}\quad \text{ for }q \in \partial _FC_{a_1a_2}^{P'}\cap \partial _B C_{a_1a_2}^{P'}. 
        \label{eq:pfAngleInq}
    \end{equation}
	Indeed, we know the dihedral angle $\measuredangle _g ^{C_{a_1a_2}^{P'}}(p)$ can be computed by the following formula (cf. \eqref{append:angle} in Appendix \ref{appdsec:meanAngleCone}):
	\[
		\cos \measuredangle _\delta ^{C_{a_1a_2}^{P'}}(q)=\frac{a_1a_2}{\sqrt{a_1^2a_2^2+a_2^2\cos \theta+a_1^2\sin ^2\theta}}.
	\]
	Hence, \eqref{eq:pfAngleInq} is equivalent to
	\begin{equation}
		\frac{a_2^2\cos ^2 \theta+a_1^2\sin ^2\theta}{a_1^2a_2^2}\ge \frac{1}{\bar{a}^2}.
		\label{eq:pfRightConeAngle}
	\end{equation}
	This is an esay consequence of $a_1\le \bar{a}, a_2\le \bar{a}$.
	
	Hence, $P:=L^{-1}(P')$ is the plane we are looking for.

	Now let us suppose equality \eqref{eq:propAngleG0} holds for some $p$, we know equality \eqref{eq:pfRightConeAngle} holds for some $\theta \in [0,2\pi]$, which implies at least one of $a_1,a_2$ equals to $\bar{a}$.
	We suppose $a_1=\bar{a}$.
	But based on the proof of $\bar{a}\ge a_1$, we know $\bar{a}=a_1$ only if \eqref{eq:pfAbound} and \eqref{eq:pfBbound} are all equalities.
	This implies $a_1=a_2=\bar{a}$.

	Hence, \eqref{eq:pfRightConeAngle} holds for every $\theta \in [0,2\pi]$.
\end{proof}

Now, we are ready to construct the desired surface.

\begin{proposition}
	\label{prop_barrierG0neqD}
	If $g_0\neq \delta$, then we can construct a surface $\Sigma$ near $0$ such that $\Sigma$ has positive mean curvature and the contact angle $\measuredangle _g(\Sigma,\partial M)>\bar{\gamma}$ at any point of $\partial \Sigma$.
\end{proposition}

\begin{proof}
Recall that $\partial M$ is given by
\[
	\left\{ (\psi(t)\hat{x},-t):t \in [0,\varepsilon),\hat{x} \in \partial D \right\}
\]
near 0.
We assume $g=g_0+O(t)$ and $g_0\neq\delta$.

From Proposition \ref{prop_coneG0}, we know we can find a plane $P$, which is of distance 1 from $O$ such that
\[
	\lim_{t\rightarrow 0} \measuredangle _g(\partial M,tP)> \lim_{t\rightarrow 0} \bar{\gamma}_t,
\]
where we can actually find $\lim_{t\rightarrow 0} \bar{\gamma}_t=\arctan(\frac{1}{\psi'(0)})$.
Here $tP:=\left\{ tp:p \in P \right\}$.

Now, we choose $\Sigma_t:=tP\cap M$.
Let $P_0$ be the plane in $\mathbb{R}^3$ which is parallel to $P$ and also passes through $O$, $\bar{N}_0$ the unit normal vector of $P_0$ under metric $\delta$.
We write $E:=\lim_{t\rightarrow 0} \frac{1}{t}(tP\cap M-t\bar{N}_0)$ in the sense of sets and note that $E$ is indeed an ellipse in $P_0$.

By the implicit function theorem, we know $\Sigma_t$ can be parameterized by
\[
	\varphi(t,\hat{x})-t \bar{N}_0 \text{ for }\hat{x} \in E,
\]
where $\varphi(t,\hat{x}):[0,\varepsilon)\times E \rightarrow P_0$ such that $\varphi(t,\hat{x})=t\hat{x}+O(t^2)$.
For simplicity, we use $\Sigma_t(\hat{x})$ to denote such a parametrization.
Now for any $u \in C^{2,\alpha}(E)\cap C^{1,\alpha}(\bar{E})$, we can define
\[
	\Sigma_{t,u}:=\left\{ \Sigma_{t+\frac{u(\hat{x})}{\left<  Y_t(\hat{x}), N_t(\hat{x}) \right>_g}}(\hat{x}):\hat{x} \in E \right\}.
\]
where $Y_t(\hat{x}):=\frac{\partial }{\partial t}\Sigma_t(\hat{x})$ is the variational vector field generated by foliation $\Sigma_t$ and $N_t$ is the normal vector field of $\Sigma_t$ under metric $g$.

Note that
\begin{align*}
	\lim_{t\rightarrow 0} N_t(\hat{x})={}&N_0,\\
	\lim_{t\rightarrow 0} Y_t(\hat{x})={}& \hat{x}-\bar{N}_0,\\
	\lim_{t\rightarrow 0} \left< Y_t(\hat{x}), N_t(\hat{x}) \right> _g={}&-\left< N_0,\bar{N}_0 \right>_{g_0}.
\end{align*}
Hence, $-\left< Y_t(\hat{x}),N_t(\hat{x}) \right>$ has a positive lower bound if we choose $t$ sufficient small.
In particular, $\Sigma_{t,u}$ is well-defined if $u=o(t)$.
Replace $u$ by $t^2u$ and assume $u=O(1)$, we have
 \begin{align*}
 	H_{t,t^2u}={} & \Delta_t^E u + H_t+O(t)\\
   \frac{\cos \gamma_{t,t^2u}-\cos \gamma_t}{t}={} & \sin \gamma_t \frac{\partial u}{\partial \nu_t^E}+(-A_{\partial M}(\eta_t,\eta_t)+\cos \gamma_t A_t(\nu_t,\nu_t) )t u+O(t)
 \end{align*}
where $\Delta_t^E$ is the Laplace-Beltrami operator on $E$ under the metric $\frac{1}{t^2}\Sigma_t^*(g)$, $\nu_t^E$ the outer unit normal vector field of $\partial E$ in $E$ under metric $\frac{1}{t^2}\Sigma_t^*(g)$. It is easy to verify that the metric $\frac{1}{t^2}\Sigma_t^*(g)$ converges to the constant metric $g_0$ as $t\to 0$.
As usual, we use subscript $t$ to denote the quantities associated with surface $\Sigma_t$ under metric $g$.
\begin{remark}
    Note that $\gamma_t$ is defined as the dihedral between $\partial M$ and $\Sigma_t$ under metric $g$.
    But the value of $\bar{\gamma}$ at point $(\psi(t)\hat{x},t)$ is still defined as the dihedral angle between $\{x_3=-t\}$ and $\partial M$, and $\bar{\gamma}_{t}$ is the function $\bar{\gamma}$ restricted on boundary of $\Sigma_{t}$.
\end{remark}
We define a map $\Psi(t,u)$ as follows
\begin{equation}
  \Psi (t, u) = \left( H_{t, t^2 u},
 t^{- 1} ( \cos \gamma_{t, t^2 u}- \cos \gamma_t) \right), \text{ }t>0. 
\end{equation}
The difference to \eqref{Psi} is that we do not subtract the average of $H_{t,t^2 u}$ over $E$ from the first component of $\Psi$ and we subtract $\cos\gamma_t$ instead of $\cos\bar{\gamma}_{t,t^2u}$ from the second component of $\Psi$.  We extend the definition to $t=0$ by taking limits, we see that
\begin{equation}
\Psi(0,u) = \lim_{t\to 0} \Psi(t,u) =\left(\Delta^E_0 u+H_0, \frac{\partial u}{\partial \nu_{0}^E}+hu\right),
  \end{equation}
where 
\begin{align}
    h:={}&\lim_{t\to 0}(-A_{\partial M}(\eta_t,\eta_t)+\cos \gamma_t A_t(\nu_t,\nu_t) )t,
\end{align}
and $\Delta^E_0$ is the Laplace-Beltrami operator on $E$, $\nu^E_0$ the outer unit normal vector field of $\partial E$ in $E$ and both of them is defined  under constant metric $g_0$ restricted on $E$.
Note that $|A_t|, |A_{\partial M}|$ are all bounded by $\frac{C}{t}$ when $t$ is small for a fixed constant $C$, we know $h$ is well-defined.

With the standard arguments for the elliptic PDE (Fredholm alternative), we can prove the following existence results.

\begin{lemma}
	\label{lem_exist_Robin}
	We consider the following elliptic PDE problem with the Robin boundary condition:
	\begin{equation}
		\begin{cases}
		\Delta_0^Eu=f_1, & \text{ in }E \\
		\frac{\partial u}{\partial \nu^E_0}+hu=f_2, & \text{ on }\partial E.
		\end{cases}
		\label{eq:RobinPDE}
	\end{equation}
	Let $\mathcal{N}$ be the collection of all $u \in C^{2,\alpha}(E)\cap C^{1,\alpha}(\overline{E})$ such that it solves Problem \eqref{eq:RobinPDE} with $f_1=0$ and $f_2=0$ (the collection of solutions to the corresponding homogeneous problem).
	Then $\mathcal{N}$ is a finite dimensional space and Problem \eqref{eq:RobinPDE} has solutions if and only if $f_1,f_2$ satisfies
	\[
		\int_{ E} f_1u+\int_{ \partial E} f_2u=0,
	\]
	for any $u \in \mathcal{N}$.
\end{lemma}

Based on Lemma \ref{lem_exist_Robin}, we claim that we can find a solution to $\Psi(0,u)=(K,f)$ for any constant $K\in \mathbb{R}$ and some suitable function $f \in C^{1,\alpha}(\partial E)$.

To prove this claim, we choose $\left\{ v_k \right\}_{k=1}^m\subset C^{2,\alpha}(E)\cap C^{1,\alpha}(\overline{E})$ to be a basis of $\mathcal{N}$.
At first, we show that $\left\{ (v_k)|_{\partial E} \right\}_{k=1}^m$ is a linearly independent subset of $C^{1,\alpha}(\partial E)$.
If not, that means we can find a non-zero vector $(a_1,a_2,\cdots ,a_m)\in \mathbb{R}^m$ such that if we define $v:=\sum_{k =1}^{m}a_kv_k$, we have $v=0$ on $\partial E$ and $v$ also solves the Problem \eqref{eq:RobinPDE} with $f_1=0,f_2=0$.
But the only function satisfying $\Delta_0^Eu=0$ in $E$ and $u=0$ on $\partial E$ is the zero function.
Hence $v=0$ in $E$ and this contradicts that $\left\{ v_k \right\}$ is a basis of $\mathcal{N}$.
Hence, for any $K \in \mathbb{R}$, we can find the following linear system
\[
	\sum_{l=1 }^{m}a_l\int_{ \partial E}v_kv_l=-\int_{ E} v_k(K-H_0), \text{ for }k=1,2,\cdots ,m,
\]
has a unique solution $(a_1,a_2,\cdots ,a_m) \in \mathbb{R}^m$.
Then, we can choose $f=\sum_{i =1}^{m}a_l v_l$ and based on Lemma \ref{lem_exist_Robin}, we can find $u_0$ such that $\Psi(0,u_0)=(K,f)$ and finish the proof of this claim.

Now, we choose $u_t=u_0$ and consider the mean curvature $H_{t,t^2u_t}$ and contact angle $\gamma_{t,t^2u_t}$ based on $u_t$ we just constructed.
At first, we can choose $K$ to be a sufficiently large number such that $K-H_0>0$.
Then, by the continuity, we have $H_{t,t^2u_t}=\Psi_1(t,u_t)>0$ for any $t>0$ sufficient small.
On the other hand, we have
\[
	\cos \gamma_{t,t^2u_t}=t \Psi_2(t,u_t)+\cos \gamma_t.
\]
Based on the construction of $\Sigma_t$ (Proposition \ref{prop_barrierG0neqD}), we know there exists sufficient small positive numbers $\varepsilon>0$ and $\delta>0$ such that $\cos \gamma_t < \cos \overline{\gamma}_t-\varepsilon$ for all $0<t<\delta$.
By the continuity of $\Psi_2$, we can find $\cos \gamma_{t,t^2u_t}\le \cos \overline{\gamma}_t-\frac{\varepsilon}{2}$ if we choose $t$ sufficiently small.
This mean the contact angle $\gamma_{t,t^2u_t}> \overline{\gamma}_{t,t^2 u_t}$.
Hence, $\Sigma_{t,t^2u_t}$ is a barrier.
\end{proof}

\subsection{Proof of conical case of main theorem}%
\label{sub:conical}

Now we can finish the proof of the conical case.
\begin{proof}[Proof of case 2 of Theorem \ref{main}]
	Our goal is to choose a barrier $\Sigma$ near the conical point.
  If $g=\delta$ at the conical point, then by Lemma \ref{nonnegativity of mean curvature of every leaf}, we can find $t> 0$ small such that $\Sigma_{t,t^2u_t}$ is of nonnegative mean curvature and
  meets the boundary $\partial M$ at prescribed angles $\bar{\gamma}$.
  Hence, we choose $\Sigma=\Sigma_{t}$.

  If $g\neq \delta$ at the conical, then based on Proposition \ref{prop_barrierG0neqD}, we choose $\Sigma$ such that the conclusion of Proposition \ref{prop_barrierG0neqD} is true.

  So
  $\Sigma$ is a barrier for the existence of the minimiser of
  \eqref{action} and we can find a minimiser of \eqref{action} in between
  $\Sigma$. The rest of the proof is the same as Section \ref{slab
  case}.
\end{proof}

\section{Spherical case}%
\label{sec:spherical}

In this section, we consider the spherical case of Theorem \ref{main}. Unlike the conical case, we need a finer analysis of asymptotic behaviors of the mean curvatures and contact angles near the spherical points.
To illustrate the key idea to the construction of foliations, we will work under the assumption $g=\delta$ at the spherical point in Subsections \ref{sub:asymptotic_analysis}-\ref{sub:cmc_foliations} at first.
We carry out the asymptotic analysis in Subsection \ref{sub:asymptotic_analysis} and showed that we can construct a good foliation near the spherical point $O$ under the assumption $g=\delta$ at $O$ in Subsections \ref{sub:linearized}, \ref{sub:cmc_foliations}.

In Subsections \ref{sub:non_euclideanSphere}-\ref{sub:construction_of_foliation}, we only need to assume $\sigma \ge \bar{\sigma}$ at the spherical point only to construct the desired barrier. The analysis applies as well to the case considered in previous Subsections \ref{sub:asymptotic_analysis}-\ref{sub:cmc_foliations} where the stronger assumption that $g = \delta$ at the spherical point is used.
Finally, we can finish the proof of the last case of our main theorem in Subsection \ref{sub:proof_of_spherical_case_of_main_theorem}.


Now, we assume $\partial M$ is given by
\[
	(\phi(t)\hat{x}, -t^k),\quad \hat{x} \in \partial D, t \in [0,\varepsilon)
\]
near vertex $p_+$ for some smooth function $\phi(t)$ defined on $[0,\varepsilon)$ with $\phi'(0)=0$ and $\phi'(0)>0$.
Without loss of generality, we assume $p_+=O$.

In particular, if $k=1$, then $\partial M$ is of a conical case.
This is the case we have considered in the last section.
In general, we only consider the case of $k\ge 2$ in this section.
If $k=2$, then $\partial M$ is of a spherical case.
Note that $M$ is given by $(\phi(t)\hat{x},-t^k)$ for $\hat{x} \in D$ and $t \in [0,\varepsilon)$ near $O$.

We define surface $\Sigma_t$ by
\[
	\Sigma_t=\{ z=-t^k \}\cap M,
\]
and for any $u \in C^{2,\alpha}\cap C^{1,\alpha}(\bar{D})$, we define
\[
	\Sigma_{t,u}:= \{ \phi((t^k+u(\hat{x}))^{\frac{1}{k}}\hat{x},-t^k-u(\hat{x})):\hat{x}\in D \}.
\]

Note that $\Sigma_{t,u}$ is well-defined if $u=o(t^k)$ for small $t$ enough.

Now we assume $u=O(t^{k+1})$.
Replace $u$ by $t^{k+1}u$, similar from \eqref{mean curvature taylor expansion} and \eqref{angle variation2}, we have
\begin{align}
	\frac{H_{t,t^{k+1}u}}{t^{k-1}}={} & \frac{t^2}{\phi^2(t)}\Delta_t^D u +\frac{H_t}{t^{k-1}}+O(t)\label{eq:meank} \\
	\frac{\cos \gamma_{t,t^{k+1}u}-\cos \bar{\gamma}_{t,t^{k+1}u}}{t^{2k-1}}={} & \frac{\sin \gamma_t}{\phi(t)t^{k-2}} \frac{\partial u}{\partial \nu_t^D}+(A_{\partial M}(\eta_t,\eta_t)-\cos \gamma_tA(\nu_t,\nu_t)\nonumber \\
	&-\bar{A}_{\partial M}(\bar{\eta}_t,\bar{\eta}_t))\frac{u}{t^{k-2}}+\frac{\cos \gamma_t-\cos \bar{\gamma}_t}{t^{2k-1}}+O(t).
	\label{eq:anglek}
\end{align}
Here, we use the subscript $t$ to denote the quantities related to surface $\Sigma_t$ and the subscript $(t,t^{k+1}u)$ to denote the quantities related to surface $\Sigma_{t,t^{k+1}u}$.
In particular, if those quantities are computed with respect to the metric $\delta$, we add a bar for those quantities.

To construct a constant mean curvature foliation, we need several preliminary results to understand the asymptotic behavior of $\frac{H_t}{t^{k-1}}$ and $\frac{\cos \gamma_t -\cos \bar{\gamma}_t}{t^{2k-1}}$.

\subsection{Asymptotic analysis}%
\label{sub:asymptotic_analysis}
We now compute the asymptotic behavior of $\cos \gamma_{t}$ and $H_t$ near $t=0$.
Note that all the computation works for general $k$ here.
In this subsection, we will assume $g=\delta$ at $O$.
But indeed, all the result in this section is valid with $g_0$ in place of $\delta$ if $g=g_0$ at $O$ where $g_0$ is a constant metric on $\mathbb{R}^3$.

\begin{proposition}
	Suppose $P$ and $P'$ are two planes in $\mathbb{R}^{3}$ with $p \in P\cap P'$.
	Let $g$ be the metric on $\mathbb{R}^3$ such that $g=\delta+ht+O(t^2)$ at $p$ where $|h|\le C$ for some fixed positive constant $C$.
	If $|P-P'|\le Ct^k$ for some fixed constant $C$ and small $t>0$ with $k\ge 1$, then we have
	\[
		\cos \measuredangle _g(P,P')=\cos \measuredangle_\delta (P,P')+L(h)t^{2k+1}+O(t^{2k+2})
	\]
	at point $p$.
	\label{prop:angPlane}
\end{proposition}

Here, $\measuredangle _g(P,P')$ denotes the angle between $P$ and $P'$ under metric $g$, and $L(h)$ denotes the term relying on $h$ linearly and it is uniformly bounded by a constant not related to $t$.

\begin{remark}
    Proposition \ref{prop:angPlane} is true even for $k=0$ in view of the formula (2.5) in \cite[Lemma 2.1]{miao-mass-2021}.
\end{remark}
We need the following lemma.
\begin{lemma}
	\label{lem_angleG}
	If $g=\delta+ht+O(t^2)$ with $h=O(1)$ is a symmetric two-tensor, and $u,v$ are two vectors such that $u=v+w$ where $w=O(t^{k})$ with $k\ge 1$. We also assume $\frac{1}{C}<|v|<C$ for some fixed constant $C$.
	Then we have
	\[
		\frac{\left<u,v\right>}{\sqrt{\left<u,u\right>\left<v,v\right>}}=1-
		\frac{|w|^2u\cdot v-(u\cdot w)(v\cdot w)}{2(u\cdot v)^2}+L(h)t^{2k+1}+O(t^{2k+2}).
	\]
In particular, we know
	\[
		\frac{\left<u,v\right>}{\sqrt{\left<u,u\right>\left<v,v\right>}}=
		\frac{u\cdot v}{|u||v|}+L(h)t^{2k+1}+O(t^{2k+2}).
	\]
\end{lemma}

\begin{proof}
	By direct computation, we have
	\begin{align*}
		\frac{\left<u,v\right>}{\sqrt{\left<u,u\right>\left<v,v\right>}}={} & \frac{\left<u,v\right>}{\sqrt{\left<u,v\right>^2+\left<w,w\right>\left<u,v\right>-\left<u,w\right>\left<v,w\right>}}\\
		={}& \left( 1+\frac{\left<w,w\right>\left<u,v\right>-\left<u,w\right>\left<v,w\right>}{\left<u,v\right>^2} \right)^{-\frac{1}{2}} \\
		={}& 1-\frac{\left<w,w\right>\left<u,v\right>-\left<u,w\right>\left<v,w\right>}{2\left<u,v\right>^2}+O(t^{2k+2})\quad \text{since }w=O(t^k) \\
		={}& 1-\frac{|w|^2u\cdot v-(u\cdot w)(v\cdot w)}{2(u\cdot v)^2}+L(h)t^{2k+1}+O(t^{2k+1}).
	\end{align*}

	For the last equality, we have used $g=\delta+ht+O(t^2)$.

	Now if we replace $g$ by $\delta$ in the above computation, we have
	\[
		\frac{u\cdot v}{|u||v|}=1-\frac{|w|^2u\cdot v-(u\cdot w)(v\cdot w)}{2(u\cdot v)^2}+O(t^{2k+2}).
	\]
	Hence, we find
	\[
		\frac{\left<u,v\right>}{\sqrt{\left<u,u\right>\left<v,v\right>}}=
		\frac{u\cdot v}{|u||v|}+L(h)t^{2k+1}+O(t^{2k+2}).
	\]
\end{proof}
\begin{remark}\label{lem_angleG0}
	Lemma \ref{lem_angleG} shows that if two vectors $u,v$ are sufficiently close to each other, then the lower order terms of the angle between $u,v$ is independent of $g$ as long as $g$ is close to $\delta$. The proof works for any constant metric $g_0$ replacing the Euclidean metric.
\end{remark}

\begin{proof}
	[Proof of Proposition \ref{prop:angPlane}]
	Suppose $\nu$ ($\nu'$ respectively) is the unit normal vector of $P$ ($P'$ respectively) under the metric $\delta$.


	We write $\left\{ \tau_1,\tau_2 \right\}$ ($\left\{ \tau_1',\tau_2' \right\}$ respectively) as the orthonormal basis of $P$ ($P'$ respectively).
	Note that since $|P-P'|\le Ct^k$, we can choose $\tau_\alpha'=\tau_\alpha+\xi_\alpha$ for $\alpha=1,2$ with $\xi_\alpha = O(t^k)$.

	We define
	\[
		g^P_{\alpha\beta}=\left< \tau_\alpha,\tau_\beta \right>,\quad g^{P'}_{\alpha\beta}=\left< \tau_\alpha',\tau_\beta' \right>,\quad 1\le \alpha,\beta\le 2,
	\]
	and write $g^{\alpha\beta}_{P}$ ($g^{\alpha\beta}_{P'}$) as the inverse of $g^P_{\alpha\beta}$ ($g_{\alpha\beta}^{P'}$ respectively).
	We write $\tilde{\nu}$ ($\tilde{\nu}'$ respectively) as the projection of $\nu$ ($\nu'$ respectively) to the normal of $P$ ($P'$ respectively) under metric $g$.
	In other words, we define
	\begin{align*}
		\tilde{\nu}:={}&\nu-g_{P}^{\alpha\beta}\left< \nu,\tau_\alpha \right>\tau_\beta ,\\
		\tilde{\nu}':={}&  \nu'-g_{P'}^{\alpha\beta}\left< \nu',\tau_\alpha' \right>\tau_\beta'.
	\end{align*}

	Now we only need to show
	\[
		\frac{\left< \tilde{\nu},\tilde{\nu}' \right> }{|\tilde{v}|_g|\tilde{v}'|_g}=\nu\cdot\nu'+L(h)t^{2k+1}+O(t^{2k+2}).
	\]

	To do that, let us suppose $\nu'=\nu+w$ for $w =O(t^k)$ and we prove $\tilde{\nu}'=\tilde{\nu}+w+L(h)t^{k+1}+O(t^{k+2})$ at first.

	Indeed, note that $g_{\alpha\beta}^{P'}=g_{\alpha\beta}^{P}+A+O(t^{k+1})$ since $\xi_\alpha=O(t^k)$, where $A$ is an $O(t^{k})$ term not related to $h$. Then we have
	\begin{align*}
		\tilde{\nu}'={} & \nu+w-g_{P'}^{\alpha\beta} h(\nu',\tau'_\alpha) \tau_\beta' \\
		={} & \nu+w-g^{\alpha\beta}_{P}h(\nu',\tau'_\alpha) \tau_\beta'+A h(\nu' \cdot \tau_\alpha')\tau_\beta'+O(t^{k+2})\\
		={}& \nu+w-g^{\alpha\beta}_P h(\nu,\tau_\alpha)\tau_\beta-g^{\alpha\beta}_Ph(\nu,\tau_\alpha )\xi_\alpha - g^{\alpha\beta}_P h(\nu,\xi_\alpha)\tau_\beta-g^{\alpha\beta}_P h(w,\tau_\alpha)\tau_\beta\\
		   &+L(h)t^{k+1}+O(t^{k+2})\\
		={}& \nu+w-g^{\alpha\beta}_P h(\nu,\tau_\alpha) +L(h)t^{k+1}+O(t^{k+2})\\
		={}& \tilde{\nu}+w+L(h)t^{k+1}+O(t^{k+2}).
	\end{align*}

	In view of Lemma \ref{lem_angleG}, we know
	\begin{align*}
		\frac{\left< \tilde{\nu},\tilde{\nu}' \right> }{|\tilde{\nu}||\tilde{\nu}'|}={} & 1- \frac{|w|^2 \tilde{\nu}\cdot \tilde{\nu}' -(\tilde{\nu}\cdot w)(\tilde{\nu}'\cdot w)}{2(\tilde{\nu}\cdot \tilde{\nu}')^2}+L(h)t^{2k+1}+O(t^{2k+2}) \\
		={} & 1-\frac{|w|^2 \nu\cdot \nu'-(\nu \cdot w)(\nu' \cdot w)}{2(\nu \cdot \nu')^2}+L(h)t^{2k+1}+O(t^{2k+2})\\
		={}& \nu\cdot \nu'+L(h)t^{2k+1}+O(t^{2k+2}).
	\end{align*}
	Here, we have used $\tilde{\nu}=\nu+O(t)$ and $\tilde{\nu} '=\nu'+O(t)$.
	This is what we want.
\end{proof}

\begin{corollary}
	\label{cor_angle}
	For $k\ge 2$, 
	we have the following results:
	\begin{align*}
		\cos \gamma_t={} & \cos \bar{\gamma}_t+L(h)t^{2k-1}+O(t^{2k}),\\
		\sin \gamma_t={} & \sin \bar{\gamma}_t+O(t^{k})=\frac{kt^{k-1}}{\phi'(0)}+O(t^k).
	\end{align*}
\end{corollary}

\begin{proof}
	For any $p=(\phi(t)\hat{x},-t^k)$ with $\hat{x} \in \partial D$, we only need to choose $P=T_p\Sigma_t$ and $P'=T_p \partial M$.
	Then, under metric $\delta$, $e_3$ is the unit normal vector of $P$ and $\nu':=\frac{(kt^{k-1}\hat{x},\phi'(t))}{\sqrt{(\phi'(t))^2+k^2t^{2k-2}}}$ is the unit normal vector of $P'$.
	It is easy to verify that $\nu'=e_3+O(t^{k-1})$ and hence we can apply Proposition \ref{prop:angPlane} to get
	\[
		\cos \gamma_t=\cos \bar{\gamma}_t+L(h)t^{2k-1}+O(t^{2k}).
	\]

	For another result, we only need to use the standard trigonometric identity $\sin \gamma = \sqrt{1-\cos ^2\gamma}$ to get $\sin \gamma_t=\sin \bar{\gamma}_t+O(t^k)$.

	At last, we know
	\[
		\sin \bar{\gamma}_t=|\nu'\times  e_3|=\frac{k t^{k-1}}{\sqrt{(\phi'(t))^2+k^2t^{2k-2}}}=\frac{kt^{k-1}}{\phi'(0)}+O(t^k).
	\]
\end{proof}

Now, we need to find the lower order term of $H_t$.
For fixed $t$, we define two functions related to the mean curvature of $\partial M$.
For any $\hat{x} \in D$, we define
\begin{align*}
	H_{t, \partial M}(\hat{x}):={}&H_{\partial M}(\phi(t)\hat{x},-[\phi^{-1}(\phi(t)|\hat{x}|)]^k)\\
	\bar{H}_{t, \partial M}(\hat{x}):={}&\bar{H}_{\partial M}(\phi(t)\hat{x},-[\phi^{-1}(\phi(t)|\hat{x}|)]^k)
\end{align*}

Note that $(\phi(t)\hat{x},-[\phi^{-1}(\phi(t)|\hat{x}|)]^k)\in \partial M$, so $H_{t, \partial M}$ and $\bar{H}_{t, \partial M}$ are all well-defined.

We also define
\[
	\tilde{H}_t(\hat{x})=H_{\phi^{-1}(\phi(t)|\hat{x}|)}\left(\hat{x}/|\hat{x}|\right).
\]
In other word, $\tilde{H}_t(\hat{x})$ is the mean curvature of plane $\left\{ z=-[\phi^{-1}(\phi(t)|\hat{x}|)]^k \right\}$ at point $(\phi(t)\hat{x},-[\phi^{-1}(\phi(t)|\hat{x}|)]^k)$ under metric $g$.

\begin{proposition}
	For any $k\ge 2$, 
	we have
	\begin{equation}
		H_{t, \partial M}=\bar{H}_{t, \partial M}+\tilde{H}_t+L(D_{P_0}g(0))t^{k-1}+O(t^{k}).
		\label{eq:thmMean}
	\end{equation}
	\label{prop:meanTaylor}
\end{proposition}
Here, $P_0$ is the plane spanned by $e_1,e_2$ and $D_{P_0}g$ means the derivative of $g$ only along the directions in $P_0$.
Hence, $L(D_{P_0}g(0))$ just means the term relies on $D_{P_0}g(0)$ linearly and it is uniformly bounded by a constant not related to $t$.
Note that Proposition \ref{prop:meanTaylor} is also true for $k=1$ in view of \cite[Proposition 2.1]{miao-mass-2021}.

\begin{proof}
		Indeed, we only need to consider the case $\hat{x} \in \partial D$.
	This is because if we suppose \eqref{eq:thmMean} holds for any $\hat{x} \in \partial D$, we consider any other point $\hat{y} \in D$.
	We write $\tau=\phi^{-1}(\phi(t)|\hat{y}|)$ and we know $\left(\phi(\tau) \frac{\hat{y}}{|\hat{y}|},-\tau^k\right)\in \partial M$.
	We use the formula \eqref{eq:thmMean} with $\tau$ in place of $t$ and get
	\begin{align*}
		H_{t, \partial M}(\hat{y})={}&
		H_{\tau,\partial M}\left(\frac{\hat{y}}{|\hat{y}|}\right)=\bar{H}_{\tau,\partial M}\left(\frac{\hat{y}}{|\hat{y}|}\right)+\tilde{H}_\tau\left( \frac{\hat{y}}{|\hat{y}|} \right)+O(\tau^{k-1})\\
		={}& \bar{H}_{t, \partial M}(\hat{y})+\tilde{H}_t(\hat{y})+O(t^{k-1}).
	\end{align*}
	since $\tau\le t$.
	Hence, \eqref{eq:thmMean} holds for any $\hat{y} \in D$.



	Note that if $\hat{x} \in \partial M$, we can find \eqref{eq:thmMean} is equivalent to the following identity
	\begin{equation}
		H_{\partial M}(\phi(t)\hat{x},-t^k)=\bar{H}_{\partial M}(\phi(t)\hat{x},-t^k)+H_t(\hat{x})+O(t^{k-1}).
		\label{eq:pfmeanTaylor}
	\end{equation}

	To show \eqref{eq:pfmeanTaylor} holds for any $\hat{x} \in \partial D$, let us choose some parametrizations for $\partial M$ and $\Sigma_t$ near $p:=(\phi(t)\hat{x},-t^k)$ at first.
	Without loss of generality, we suppose $\hat{x}=(1,0)$.

	We consider the map
	\begin{align*}
		\varphi_{\partial M}(u_1,u_2)={}&(\phi(u_1)\cos u_2,\phi(u_1)\sin u_2,-u_1^k)\\
		\varphi_{\Sigma_t}(u_1,u_2)={}&  (\phi(u_1)\cos u_2,\phi(u_1)\sin u_2,-t^k)
	\end{align*}
	for $u_1$ close to $t$ and $u_2$ close to 0.

Denote $Y^{\partial M}_\alpha:=\frac{\partial \varphi_{\partial M}}{\partial u_\alpha}$ and $Y^{\Sigma_t}_\alpha:=\frac{\partial \varphi_{\Sigma_t}}{\partial u_\alpha}$.
Now, we consider the matrix $g^{\cdot}_{\alpha\beta}$ defined by
\[
	g^{\cdot}_{\alpha\beta}=\left< Y^{\cdot}_\alpha,Y^{\cdot}_{\beta} \right>.
\]
Here, we take $\cdot$ to be $\partial M$ or $\Sigma_t$.

Similarly, we write
\[
	\delta^{\cdot}_{\alpha\beta}=Y^{\cdot}_\alpha\cdot Y^{\cdot}_\beta.
\]

We also write $g^{\alpha\beta}_{\cdot}$ ($\delta^{\alpha\beta}_{\cdot}$ respectively) as the inverse matrix of $g_{\alpha\beta}^{\cdot}$ ($\delta_{\alpha\beta}^{\cdot}$ respectively).

With these notations, we can write mean curvatures in the following ways,
\begin{align*}
	H_{\partial M}={} & g^{\alpha\beta}_{\partial M}\left< \nabla_{Y^{\partial M}_\alpha}Y^{\partial M}_\beta,X \right> ,\\
	\bar{H}_{\partial M}={} & \delta^{\alpha\beta}_{\partial M} D_{Y^{\partial M}_\alpha}Y^{\partial M}_\beta\cdot\bar{X},\\
	H_t={}& g^{\alpha\beta}_{\Sigma_t}\left< \nabla_{Y^{\Sigma_t}_\alpha}Y^{\Sigma_t}_\beta,N \right>.
\end{align*}

The remaining proof relies on the following two lemmas.

\begin{lemma}
	\label{lem_meanOnBdry}
	We have
	\[
		g^{\alpha\beta}_{\partial M}\left< D_{Y^{\partial M}_\alpha}Y^{\partial M}_\beta,X \right>=\delta^{\alpha\beta}_{\partial M} D_{Y^{\partial M}_\alpha}Y^{\partial M}_\beta \cdot \bar{X}+L(h)t^{k-1}+O(t^{k}).
	\]
    if we write $g=\delta+ht+O(t^2)$ for some $|h|\le C$.
%
%
\end{lemma}

\begin{proof}
	At first, we know
	\[
		D_{Y^{\partial M}_1}Y^{\partial M}_1=(\phi''(t),0,-k(k-1)t^{k-2}).
	\]
	Note that $\left<  Y^{\partial M}_1,X \right>=0$, then we have
	\begin{align*}
		\left< D_{Y^{\partial M}_1}Y^{\partial M}_1,X \right> ={} & \left< \left(0,0,-k(k-1)t^{k-2}+\frac{kt^{k-1}\phi''(t)}{\phi'(t)}\right),X \right>  \\
		={} & -k(k-1)t^{k-2}+\frac{kt^{k-1}\phi''(t)}{\phi'(t)}+L(h)t^{k-1}+O(t^{k}).
	\end{align*}

	Since $\left< Y^{\partial M}_2,X \right> =0$ and $D_{Y^{\partial M}_1}Y^{\partial M}_2=D_{Y^{\partial M}_2}Y^{\partial M}_1=(0,\phi'(t),0)$, we know
	\[
		\left< D_{Y^{\partial M}_1}Y^{\partial M}_2,X \right> =
		\left< D_{Y^{\partial M}_2}Y^{\partial M}_1,X \right> =0.
	\]

	At last, we have
	\begin{align*}
		\left< D_{Y^{\partial M}_2}Y^{\partial M}_2,X \right> ={} & \left< \left( 0,0,\frac{-kt^{k-1}\phi(t)}{\phi'(t)}\right),X  \right> =-kt^{k}+L(h)t^{k+1}+O(t^{k+2}).
	\end{align*}

	Similarly, we can find
	\begin{align*}
		D_{Y^{\partial M}_1}Y^{\partial M}_1 \cdot \bar{X}={} & -k(k-1)t^{k-2}+\frac{kt^{k-1}\phi''(t)}{\phi'(t)}+O(t^k) \\
		D_{Y^{\partial M}_1}Y^{\partial M}_2 \cdot \bar{X}={} & D_{Y^{\partial M}_2}Y^{\partial M}_1 \cdot \bar{X}=0\\
		D_{Y^{\partial M}_2}Y^{\partial M}_2 \cdot \bar{X}={}& -kt^{k}+O(t^{k+2})
	\end{align*}

	Hence, we find
	\begin{align}
		&\left( \left< D_{Y^{\partial M}_\alpha}Y^{\partial M}_\beta,X  \right>  \right)_{\alpha\beta}\nonumber \\
		={}&\left( D_{Y^{\partial M}_\alpha}Y^{\partial M}_\beta\cdot \bar{X}\right)_{\alpha\beta}+
		\begin{bmatrix}
			L(h)t^{k-1}+O(t^k)& 0\\
			0& L(h)t^{k+1}+O(t^{k+2}) \\
		\end{bmatrix}
		\label{eq:pf2ndExt}
	\end{align}

	Now, let's compute $g^{\partial M}_{\alpha\beta}$. Note
	\begin{align*}
		(g^{\partial M}_{\alpha\beta})_{\alpha\beta}={} & (g(Y^{\partial M}_\alpha,Y^{\partial M}_\beta))_{\alpha\beta} \\
		={} & 
		\begin{bmatrix}
			(\phi')^2+L(h)t+O(t^2) & L(h)t^2+O(t^3)\\
			L(h)t^2+O(t^3)& \phi^2(t)+L(h)t^3+O(t^4)
		\end{bmatrix}\\
		={}& 
		\begin{bmatrix}
			\phi' & 0\\
			0& \phi \\
		\end{bmatrix}(\mathrm{Id}+M)
		\begin{bmatrix}
			\phi' & 0\\
			0&\phi \\
		\end{bmatrix}+O(t^2)
	\end{align*}
	where $M$ is a matrix that each entry of $M$ is of $L(h)t$.
	Hence
	\begin{align*}
		(g^{\alpha\beta}_{\partial M})_{\alpha\beta}={}&
		\begin{bmatrix}
			\frac{1}{\phi'} & 0\\
			0& \frac{1}{\phi} \\
		\end{bmatrix}
		(\mathrm{Id}-M)
		\begin{bmatrix}
			\frac{1}{\phi'} & 0\\
			0& \frac{1}{\phi} \\
		\end{bmatrix}+O(t^2)\\
		={}& (\delta^{\alpha\beta}_{\partial M})_{\alpha\beta}-\frac{1}{(\phi'(0))^2}
		\begin{bmatrix}
			1 & 0\\
			0& \frac{1}{t} \\
		\end{bmatrix}M
		\begin{bmatrix}
			1 & 0\\
			0& \frac{1}{t} \\
		\end{bmatrix}+O(t^2)
	\end{align*}
	
	Combining with \eqref{eq:pf2ndExt}, we find
	\[
		g^{\alpha\beta}_{\partial M}
		\left< D_{Y^{\partial M}_\alpha}Y^{\partial M}_\beta,X \right> = \delta^{\alpha\beta}_{\partial M} D_{Y^{\partial M}_\alpha}Y^{\partial M}_\beta \cdot \bar{X}+L(h)t^{k-1}+O(t^k).
	\]
\end{proof}

Note that given metric $g$ which $g(0)=\delta$, if we consider the metric at point $p=(\phi(t)\hat{x},-t^{k})$ for $k\ge 2$, we know we can write $g=\delta+ht+O(t^2)$ where $h(e_i,e_j)=\hat{x}_\alpha D_{\alpha}g_{ij}(0)$, which is bounded.
Hence, we can also rewrite the result in Lemma \ref{lem_meanOnBdry} as
	\[
		g^{\alpha\beta}_{\partial M}\left< D_{Y^{\partial M}_\alpha}Y^{\partial M}_\beta,X \right>=\delta^{\alpha\beta}_{\partial M} D_{Y^{\partial M}_\alpha}Y^{\partial M}_\beta \cdot \bar{X}+L(D_{P_0}g(0))t^{k-1}+O(t^{k}).
	\]

For another lemma, let us define a tensor from Christoffel symbols in the following ways. For any three vectors $X=X^ie_i, Y=Y^ie_i,Z=Z^ie_i$ in $\mathbb{R}^3$, we define $\Gamma(X,Y,Z)$ by
\[
	\Gamma(X,Y,Z):=X^iY^jZ^k\Gamma_{ij,k}
\]
where
\[
	\Gamma_{ij,k}=\frac{1}{2}\left( D_ig_{kj}+D_jg_{ik}-D_kg_{ij} \right).
\]

In particular, we know
\[
	\left< \nabla_{X}Y,Z \right> = \left< D_XY,Z \right> +\Gamma(X,Y,Z).
\]

We note that when $t$ is small, the tensor $\Gamma$ is bounded and its leading term depends on the value of $D_i g_{j k}$ at $0$ linearly.
Hence, we can get the following lemma.
\begin{lemma}
	
	We have
	\[
		g^{\alpha\beta}_{\partial M}\Gamma(Y^{\partial M}_\alpha,Y^{\partial M}_\beta,X)=g^{\alpha\beta}_{\Sigma_t} \Gamma(Y^{\Sigma_t}_{\alpha},Y^{\Sigma_t}_\beta,N)+L(D_{P_0}g(0))t^{k-1}+O(t^k).
	\]
	\label{lem_meanWithFlat}
\end{lemma}

\begin{proof}
	From the proof of Proposition \ref{prop:angPlane}, we find
	\begin{equation}
		X-N=\bar{X}-e_3+O(t^{k}).
		\label{eq:pfNormComp}
	\end{equation}
	From the definition of $Y^{\partial M}_\alpha$ and $Y^{\Sigma_t}_\alpha$, we have
	\begin{equation}
		Y^{\partial M}_1=Y^{\Sigma_t}_1-kt^{k-1}e_3,\quad Y^{\partial M}_2=Y^{\Sigma_t}_2.
		\label{eq:pfTanExpand}
	\end{equation}

	Similar as the proof of Lemma \ref{lem_meanOnBdry}, we can find
	\begin{align}
		( g_{\alpha\beta}^{\partial M} )_{\alpha\beta}={}& (g^{\Sigma_t}_{\alpha\beta})_{\alpha\beta}+
		\begin{bmatrix}
			O(t^2) & O(t^3)\\
			O(t^3)& 0 \\
		\end{bmatrix},\nonumber\\
		(g^{\alpha\beta}_{\partial M})_{\alpha\beta}={}&
		(g^{\alpha\beta}_{\Sigma_t})+
		\begin{bmatrix}
			O(t^2) & O(t)\\
			O(t)& O(1) \\
		\end{bmatrix}.
		\label{eq:pfMetTwoPlane}
	\end{align}
	Note that the leading term in $\Gamma(Y^{\partial M}_\alpha,Y^{\partial M}_\beta,X)-\Gamma(Y^{\Sigma_t}_{\alpha},Y^{\Sigma_t}_\beta,N)$ is not related to $h$.
So from \eqref{eq:pfNormComp},\eqref{eq:pfTanExpand}, and \eqref{eq:pfMetTwoPlane} we can get
	\[
		g^{\alpha\beta}_{\partial M}\Gamma(Y^{\partial M}_\alpha,Y^{\partial M}_\beta,X)=g^{\alpha\beta}_{\Sigma_t} \Gamma(Y^{\Sigma_t}_{\alpha},Y^{\Sigma_t}_\beta,N)+L(D_{P_0}g(0))t^{k-1}+O(t^k).
	\]
\end{proof}
Let us go back to the proof of Proposition \ref{prop:meanTaylor}.
Indeed, note that
\begin{align*}
	H_{\partial M}={}&g^{\alpha\beta}_{\partial M} \left< D_{Y^{\partial M}_\alpha}Y^{\partial M}_\beta,X \right>+g^{\alpha\beta}_{\partial M}\Gamma({Y^{\partial M}_\alpha},Y^{\partial M}_\beta,X)\\
	H_{\Sigma_t}={}& g^{\alpha\beta}_{\Sigma_t}\Gamma(Y^{\Sigma_t}_\alpha,Y^{\Sigma_t}_\beta,N)
\end{align*}
Combining with Lemma \ref{lem_meanOnBdry} and Lemma \ref{lem_meanWithFlat}, we have
\[
	H_{t, \partial M}=\bar{H}_{t, \partial M}+\tilde{H}_t+L(D_{P_0}g(0))t^{k-1}+O(t^{k}).
\]
\end{proof}

Note that if we have two points $p_1,p_2 \in \mathbb{R}^3$ such that $|p_1-p_2|\le Ct^k$, then we know the metric $g(p_1)=g(p_2)+O(t^k)$.
Hence, we know $\tilde{H}_t(\hat{x})-H_t(\hat{x})=O(t^k)$ for any $\hat{x} \in D$.
Hence, we have the following corollary.

\begin{corollary}
	\label{cor_meanDelta}
	The mean curvature of $\Sigma_t$ satisfies the following identity.
	\[
	H_{t, \partial M}=\bar{H}_{t, \partial M}+H_t+L(D_{P_0}g(0))t^{k-1}+O(t^k).
\]
\end{corollary}

At last, we need to show the following result.
\begin{lemma}
	\label{lem_2ndFundForm}
	For any $k\ge 1$, we have
	\[
		A_{\partial M}(\eta_t,\eta_t)-\cos \gamma_t A(\nu_t,\nu_t)-\bar{A}_{\partial M}(\bar{\eta}_t,\bar{\eta}_t)=O(t^{k-1}).
	\]
\end{lemma}
\begin{proof}
	From lemma \ref{lem:IIandMean}, we have
	\[
		A_{\partial M}(\eta_t,\eta_t)-\cos \bar{\gamma}_t A(\nu_t,\nu_t)=H_{\partial M}-\cos \bar{\gamma}_t H_t -\sin \bar{\gamma}_t\kappa.
	\]
	We apply this lemma again with $g=\delta$, we have
	\[
		\bar{A}_{\partial M}(\bar{\eta}_t,\bar{\eta}_t)=\bar{H}_{\partial M}-\sin \bar{\gamma}_t\bar{\kappa}.
	\]

	Using the fact $\cos \gamma_t=\cos \bar{\gamma}_t+O(t^{k-1})=1+O(t^{k-1})$, $\sin \bar{\gamma}_t=O(t^{k-1})$, and note $\kappa=\bar{\kappa}+O(1)$, we have
	\begin{align*}
		&A_{\partial M}(\eta_t,\eta_t)-\cos \bar{\gamma}_t A(\nu_t,\nu_t)-\bar{A}_{\partial M}(\bar{\eta}_t,\bar{\eta}_t)\\
		={}& H_{\partial M}-H_t -\bar{H}_{\partial M}-\sin \bar{\gamma}_t \kappa + \sin \bar{\gamma}_t \bar{\kappa}_t=O(t^{k-1})
	\end{align*}
	by Proposition \ref{prop:meanTaylor}.
\end{proof}

\subsection{Linearized operators and foliation near the spherical points}%
\label{sub:linearized}

Based on Proposition \ref{prop:meanTaylor}, we can find functions $f_0,f_1,f_2,\cdots ,f_{k-2}$ defined on $D$ smoothly such that
\[
	H_{t, \partial M}-\bar{H}_{t, \partial M}=F_t+O(t^{k-1})
\]
where $F_t=f_0+f_1 t+\cdots +f_{k-2}t^{k-2}$.
Note that $k\ge 2$, by our definition, we can easily find
\[
	f_0=\lim_{\partial M\ni p\rightarrow 0}H_{\partial M}- \bar{H}_{\partial M}
\]
and hence $f_0$ is indeed a nonnegative constant since $H_{\partial M}-\bar{H}_{\partial M}\ge 0$.



We define the operator $\Phi (t, u) = (\Phi_1 (t, u), \Phi_2 (t, u))$ by
setting
\begin{align}
  \Phi_1 (t, u) & = \frac{H_{t, t^{k + 1} u} - F_t}{t^{k - 1}} - \frac{1}{|D|}
  \int_D \left( \frac{H_{t, t^{k + 1} u} - F_t}{t^{k - 1}} \right),
 \\
  \Phi_2 (t, u) & = \frac{\cos \gamma_{t, t^{k + 1} u} - \cos \bar{\gamma}_{t,
  t^{k + 1} u}}{t^{2 k - 1}} . 
\end{align}
From Corollary \ref{cor_angle}, Proposition \ref{prop:meanTaylor}, and together with \eqref{eq:meank}, \eqref{eq:anglek}, we can find
\begin{align}
\Phi_1(t,u)={}&\frac{1}{(\phi'(0))^2} \Delta u+ \frac{H_t-F_t}{t^{k-1}}+O(t),\label{eq:OperatorL}\\
 \Phi_2(t,u)={}& \frac{k}{(\phi'(0))^2}\frac{\partial u}{\partial \bar{\nu}}+\frac{\cos \gamma_t -\cos \bar{\gamma}_t}{t^{2k-1}}+O(t).\label{eq:OperatorB}
\end{align}
Therefore, we can define $\Phi(0,u)$ by taking limits, that is,
\[\Phi_1(0,u)=\frac{1}{(\phi'(0))^2}\Delta u+ \lim_{t\rightarrow 0} \frac{H_{t}-F_t}{t^{k-1}}\] and \[\Phi_2(0,u)=\frac{k}{(\phi'(0))^2}\frac{\partial u}{\partial \bar{\nu}}+\lim_{t\rightarrow 0} \frac{\cos \gamma_t-\cos \bar{\gamma}_t}{t^{2k-1}},\] and we can make sure $\Phi(t,u)$ is a $C^{1}$ operator defined for $t\in[0,\varepsilon)$.
\begin{proposition}
	\label{prop:solution}
	For any $t \in [0,\varepsilon)$ for $\varepsilon$ small enough, there exists a function $u(\cdot,t)$ such that $\Phi(t,u(\cdot,t))=0$.
\end{proposition}

\begin{proof}
	The proof is similar to the proof for Proposition \ref{prop cmc capillary at zero}. We outline the key steps.
	We write $f=\lim_{t\rightarrow 0} \frac{H_t-F_t}{t^{k-1}}$, $g=\lim_{t\rightarrow 0} \frac{\cos \gamma_t-\cos \bar{\gamma}_t}{t^{2k-1}}$.
At first, let us show we can find $u_0$ such that $\Phi(0,u_0)=0$.
Indeed, we just need to find a unique minimiser $u_0$ of the following energy,
	\[
		E(u):=\int_{ D}\left( \frac{\left|\nabla u\right|^2}{2(\phi'(0))^2}-fu \right) + \int_{ \partial D} \frac{kgu}{(\phi'(0))^2}
	\]
	on the space $\int_{ D} u=0$.
	Then, $u_0$ is the function such that $\Phi(0,u_0)=0$.
	Note that $D\Phi_{(0,u_0)}(0,v)=\frac{1}{(\phi'(0))^2}\left( \Delta v- \frac{1}{|D|}\int_{ D} \Delta v, k \frac{\partial u}{\partial  \bar{\nu}} \right)$.
	Hence, we can apply the implicit function theorem to conclude we can find a $C^1$ map $t\rightarrow u(\cdot,t)$ such that $\Phi(t,u(\cdot,t))=0$.
\end{proof}

\subsection{Behavior of the mean curvature of the leaf}%
\label{sub:cmc_foliations}

In this part, we want to show that each leaf of the foliation constructed in Proposition \ref{prop:solution} has non-negative mean curvature for $k=2$.

Based on Proposition \ref{prop:solution}, we know we can find $u_t(\cdot)=u(t,\cdot) \in C^{2,\alpha}(D)\cap C^{1,\alpha}(\bar{D})$ for each $t \in (0,\varepsilon)$ such that the mean curvature of $\Sigma_{t,t^{k+1}u_t}$ is given by $t^{k-1}\lambda_t+F_t$.
Here, we write $\lambda_t=\frac{H_{t,t^{k+1}u_t}}{t^{k-1}}$.
At first, we can obtain the limit behavior of $\lambda_t$ similar to Lemma \ref{limit of mean}.

\begin{lemma}
	For any $k\ge 2$, if $F_t=0$, we have
	\[
		\lim_{t\rightarrow 0} \lambda_t\ge 0.
	\]
	\label{lem:limLam}
\end{lemma}
\begin{remark}
	Note that Lemma \ref{lem:limLam} also holds for $k=1$ from Lemma \ref{limit of mean} since $F_t=0$ if $k=1$.
\end{remark}

\begin{proof}
	At first, we integrate $\frac{\lambda_t}{t^{k-1}}=\frac{1}{(\phi'(0))^2}\Delta u_t+\frac{H_t}{t^{k-1}}+O(t)$ on $D$ and use \eqref{eq:OperatorB}, we have
	\begin{align}
		\lambda_t |D|={} & \frac{1}{(\psi'(0))^2}\int_{ D} \Delta u_t + \frac{H_t}{t^{k-1}}+O(t)\nonumber \\
		={} & \frac{1}{(\phi'(0))^2} \int_{ D} \frac{H_t}{t^{k-1}}-\frac{1}{k}\int_{ \partial D} \frac{\cos \gamma_t-\cos \bar{\gamma}_t}{t^{2k-1}}+O(t).
		\label{eq:pfLambdaExp}
	\end{align}

	Based on Proposition \ref{prop:meanTaylor} and Corollary \ref{cor_angle}, we write $L=\lim_{t\rightarrow 0} \frac{H_{t, \partial M}-\bar{H}_{\partial M}-H_t}{t^{k-1}}$, $B=\lim_{t\rightarrow 0} \frac{\cos \gamma_t -\cos \bar{\gamma}_t}{t^{2k-1}}$.
	Note that $L$ ($B$ respectively) is a function on $D$ (on $\partial D$ respectively) and they rely on $D_{P_0}g(0)$ linearly.

	Now if we choose $\hat{x} \in D$ and consider the metric at $p_{\hat{x}}:=(\phi(t)\hat{x},-t^{k})$.
	We can write this metric as
	\[
		g_{ij}(p_{\hat{x}})=\delta_{ij}+\phi(t)\hat{x}_\alpha D_{\alpha}g_{ij}(0)+O(t^2).
	\]
	But note that the metric at $p_{-\hat{x}}$ can also be written as
	\[
		g_{ij}(p_{-\hat{x}})=\delta_{ij}-\phi(t)\hat{x}_\alpha D_{\alpha}g_{ij}(0)+O(t^2).
	\]
	In particular, $g_{ij}(p_{\hat{x}})+g_{ij}(p_{-\hat{x}})-2\delta_{ij}=O(t^2)$.
	Hence, $L(\hat{x})=-L(-\hat{x})$ since $L$ relies on $D_{P_0}g(0)$ linearly.

	In particular, this tells us
	\[
		\int_{ D} L =0.
	\]

	Similarly, we have
	\[
		\int_{ \partial D} B=0.
	\]

	Now from \eqref{eq:pfLambdaExp}, we have
	\begin{align*}
		\lambda_t|D|={} & \frac{1}{(\phi'(0))^2}\int_{ D} \frac{H_{t, \partial M}-\bar{H}_{t, \partial M}}{t^{k-1}}+L + O(t)-\frac{1}{k}\int_{ \partial  D} B+O(t)  \\
		={} & \frac{1}{(\phi'(0))^2}\int_{ D} \frac{H_{t, \partial M}-\bar{H}_{t, \partial M}}{t^{k-1}}+O(t)\ge O(t)
	\end{align*}
	since $H_{\partial M}\ge \bar{H}_{\partial M}$ on $\partial M$.
\end{proof}

Now, we are ready to show that each leaf has nonnegative mean curvature.

\begin{proposition}
	\label{prop:positiveCMC}
	For $k=2$, we can construct a surface $\Sigma$ near $p$ such that the mean curvature of $\Sigma$ is nonnegative and it has prescribed contact angle $\bar{\gamma}$ along $\partial \Sigma$.
\end{proposition}
\begin{proof}
	Note that when $k=2$, $F_t=f_0$ is a constant and it can be computed by	
	\[
		f_0=\lim_{\partial \ni p \rightarrow 0} H_{\partial M}-\bar{H}_{\partial M}\ge 0.
	\]
	If $f_0>0$, we can choose $t$ small such that $\Sigma_{t,t^3u_t}$ has positive mean curvature with prescribed contact angle $\bar{\gamma}$ since $\lim_{t\rightarrow 0} H_{t,t^3u_t}= f_0$.

	From now on, we work on the case of $F_t=f_0=0$.

	Again, if $\lim_{t\rightarrow 0} \lambda_t>0$, we can also choose $t \in (0,\varepsilon)$ small enough such that
	$\frac{H_{t,t^{3}u_t}}{t}>0$ and $\Sigma_{t,t^{3}u_t}$ is the desired surface.

	We only have to handle the case $\lim_{t\rightarrow 0} \lambda_t=0$.

	Let $\psi:D \times (0,\varepsilon)\rightarrow M$ pararameterize the foliation $\Sigma_{t,t^3u_t}$, $Y=\frac{\partial \psi}{\partial t}$, $v_t=\left< Y,N_t \right> $.

	We denote $h(t):=H_{t,t^3u_t}$.
	Note that by the construction of $\Sigma_{t,t^3u_t}$, we have $v_t=t+O(t^2)$ and hence $v_t>0$ for all $t$ sufficent small.
	Hence, following the proof of Theorem \ref{H ode} (see also the proof of Lemma \ref{nonnegativity of mean curvature of every leaf}), we have
	\begin{equation}
		\frac{\mathrm{d}h(t)}{\mathrm{d}t}+\Psi(t)h(t)\ge 0,
		\label{eq:pfODELambda}
	\end{equation} 
	where
	\[
		\Psi(t)=\left( \int_{ \Sigma_{t,t^3u_t}} \frac{1}{v_t} \right)^{-1} \int_{ \partial \Sigma_{t,t^3u_t}} \cos \bar{\gamma}_t,
	\]
    and $v_t:=-\left<Y_{t,t^3u_t},N_{t,t^3u_t}\right>$, $Y_{t,t^3u_t}$ is the variational vector field associated with the foliation $\Sigma_{t,t^3u_t}$ and $N_{t,t^3u_t}$ the unit normal vector field along $\Sigma_{t,t^3u_t}$ and pointing upward.
    Hence, we can see that
    \[
    v_t=-\left<Y_{t,t^3u_t},N_{t,t^3u_t}\right>=\left<2te_3,e_3\right>+O(t^2)=2t+O(t^2)
    \]
    by the definition of $\Sigma_t$ and $\Sigma_{t,t^3u_t}$.
	Hence, we have
	\begin{align*}
		\int_{ \Sigma_{t,t^3u_t}} \frac{1}{v_t}={}&\frac{\pi \phi^2(t)}{2t}+O(t^{2})=\pi (\phi'(0))^2t+O(t^{2}),\\
		\int_{ \partial \Sigma_{t,t^3u_t}} \cot \bar{\gamma}={}&
		\frac{2\pi\phi'(0) \phi(0)}{2t}+O(1)=\pi (\phi'(0))^2+O(t).
	\end{align*}
	where we have used Corollary \ref{cor_angle}.

	Therefore,
	\[
		\Psi(t)=2t^{-1}+C_1(t)
	\]
	where $C_1(t)$ is a continuous function of order $O(1)$.
	Hence, from \eqref{eq:pfODELambda}, we find $h(t)$ satisfies the following ordinary differential inequality
	\[
		\frac{\mathrm{d}}{\mathrm{d}t}\left[ \exp\left( \int_{ 0} ^t C_1(s)\mathrm{d}s \right)t^2 h(t) \right]\ge 0.
	\]
        Therefore, have $h(t)\ge 0$ for every $t \in (0,\varepsilon)$. We pick such a small $t$, then $\Sigma_{t,t^{3}u_t}$ is a surface with nonnegative mean curvature and prescribed contact angle $\bar{\gamma}$.
\end{proof}

\subsection{Non-Euclidean metric at spherical points}%
\label{sub:non_euclideanSphere}

In this section and next section, we will only assume 
\begin{equation}
    \sigma\ge \bar{\sigma} \quad \text{ at the spherical point }O.
    \label{eq:assumptionMetric}
\end{equation}
Similar to the conical case \eqref{eq:lessLengthAtConical}, it is also possible that we can find a non-zero vector $v \in T_OM$ such that $|v|_g<|v|$.
Note that we still need to assume $\sigma\ge \bar{\sigma}$ for any $p \in \partial M$ for Theorem \ref{main}.

We need to construct a new $\Sigma_t$, similar to the $\Sigma_t$ defined in the previous section.
We denote $\bar{\Sigma}_t:=M\cap \left\{ x_3=-t^2 \right\}$, the $\Sigma_t$ we defined in the previous section.

We assume $g=g_0+th+O(t^2)$ where $g_0=g(0)$ is a constant metric and $h$ is a bounded symmetric 2-tensor.
By rotating the coordinate if necessary, we assume $g_0=
\begin{bmatrix}
	a_{11} & 0 & a_{13}\\
	0 & a_{22} & a_{23} \\
	a_{13} & a_{23} & a_{33} \\
\end{bmatrix}$.
We write $a^{ij}$ as the inverse of metric $g_0$.
To simplify the computation in this section, we use the following notations.
Recall that $e_i=\frac{\partial }{\partial x^i}$ is the standard orthonormal basis of $\mathbb{R}^3$.
Since we are interested in the constant metric $g_0$ in this subsection, we write $\{e^i\}_{i=1}^3$ as the dual basis of $\left\{ e_i \right\}_{i=1}^3$ under metric $g_0$.
In other words, $e^i=a^{ij}e_j$ and $\left< e^i,e_j \right>_{g_0}=\delta^i_j$.

By our assumption of the metric \eqref{eq:assumptionMetric}, we know $a_{11}\ge 1$ and $a_{22}\ge 1$.

This time we assume $\partial M$ is given by
\[
	(x,-\varphi(\left|x\right|^2))
\]
with $\varphi(0)=0$ and $\varphi'(0)>0$.
We define
\begin{equation}
	b_1=\sqrt{a_{11}a^{33}}=\frac{a_{11}\sqrt{a_{22}}}{\sqrt{\mathrm{det}(g_0)}},\text{ }
	b_2=\sqrt{a_{22}a^{33}}=\frac{a_{22}\sqrt{a_{11}}}{\sqrt{\mathrm{det}(g_0)}}.
\end{equation}

For any $t>0,s\ge 0$,
we consider the function $G_{s,t}$ defined by
\[
	G_{s,t}(x_1,x_2)=\varphi'(0)x_1^2(b_1-1+s)+\varphi'(0)x_2^2(b_2-1+s)-t^2.
\]

We define $\Sigma_{s,t}$ be the intersection of graph of $G_{s,t}$ with the interior of $M$.
That means
\[
	\Sigma_{s,t}:=\left\{(x,G_{s,t}(x)):x \in D \text{ and }G_{s,t}(x)\le -\varphi(|x|^2) \right\}.
\]

If $s=0$, we define $\Sigma_t:=\Sigma_{0,t}$.
\begin{remark}
	Indeed, our construction of foliation is based on $\Sigma_t$ only, see Subsection \ref{sub:construction_of_foliation}.
	But that proof only works when $a_{11}=a_{22}=1$.
	Otherwise, we need to use $\Sigma_{s,t}$ to directly give a good barrier. See Proposition \ref{prop_sphereStrictBarrier}.
\end{remark}

By the implicit function theorem, we can write
\[
	\Sigma_{s,t}:=\left\{ (\phi_1(s,t,\hat{x}), \phi_2(s,t,\hat{x}),G_{s,t}(\phi_1(s,t,\hat{x}), \phi_2(s,t,\hat{x})):\hat{x} \in D \right\}
\]
where $\varphi_1,\varphi_2$ are all smooth functions such that
\[
	\phi_1(s,t,\hat{x})=\frac{t \hat{x}_1}{\sqrt{\varphi'(0)(b_1+s)}}+O(t^3),\quad 
	\phi_2(s,t,\hat{x})=\frac{t \hat{x}_2}{\sqrt{\varphi'(0)(b_2+s)}}+O(t^3).
\]

Hence, it is better to use an ellipse instead of a disk to parameterize $\Sigma_{s,t}$.
We define a set $E_s\subset \mathbb{R}^2$ by
\[
	E_s:=\left\{ \hat{x} \in \mathbb{R}^2: (b_1+s)\hat{x}_1^2+ (b_2+s)\hat{x}_2^2<1 \right\}.
\]
By the definition of $E_s$, it is easy to notice that $\frac{\sqrt{\varphi'(0)}}{t}\Sigma_{s,t}\to E_s$ as sets when we take $t\to 0$.

We also use $E=E_0$ for simplicity.

To simplify our notations, we assume
\[
	\Phi_{s,t}(\hat{x})=(\phi_1(s,t,\sqrt{b_1+s}\hat{x}), \phi_2(s,t,\sqrt{b_2+s}\hat{x}))\quad \text{ for }\hat{x} \in E_s.
\]
We also view $\Sigma_{s,t}$ as a map $E_s\rightarrow \Sigma_{s,t}$ such that
\begin{align}
    \Sigma_{s,t}(\hat{x}):={}&(\Phi_{s,t}(\hat{x}),G_{s,t}(\Phi_{s,t}(\hat{x})))\nonumber \\
    ={}&\left(\frac{t\hat{x}}{\sqrt{\varphi'(0)}}+O(t^3),t^2\hat{x}_1^2(b_1+s-1)+t^2\hat{x}_2^2(b_2+s-1)-t^2+O(t^4)\right).
	\label{eq:defSigmaST}
\end{align}
Note that by our choice of $\Phi_{s,t}(\hat{x})$, we have
$\Phi_{s,t}(\hat{x})=\frac{1}{\sqrt{\varphi'(0)}}t\hat{x}+O(t^3)$.

\subsubsection{Asympotoic behavior of angles near the spherical point}%
\label{ssub:asympotoic_behavior_of_angles_near_the_spherical_point}
We denote $\gamma_{s,t}(\hat{x})$ the angle between $\Sigma_{s,t}$ and $\partial M$ under metric $g$, $\gamma^{g_0}_{s,t}(\hat{x})$ the angle between $\Sigma_{s,t}$ and $\partial M$ under metric $g_0$ at point $\Sigma_{s,t}(\hat{x})$ for any $\hat{x} \in \partial E_s$.

%
%
Recall that for any $\hat{x} \in\partial E_s$, $\bar{\gamma}_{s,t}(\hat{x})$ is defined as the value of $\bar{\gamma}$ at point $\Sigma_{s,t}(\hat{x})$.
If we assume $s=0$, we write $\gamma_t:=\gamma_{0,t}, \gamma_t^{g_0}:=\gamma_{s,t}^{g_0}$, and $\bar{\gamma}_t:=\bar{\gamma}_{s,t}$.

Now, we need to find the relation of $\gamma_{s,t}^{g_0}$ and $\bar{\gamma}_{s,t}$.
This can be summarized as the following proposition.
\begin{proposition}
	\label{prop:angleg0}
	We have the following asymptotic behavior for $\gamma_{s,t}^{g_0}$ and $\bar{\gamma}_{s,t}$:
	\begin{align*}
		\cos \gamma^{g_0}_{s,t}(\hat{x})={} & \cos \bar{\gamma}_{s,t}(\hat{x})+2\varphi'(0)\hat{x}_\alpha^2 t^2\left( 1-\frac{(b_\alpha+s)^2}{a_{\alpha\alpha}a^{33}} \right)+A(\hat{x})t^3+O(t^4),
	\end{align*}
	for any $\hat{x} \in \partial E_s$.
	Here, we use $A(\hat{x})$ to denote the bounded term not related to $t$, and it is also odd symmetric with respect to $\hat{x}$, that is, $A(\hat{x})=-A(-\hat{x})$.
\end{proposition}


\begin{proof}
	For any $\hat{x} \in \partial E_s$, using \eqref{eq:defSigmaST},
	we can find that a basis of tangent space of $\Sigma_{s,t}$ is given by
	\[
		X_\alpha=\frac{e_\alpha}{\sqrt{\varphi'(0)}}+2t(b_\alpha+s-1)\hat{x}_\alpha e_3+\sum_{\beta=1 }^{2}O(t^2)e_\beta+O(t^3)e_3.
	\]

	Hence, we can choose the normal vector $u$ of $\Sigma_{s,t}$ at point $\Sigma_{s,t}(\hat{x})$ under metric $g_0$ by
	\[
		u=-\frac{(2t \hat{x}_\alpha (b_\alpha+s-1)+O(t^3))e^\alpha}{\sqrt{\varphi'(0)}}+\frac{1+O(t^2)}{\varphi'(0)}e^3.
	\]

	Similarly, we choose the normal vector $v$ of $\partial M$ at $\Sigma_{s,t}(\hat{x})$ under metric $g_0$ as,
	\[
		v=\frac{(2t+O(t^3))\hat{x}_\alpha e^\alpha}{\sqrt{\varphi'(0)}}+\frac{1}{\varphi'(0)}e^3.
	\]

	Hence
	\[
		\cos\measuredangle _{g_0}(\Sigma_{s,t},\partial M)=\frac{g_0(u,v)}{\sqrt{g_0(u,u)g_0(v,v)}}\quad \text{ at point }\Sigma_{s,t}(\hat{x}).
	\]

	Write $w=u-v=-\frac{2t\hat{x}_\alpha (b_\alpha+s) e^\alpha}{\sqrt{\varphi'(0)}}+O(t^2)$. 
	We use Lemma \ref{lem_angleG} with $g_0$ in place of $g$, $g_0$ in place of $\delta$, and $1$ in place of $k$ (see Remark \ref{lem_angleG0}), and note that $u=\frac{1}{\varphi'(0)}e^3+tA(\hat{x})+O(t^2)$, $v=\frac{1}{\varphi'(0)}e^3+tA(\hat{x})+O(t^2)$.
	Then we have
	\begin{align*}
	&\cos\measuredangle _{g_0}(\Sigma_{s,t},\partial M)\\
		={} & 1-\frac{g_0(w,w)g_0(u,v)-g_0(w,u)g_0(w,v)}{2g_0(u,v)^2} +O(t^4)\\
		={} & 1-\frac{[\varphi'(0)]^2g_0(w,w)g_0(e^3,e^3)-[\varphi'(0)]^2g_0(e^3,w)^2}{2g(e^3,e^3)^2}+A(\hat{x})t^3+O(t^4)\\
		={}& 1-\frac{2\varphi'(0)(b_\alpha+s)(b_\beta+s) \hat{x}_\alpha \hat{x}_\beta (a^{\alpha\beta}a^{33}-a^{3\alpha}a^{3\beta})}{(a^{33})^2}t^2+A(\hat{x})t^3+O(t^4) \\
		={}& 1-\frac{2\varphi'(0)(b_\alpha+s)^2\hat{x}_\alpha^2(a_{11}a_{22})}{(\mathrm{det}g_0)a_{\alpha\alpha}(a^{33})^2}t^2+A(\hat{x})t^3+O(t^4) \\
		={}& 1-\frac{2\varphi'(0)(b_\alpha+s)^2\hat{x}_\alpha^2}{a_{\alpha\alpha}a^{33}}t^2+A(\hat{x})t^3+O(t^4)
	\end{align*}
	Here, we have used $a^{\alpha\beta}a^{33}-a^{3\alpha}a^{3\beta}=\frac{a_{11}a_{22}}{a_{\alpha\alpha}\mathrm{det}g_0}\delta_{\alpha\beta}$ and $(\mathrm{det}g_0)a^{33}=a_{11}a_{22}$.

	Similarly, if we choose $g_0=\delta$ and repeat the previous computation, we find
	\begin{equation}
		\cos \bar{\gamma}_{s,t}(\hat{x})=1-2\varphi'(0)\hat{x}_\alpha^2t^2+A(\hat{x})t^3+O(t^4).
		\label{eq:pfGammaBarDelta}
	\end{equation}
	
	Hence,
	\[
		\cos \gamma_{s,t}^{g_0}-\cos \bar{\gamma}_{s,t}=2\varphi'(0)\hat{x}_\alpha^2 t^2\left( 1-\frac{(b_\alpha+s)^2}{a_{\alpha\alpha}a^{33}} \right)+A(\hat{x})t^3+O(t^4),
	\]
	for any $\hat{x} \in \partial E_s$.
\end{proof}

\begin{corollary}
	\label{cor_angleGtoD0}
	Suppose the metric $g$ can be written as $g=g_0+t h+O(t^2)$ where $g_0$ is a constant metric and $h$ is a bounded symmetric two-tensor.
	Then, we have the following formula for the behavior of angle $\gamma_{s,t}$,
	\begin{align*}
		\cos \gamma_{s,t}(\hat{x})={}&\cos \bar{\gamma}_{s,t}(\hat{x})+2\varphi'(0)\hat{x}_\alpha^2 t^2\left( 1-\frac{(b_\alpha+s)^2}{a_{\alpha\alpha}a^{33}} \right)+A(\hat{x})t^3+L(h)t^3+O(t^4),
	\end{align*}
    for any $\hat{x} \in E_s$.
	Here, $A(\hat{x})$ is a bounded term (not related to $t$ and $h$) which is also odd symmetric with respect to $\hat{x}$, $L(h)$ is a bounded term (not related to $t$) relying on $h$ linearly.
\end{corollary}
\begin{proof}
	Note that Corollary \ref{cor_angle} is still true with $g_0$ in place of $\delta$, $\gamma_{s,t}^{g_0}$ in place of $\bar{\gamma}_{s,t}$, $2$ in place of $k$ for a fixed $s\ge 0$ by the exact same proof.
	Together with Proposition \ref{prop:angleg0}, we can finish the proof.
\end{proof}

If we take $s=0$ in Corollary \ref{cor_angleGtoD0} together with the trigonometric identity, we can get the following corollary.
\begin{corollary}
	\label{cor_angleGtoD}
	Suppose the conditions in Corollary \ref{cor_angleGtoD0} hold.
	Then, we have the following formulas for the behavior of angle $\gamma_t$.
	\begin{align}
	\label{eq:corAngleCos}
		\cos \gamma_{t}(\hat{x})={}&\cos \bar{\gamma}_{t}(\hat{x})+A(\hat{x})t^3+L(h)t^3+O(t^4),\\
		\sin \gamma_{t}(\hat{x})={}&\sin \bar{\gamma}_t(\hat{x})+O(t^2)=2\sqrt{\varphi'(0)}|\hat{x}|t+O(t^2),\label{eq:corAngleSin}
	\end{align}
	for any $\hat{x} \in \partial E$.
\end{corollary}

\begin{proof}
	The equation \eqref{eq:corAngleCos} is a direct consequence of Corollary \ref{cor_angleGtoD0} by taking $s=0$.

	For \eqref{eq:corAngleSin}, we have
	\begin{align*}
		\sin \gamma_t(\hat{x})={}&\sqrt{1-\cos ^2\gamma_t(\hat{x})}=\sqrt{1-\cos ^2\bar{\gamma}_t(\hat{x})}+O(t^2)=\sin \bar{\gamma}_t(\hat{x})+O(t^2).
	\end{align*}
	On the other hand, using \eqref{eq:pfGammaBarDelta}, we have
	\[
		\sin \gamma_t(\hat{x})=2\sqrt{\varphi'(0)}|\hat{x}|t+O(t^2),
	\]
	for any $\hat{x} \in \partial E$.
\end{proof}

\begin{remark}
\label{rmk:cotG0}
Note that $\cos \gamma_t(\hat{x})=1-2\varphi'(0)|\hat{x}|^2t^2+O(t^3)$ and hence, $\cot\gamma_t(\hat{x})=\frac{1}{2\sqrt{\varphi'(0)}|\hat{x}|t}+O(t^2)$.
\end{remark}
Another corollary of Corollary \ref{cor_angleGtoD0} is, we can choose a positive $s$ such that $\gamma_{s,t}>\bar{\gamma}_{s,t}$ if $t$ is sufficient small.
We summarize it here.

\begin{corollary}
	\label{cor_GreaterAngleG0}
	Suppose the conditions in Corollary \eqref{eq:thmLimU} holds.
	Then, for any $s>0$, we can find $t_0>0$ (might rely on $s$) such that for any $t<t_0$, we have
	\begin{equation}
		\gamma_{s,t}(\hat{x})>\bar{\gamma}_{s,t}(\hat{x})
		\label{eq:corGreaterAngle}
	\end{equation}
	for any $\hat{x} \in \partial E_s$.
\end{corollary}
\begin{proof}
	By Corollary \eqref{eq:thmLimU}, we have
	\[
		\cos \gamma_{s,t}(\hat{x})\le \cos \bar{\gamma}_{s,t}(\hat{x})-\frac{2\varphi'(0)\hat{x}_\alpha^2s^2t^2}{a_{\alpha\alpha}a^{33}}+Ct^3,
	\]
	where $C$ is a positive constant not relying on $t$.
	Hence, we can take $t_0=\frac{\varphi'(0)\hat{x}_\alpha^2s^2}{C a_{\alpha\alpha}a^{33}}$ and find $\cos \gamma_{s,t}(\hat{x})<\cos \bar{\gamma}_{s,t}(\hat{x})$ for any $\hat{x} \in \partial E_s$ and any $t<t_0$, which is equivalent to \eqref{eq:corGreaterAngle}.
\end{proof}

\subsubsection{Asympotoic behavior of mean curvatures}%
\label{ssub:asympotoic_behavior_of_mean_curvatures}

We define the following functions.
\begin{align*}
	H_{s,t}^{g}(\hat{x}):={}&\text{Mean curvature of $\Sigma_{s,t}$ at $\Sigma_{s,t}(\hat{x})$ under metric $g$},\\
	H_{s,t,\partial M}^{g}(\hat{x}):={}& \text{Mean curvature of $\partial M$ at $(\Phi_{s,t}(\hat{x}),-\varphi(|\Phi_{s,t}(\hat{x})|^2))$ under metric $g$}.
\end{align*}

We also write $H^g_{t}=H^g_{0,t}$ and $H^g_{t,\partial M}=H^g_{0,t,\partial M}$.
In particular, we only focus on the behaviors of mean curvature at $s=0$.

\begin{proposition}
	\label{prop:Hg0}
	The difference of the mean curvature $H^{g_0}_t$ and $H^{g_0}_{t,\partial M}$ can be estimated by the following
\[
	H_t^{g_0}(\hat{x})-H_{t,\partial M}^{g_0}(\hat{x})=-\frac{2\varphi'(0)}{\sqrt{a_{11}}}-\frac{2\varphi'(0)}{\sqrt{a_{22}}}+tL(\hat{x}) +O(t^2),
\]
where $L(\hat{x})$ is a bounded term depending on $\hat{x}$ linearly.
\end{proposition}

We need a lemma to prove this proposition.

\begin{lemma}
	\label{lem_meang0}
	Suppose $\Sigma_t$ is a family of surfaces indexed by $t$ and given by
	\[
          \Sigma_t=\left\{ (At x+O(t^3),a_1 t^2 x_1^2+a_2t^2x_2^2+O(t^4)):x=(x_1,x_2)\in R
	\right\}	\]
	where $R$ is a open set in $\mathbb{R}^2$ and $A$ is a positive number.
	Then the mean curvature $H_{\Sigma_t}$ under the constant metric $g_0$ is given by
	\[
		H_{\Sigma_t}(x)=-\frac{1}{A^2\sqrt{a^{33}}}
		\left( \frac{2a_1}{a_{11}}+\frac{2a_2}{a_{22}} \right)+tL(x)+O(t^2)\quad  \text{ for }x \in R
	\]
        for $t$ small.
\end{lemma}
\begin{proof}
	A basis of the tangent space of $\Sigma_t$ is given by
	\[
		X_\alpha=Ae_\alpha+2a_\alpha t x_\alpha e_3+\sum_{\beta=1}^2O(t^2)e_\beta+O(t^3)e_3.
	\]
	One of the normal vectors $\tilde{\nu}$ at $(x_1,x_2)$ satisfies $\tilde{\nu}:=(-2ta_\alpha x_\alpha+O(t^3)) e^\alpha+(A+O(t^2))e^3$.
	We write $P$ as the space spanned by $\left\{ X_\alpha \right\}_{\alpha=1}^2$ and $g_{\alpha\beta}^P:=g_0(X_\alpha,X_\beta)$.
	Hence
	\[
		H^{g_0}_{\Sigma_t}=-g^{\alpha\beta}_P g_0(\nabla_{X_\alpha}^{g_0}{X_\beta},\tilde{\nu}/|\tilde{\nu}|_{g_0})
	\]
	where $g^{\alpha\beta}_P$ is the inverse matrix of $g^{P}_{\alpha\beta}$ and $g_{\alpha\beta}^P$ is given by
	\[
		g_{\alpha\beta}^P=
		A^2\begin{bmatrix}
			a_{11} & 0 \\
			0 & a_{22} \\
		\end{bmatrix}+t
		A\begin{bmatrix}
			4 a_1 x_1 a_{13} & 2a_1x_2 a_{13}+2a_2x_1a_{23}\\
			2a_1x_2 a_{13}+2a_2x_1a_{23}& 4a_2x_2a_{23} \\
		\end{bmatrix}+O(t^2).
	\]
	For simplicity, we write $L(x)$ as the term depending on $x$ linearly and denote $g^P_{\alpha\beta}$,
	\[
		g^P_{\alpha\beta}=
		A^2\begin{bmatrix}
			a_{11} & 0\\
			0& a_{22} \\
		\end{bmatrix}+t L(x)+O(t^2),\quad 
		g^{\alpha\beta}_P=
		\frac{1}{A^2}\begin{bmatrix}
			\frac{1}{a_{11}}& 0\\
			0& \frac{1}{a_{22}} \\
		\end{bmatrix}-tL(x)+O(t^2).
	\]
	On the other hand, 
	\[
		\nabla_{X_\alpha}^{g_0}X_\beta=A(2a_\alpha\delta_{\alpha\beta}+O(t^2))e_3+\sum_{\beta=1 }^{2}O(t)e_\beta,
	\]
	hence
	\begin{align*}
		g^{\alpha\beta}_P g_0(\nabla_{X_\alpha}^{g_0}X_\beta,\tilde{\nu})={}&\frac{1}{Aa_{11}}g_0(2a_1e_3,e^3)+\frac{1}{Aa_{22}}g_0(2a_2e_3,e^3)+tL(x)+O(t^2)\\
		={}& \frac{2a_1}{Aa_{11}}+\frac{2a_2}{Aa_{22}}+tL(x) +O(t^2).
	\end{align*}
	Note that
	\[
		|\tilde{v}|_{g_0}^2=A^2g^{33}+tL(x)+O(t^2).
	\]
	Hence
	\[
		H_{\Sigma_t}=-\frac{1}{A^2\sqrt{g^{33}}}\left( \frac{2a_1}{a_{11}}+\frac{2a_2}{a_{22}} \right)+tL(x)+O(t^2).
	\]
\end{proof}

\begin{proof}
	[Proof of Proposition \ref{prop:Hg0}]
	Recall that we can write
	\begin{align*}
		\Sigma_t(\hat{x})={}&\left\{ \left(\frac{t \hat{x}}{\sqrt{\varphi'(0)}}+O(t^3),\hat{x}_1^2(b_1-1)+\hat{x}_2^2(b_2-1)+O(t^4)\right):\hat{x} \in E \right\}.
	\end{align*}
	Then from Lemma \ref{lem_meang0}, we have,
	\begin{align*}
				H_{t}^{g_0}(\hat{x})={}&-\frac{\varphi'(0)}{\sqrt{a^{33}}}\left( \frac{2(b_1-1)}{a_{11}}+\frac{2(b_2-1)}{a_{22}} \right)+tL(\hat{x})+O(t^2).
	\end{align*}

	We also note that
	\[
		\partial M=\left\{ \left( \frac{x}{\sqrt{\varphi'(0)}},-|x|^2 +O(|x|^4)\right): \text{ for }x \text{ near }O\right\},
	\]
	then, we have
	\begin{align*}
		H_{t,\partial M}^{g_0}(\hat{x})={}&
		\frac{\varphi'(0)}{\sqrt{a^{33}}}\left(\frac{2}{a_{11}}+\frac{2}{a_{22}} \right)+tL(\hat{x})+O(t^2)
	\end{align*}
	Hence, based on the definition of $b_1,b_2$, we have
	\begin{align*}
		H_t^{g_0}(\hat{x})-H_{t,\partial M}^{g_0}(\hat{x})={} &-\frac{\varphi'(0)}{\sqrt{a^{33}}}\left( \frac{2b_1}{a_{11}}+\frac{2b_2}{a_{22}} \right) +tL(\hat{x})+O(t^2)\\
		={}& -2\varphi'(0)\left( \frac{1}{\sqrt{a_{11}}}+\frac{1}{\sqrt{a_{22}}} \right)+tL(\hat{x})+O(t^2),
	\end{align*}
	for any $\hat{x}\in E$.
	This is exactly what we want.
\end{proof}

\begin{corollary}
	\label{cor_meanGDelta}
	Suppose the metric $g$ can be written as $g=g_0+t h+O(t^2)$ where $g_0$ is a constant metric and $h$ is a bounded symmetric two-tensor. 
	Then, we have the following formula for the behavior of mean curvature
	\begin{equation}
		H_t(\hat{x})=H_{t,\partial M}(\hat{x})-\bar{H}_{t,\partial M}(\hat{x})+2\varphi'(0)\left( 2-\frac{1}{\sqrt{a_{11}}}-\frac{1}{\sqrt{a_{22}}} \right)+tL(\hat{x})+O(t^2),
		\label{eq:corMean}
	\end{equation}
	for any $\hat{x} \in E$.
	Here, we denote $H_t=H^{g}_t,H_{t,\partial M}=H_{t,\partial M}^g,\bar{H}_{t,\partial M}=H^\delta_{t,\partial M}$.
\end{corollary}
\begin{proof}
	Note that we can find the following formula if we follow the argument for Corollary \ref{cor_meanDelta},
	\[
		H_{t,\partial M}-H_{t,\partial M}^{g_0}=H_t-H_t^{g_0}+L(D_{P_0}g(0))t+O(t^2).
	\]
	We can also rewrite it in the following form,
	\[
		H_{t,\partial M}(\hat{x})-H^{g_0}_{t,\partial M}(\hat{x})=H_t(\hat{x})-H_t^{g_0}(\hat{x})+tL(\hat{x})+O(t^2).
	\]
	Together with Proposition \ref{prop:Hg0} and note that $\bar{H}_{t,\partial M}=4\varphi'(0)+O(t^2)$, we can show the mean curvature $H_t(\hat{x})$ can be computed by \eqref{eq:corMean}.
\end{proof}

\subsubsection{Construction of barriers when $H_0>0$}%
\label{ssub:construction_of_barriers_when_h_0_0_}
Now, we are ready to give the proof of the existence of a good barrier in the case of $H_0:=\lim_{t\rightarrow 0} H_t(\hat{x})>0$.
Note that by Corollary \ref{cor_meanGDelta}, we know $H_0$ is indeed a constant and it can be computed by
\begin{equation}
	H_0=H_{\partial M}(O)-\bar{H}_{\partial M}(O)+2\varphi'(0)\left( 2-\frac{1}{\sqrt{a_{11}}}-\frac{1}{\sqrt{a_{22}}} \right),
	\label{eq:H0}
\end{equation}
where $H_{\partial M}(O)$ ($\bar{H}_{\partial M}(O)$ respectively) denotes the mean curvature of $\partial M$ at $O$ under metric $g$ ($\delta$ respectively).
Note that $H_0\ge 0$ by $H_{\partial M}\ge \bar{H}_{\partial M}$ and $g|_{\partial M}\ge \delta|_{\partial M}$ on $\partial M$.

\begin{proposition}
	\label{prop_sphereStrictBarrier}
	Suppose the condition in Corollary \ref{cor_meanGDelta} holds and $H_0>0$.
	Then we can choose some $s>0$,$t>0$ such that $H_{s,t}(\hat{x})>0$ for each $\hat{x} \in E_s$ and $\gamma_{s,t}(\hat{x})>\bar{\gamma}_{s,t}(\hat{x})$ for each $\hat{x} \in \partial E_s$.
\end{proposition}
\begin{proof}
	Note that $H_{s,t}$ continuously relies on $s,t$ by the definition of $\Sigma_{s,t}$.
	Then if $H_0>0$, we know we can find $s_0,t_0>0$ such that for all $0<t<t_0,0\le s\le s_0$ such that
	\[
		H_{s,t}(\hat{x})>0\quad  \text{ for all }\hat{x} \in E_s.
	\]
	Now we fix some $s \in (0,s_0)$.
	Based on Corollary \ref{cor_GreaterAngleG0}, we can find some $t>0$ small enough such that
	\[
		\gamma_{s,t}(\hat{x})>\bar{\gamma}_{s,t}(\hat{x})
	\]
	for any $\hat{x} \in \partial E_s$.
\end{proof}

\subsection{Construction of foliation}%
\label{sub:construction_of_foliation}

From this subsection, we always assume $H_0=0$, which means we assume 
\begin{equation}
	H_0=0,\quad H_{\partial M}(O)=\bar{H}_{\partial M}(O),\quad \text{ and }\quad a_{11}=a_{22}=1
	\label{eq:assumeG0}
\end{equation}
in view of \eqref{eq:H0}.

Since $\Sigma_t$ is a foliation near $O$, we can define the variational vector field $Y_t$ on $\Sigma_t$ by
\[
	Y_t(\hat{x}):=\frac{\partial }{\partial t}\Sigma_t(\hat{x}).
\]
We write $N_t$ as the unit normal vector field of $\Sigma_t$.
Given $u \in C^{1,\alpha}(\bar{E})\cap C^{2,\alpha}(E)$, we consider the surface
\[
	\Sigma_{t,u}:=\left\{ \Sigma_{t-\frac{u}{\left< Y_t(\hat{x}), N_t(\hat{x}) \right> }}(\hat{x}): \hat{x} \in E \right\}.
\]
Recall that by the definition of $\Sigma_t$, we have
\begin{align}
	Y_t(\hat{x})={} & \left(\frac{\hat{x}}{\sqrt{\varphi'(0)}},2t\hat{x}_1^2 b_1+2t\hat{x}_2^2b_2-2t\right)+O(t^2),\\
	N_t={} & \frac{-2t b_\alpha \hat{x}_\alpha e^\alpha+e^3 /\sqrt{\varphi'(0)}}{\left|-2t b_\alpha \hat{x}_\alpha e^\alpha+e^3 / \sqrt{\varphi'(0)}\right|_{g_0}}+O(t^2),\\
	Y_t\cdot N_t={}& \frac{-2t /\sqrt{\varphi'(0)}}{\left|-2t b_\alpha \hat{x}_\alpha e^\alpha+e^3/ \sqrt{\varphi'(0)}\right|_{g_0}}+O(t^2)=-\frac{2t}{|e^3|_{g_0}}+O(t^2).
 \label{eq:YDotN}
\end{align}
Hence $\Sigma_{t,u}$ is well-defined if we assume $u=o(t^2)$.
Now, let us replace $u$ by $ut^3$ and assume $u=O(1)$. Similar to \eqref{mean curvature taylor expansion} and \eqref{angle variation2}, we have
\begin{align*}
	\frac{H_{t,t^3u}}{t}={} & -\Delta_t^E u +\frac{H_t}{t}+O(t),\\
	\frac{\cos \gamma_{t,t^3u}-\cos \bar{\gamma}_{t,t^3u}}{t^3}={} & 
	-2\sqrt{\varphi'(0)}|\hat{x}| \frac{\partial u}{\partial \nu_t^E}+(A_{\partial M}(\eta_t,\eta_t)-\cos \gamma_t A(\nu_t,\nu_t)\\
&-\bar{A}_{\partial M}(\bar{\eta}_t,\bar{\eta}_t))u+
\frac{\cos \gamma_t-\cos \bar{\gamma}_t}{t^3}+O(t),
\end{align*}
where $\Delta_t^E$ is the Laplace-Beltrami operator on $E$ under the metric $\frac{1}{t^2}\Sigma_t^*(g)$, $\nu_t^E$ the outer unit normal vector field of $\partial E$ in $E$ under metric $\frac{1}{t^2}\Sigma_t^*(g)$.

\begin{corollary}
	\label{cor_secondg0}
	We have the following asymptotic expansion,
\begin{align}
		A_{\partial M}(\eta_t,\eta_t)={}&\cos \gamma_t A(\nu_t,\nu_t)+\bar{A}_{\partial M}(\bar{\eta}_t,\bar{\eta}_t) +O(t).
		\label{eq:corIIG0}
\end{align}
\end{corollary}

\begin{proof}
	Similar to Lemma \ref{lem_2ndFundForm}, we have
	\begin{align}
{} & A_{\partial M}(\eta_t,\eta_t)-\cos \gamma_t A_t(\nu_t,\nu_t)-\bar{A}_{\partial M}(\bar{\eta}_t,\bar{\eta}_t) \\
		={} & H_{t,\partial M}-\cos \gamma_t H_t-\bar{H}_{t,\partial M}-\sin \gamma_t \kappa_t + \sin \bar{\gamma}_t \bar{\kappa}_t.
		\label{eq:pfIIrewrite}
	\end{align}
	At first, we have
	\begin{equation}
		\lim_{t\rightarrow 0} \left( H_{t,\partial M}-\cos \gamma_t H_t-\bar{H}_{t,\partial M} \right)=H_{\partial M}(O)-H_0-\bar{H}_{\partial M}(O)=0
		\label{eq:pfMean}
	\end{equation}
	by our assumption \eqref{eq:assumeG0}.

	Secondly, we need to deal with $\kappa_t$.
	Similarly with the asymptotic behavior of mean curvature (for example, by the proof Corollary \ref{cor_meanGDelta}), we can show that $\kappa_t=\kappa_t^{g_0}+O(1)$ where $\kappa_t^{g_0}$ is the geodesic curvature of $\partial \Sigma_t$ under metric $g_0$.
Note that $\frac{1}{t}\partial \Sigma_t$ converges to a planar ellipse in $\left\{ x_3=0 \right\}$.
Since we assume $a_{11}=a_{22}=1$ based on \eqref{eq:assumeG0}, we know $g_0|_{\left\{ x_3=0 \right\}}=\delta|_{\left\{ x_3=0 \right\}}$.
Hence, $\kappa_t=\bar{\kappa}_t+O(1)$.
Together with Corollary \ref{cor_angleGtoD}, we have
\begin{equation}
	\lim_{t\rightarrow 0} \sin \bar{\gamma}_t(\hat{x}) \bar{\kappa}_t(\hat{x})-\sin \gamma_t(\hat{x}) \kappa_t(\hat{x})=0,
	\label{eq:pfLimGeodesic}
\end{equation}
for any $\hat{x} \in \partial E$.

Combining \eqref{eq:pfIIrewrite}, \eqref{eq:pfMean}, and \eqref{eq:pfLimGeodesic} gives us \eqref{eq:corIIG0}.
\end{proof}

%

Now we define the operator $\Phi (t, u) = (\Phi_1 (t, u), \Phi_2 (t, u))$ by
setting
\begin{align}
  \Phi_1 (t, u) & =  \frac{H_{t, t^3 u}}{t}  -
  \frac{1}{|E|_{g_0}}  \int_E  \frac{H_{t, t^3 u}}{t} 
  \mathrm{d} V_{g_0}, \\
  \Phi_2 (t, u) & = \frac{\cos \gamma_{t, t^3 u} - \cos \bar{\gamma}_{t, t^3
  u}}{t^3},
\end{align}
for $t \in (0, \varepsilon)$ where $\mathrm{d} V_{g_0}$ is the volume measure
under metric $g_0$ and $|E|_{g_0} = \int_E \mathrm{d}V_{g_0}$.

\begin{remark}
	Here, since we already assumed $a_{11}=a_{22}=1$, we know $g_0=\delta$ when restrict it on $\left\{ x_3=0 \right\}$.
	But we still write it $g_0$ instead of $\delta$ to avoid possible confusion.
\end{remark}

Similar to Proposition \ref{prop:solution}, we have the following result.

\begin{proposition}
	For any $t \in (0,\varepsilon)$ with $\varepsilon$ small enough, there exists a function $u(\cdot,t)$ such that $\Phi(t,u(\cdot,t))=0$.
	Moreover, $u(\cdot,t)$ satisfies
	\begin{equation}
		\lim_{t\rightarrow 0} (u(\hat{x},t)+u(-\hat{x},t))=0
		\label{eq:thmLimU}
	\end{equation}
	for any $\hat{x} \in E$.
	\label{prop:solutiong0}
\end{proposition}
\begin{proof}
	The existence of $u(\cdot,t)$ is similar to the proof of Proposition \ref{prop:solution}.
	We only need to prove \eqref{eq:thmLimU}.

	We write
	\begin{align*}
		f_1(\hat{x}):={} & \lim_{t\rightarrow 0} \frac{H_t(\hat{x})}{t},\\
		f_2(\hat{x}):={}& \lim_{t\rightarrow 0} \frac{\cos \gamma_t-\cos \bar{\gamma}_t}{t^3}.
	\end{align*}
	We define $u(\cdot,0)$ the (unique) solution of the following problem,
	\[
		\begin{cases}
		-\Delta_0^E u + f_1=c, & \text{ in }E,\\
		-2\sqrt{\varphi'(0)}|\hat{x}|\frac{\partial u}{\partial \nu_0^E}+f_2=0, & \text{ on }\partial E,
		\end{cases}
	\]
	with condition $\int_{ E} u dV_{g_0}=0$.
	Here, $\Delta^E_0$ is the Laplace-Beltrami operator on $E$ under metric $g_0$, $\nu_0^E$ is the unit outer unit vector field along $\partial E$ in $E$ under metric $g_0$.
	Note that since $E$ is a disk by assumption \eqref{eq:assumeG0}, we know $|\hat{x}|$ is a constant.
	Since $f_1(\hat{x})$ and $f_2(\hat{x})$ are odd symmetric, we can find $-u(-\hat{x})$ also solves above problem.
	By the uniqueness of solutions, we know
	\begin{equation}
		u(\hat{x},0)+u(-\hat{x},0)=0.
		\label{eq:pfUsymmetric}
	\end{equation}
	
	Note that by the proof of Proposition \ref{prop:solution}, we know $\lim_{t\rightarrow 0} u(\cdot,t)=u(\cdot,0)$.
	Together with \eqref{eq:pfUsymmetric}, we can prove \eqref{eq:thmLimU}.
\end{proof}

Now, we write $u_t(\cdot)=u(t,\cdot) \in C^{2,\alpha}(E)\cap C^{1,\alpha}(\bar{E})$ such that Proposition \ref{prop:solutiong0} holds. 
In particular, we know $H_{t,t^3u_t}$ is a constant function defined on $E$ by the construction of $u_t$.

\begin{proposition}
	\label{prop:CMCg0}
	We can construct a surface $\Sigma$ near $O$ such that the mean curvature of $\Sigma$ is nonnegative and it has prescribed contact angle $\bar{\gamma}$ along $\partial \Sigma$.
\end{proposition}
\begin{proof}
	First $\lim_{t\rightarrow 0} H_{t,t^3u_t}=0$ by assumtion \eqref{eq:assumeG0}.
	We write write $\lambda_t=\frac{H_{t,t^3u_t}}{t}$.
	We use the same method for proving Lemma \ref{lem:limLam} and it gives
	\begin{align*}
		\lim_{t\rightarrow 0} \lambda_t |E|_{g_0}\ge{} & \int_{ E} f_1(\hat{x})dV_{g_0}- \frac{1}{2\sqrt{\varphi'(0)}|\hat{x}|}\int_{ \partial E} f_2(\hat{x}) d\sigma_{g_0}=0,
	\end{align*}
	where $f_1,f_2$ are the functions defined in the proof of Proposition \ref{prop:solutiong0} and we have to use the symmetric properties of those functions.
        So $\lim_{t\rightarrow 0} \lambda_t\ge 0$.
	If $\lim_{t\rightarrow 0} \lambda_t>0$, then we again choose $\Sigma=\Sigma_{t,t^3u_t}$ for $t$ small enough which is a desired surface.
	For the case $\lim_{t\rightarrow 0} \lambda_t=0$, we can follow the proof for Proposition \ref{prop:positiveCMC} and find $h(t):=H_{t,t^3u_t}$ satisfies the differential inequality
	\[
		h'(t)+\Psi(t)h(t)\ge 0,
	\]
    for
	\[
		\Psi(t)=\left( \int_{ \Sigma_{t,t^3u_t}} \frac{1}{v_t} \right)^{-1} \int_{ \partial \Sigma_{t,t^3u_t}} \cot \bar{\gamma}_{t,t^3u_t},
	\]
    and $v_t:=-\left<Y_{t,t^3u_t},N_{t,t^3u_t}\right>$, $Y_{t,t^3u_t}$ is the variational vector field associated with the foliation $\Sigma_{t,t^3u_t}$ and $N_{t,t^3u_t}$ the unit normal vector field along $\Sigma_{t,t^3u_t}$ and pointing upward.
    This time, we note $\Sigma_{t,t^3u_t}$ is a higher order perturbation of $\Sigma_t$ and using \eqref{eq:YDotN}, we have
    \[
    v_t=-\left<Y_{t,t^3u_t},N_{t,t^3u_t}\right>=-\left<Y_{t},N_{t}\right>+O(t^2)=-Y_t\cdot N_t+O(t^2)=\frac{2t}{|e^3|_{g_0}}+O(t^2).
    \]
    Recall that $\frac{\sqrt{\varphi'(0)}}{t}\Sigma_{t,t^3u_t}\to E$ as sets, and $E$ is a disk with radius $(a^{33})^{\frac{1}{4}}$ by our assumption \eqref{eq:assumeG0} and the definition of $E$.
    Hence, we have
    \begin{align}
        \int_{\partial \Sigma_{t,t^3u_t}} \cot \bar{\gamma}_{t,t^3u_t}=2\pi t \frac{(a^{33})^{\frac{1}{4}}}{\sqrt{\varphi'(0)}}\cdot \frac{1}{2(a^{33})^{\frac{1}{4}}\sqrt{\varphi'(0)}t}+O(t)=\frac{\pi}{\varphi'(0)}+O(t),
    \end{align}
    in view of Remark \ref{rmk:cotG0}.
    On the other hand, we can compute
    \begin{align}
        \int_{\Sigma_{t,t^3u_t}}\frac{1}{v_t}=\pi t^2 \frac{\sqrt{a^{33}}}{\phi'(0)}\cdot \frac{|e^3|_{g_0}}{2t}+O(t^2)=\frac{\pi a^{33}}{2\phi'(0)}t+O(t^2).
    \end{align}
    Hence, we find $\Psi(t)=\frac{2}{a^{33}t}+C_1(t)$ for some bounded continuous function of order $O(1)$.
    Then, we know $h(t)$ satisfies
    \[
		\frac{\mathrm{d}}{\mathrm{d}t}\left[ \exp\left( \int_{ 0} ^t C_1(s)\mathrm{d}s \right)t^{\frac{2}{a^{33}}} h(t) \right]\ge 0.
	\]
    Combining with $\lim_{t\rightarrow 0} h(t)=0$, we get $h(t)\ge 0$ for $t \in (0,\varepsilon)$.
	Then we can choose $\Sigma=\Sigma_{t,t^3u_t}$ for any $t \in (0,\varepsilon)$.
\end{proof}

	\subsection{Proof of spherical case of main theorem}%
\label{sub:proof_of_spherical_case_of_main_theorem}

\begin{proof}[Proof of the Case \ref{case spherical} of Theorem \ref{main}]
	Similarly as the Case \ref{case conical}, we can construct a nonnegative mean curvature surface $\Sigma_{t}$ with prescribed angle $\bar{\gamma}$ based on Proposition \ref{prop:positiveCMC}, Proposition \ref{prop_sphereStrictBarrier}, and Proposition \ref{prop:CMCg0} for some small $t$.
	All the remaining proof is essentially the same as the Case \ref{case conical}.
\end{proof}

\section{Generalizations}\label{generalizations}

In this section, we discuss generalizations of Theorem \ref{main} to the
hyperbolic space and space forms with $\mathbb{S}^1$ symmetry. We list the
cases which we are unable to handle.

\subsection{Hyperbolic case}

First, we describe the model. We use the upper half space model of the
hyperbolic 3-space $\mathbb{H}^3$ which is given by the metric
\begin{equation}
  \bar{g} = \tfrac{1}{(x^3)^2} ((\mathrm{d} x^1)^2 + (\mathrm{d} x^2)^2 +
  (\mathrm{d} x^3)), x \in \mathbb{R}^3_+, x^3 > 0.
\end{equation}
Let $t_{\pm}$ be two positive real numbers satisfying $t_- < t_+$, we assume
that the surface $\partial M$ is rotationally symmetric with respect to the
$x^3$-axis and lies between the two coordinate planes
\[ P_{\pm} = \{x \in \mathbb{H}^3 : x^3 = t_{\pm} \} \]
and $\partial M \cap P_{\pm}$ are nonempty. As before, $\partial M \cap
P_{\pm}$ has three types of geometries.

\begin{theorem}
  \label{hyperbolic}Let $(M^3, g)$ be a compact 3-manifold with its scalar
  curvature satisfying the lower bound
  \begin{equation}
    R_g \geq - 6.
  \end{equation}
  such that its boundary $\partial M$ is diffeomorphic to a weakly convex
  surface in $(\mathbb{R}^3_+,\delta)$ and rotationally symmetric with respect to the
  $x^3$-axis. The boundary $\partial M$ bounds a region $\bar{M}$ (which we
  call a model or a reference) in $\mathbb{H}^3$, let the induced metric of
  the hyperbolic metric be $\bar{\sigma}$ and the induced metric of $g$ on
  $\partial M$ be $\sigma$. We assume that $\sigma \geq \bar{\sigma}$ and
  $H_{\partial M} \geq \bar{H}_{\partial M}$ on $\partial M \cap \{x \in
  \mathbb{H}^3 : t_- < x^3 < t_+ \}$.
  We also assume $\partial M$ satisfies one of the following conditions near $p_\pm $.
  \begin{enumerate}
    \item $\partial M \cap P_{\pm}$ is a disk. For this case, we further assume that
    $H_{\partial M} \geq \mp 2$ at $\partial M \cap P_{\pm}$ and the
    dihedral angles forming by $P_{\pm}$ and $\partial M\backslash (P_+ \cup
    P_-)$ are no greater than the hyperbolic reference.
    
    \item $\partial M$ is conical at $p_{\pm}$.
\item $\partial M$ is spherical at $p_{\pm}$. 
  \end{enumerate}
  Then $(M, g)$ is hyperbolic.
\end{theorem}

\begin{remark}
The assumption that $\partial M$ is diffeomorphic to a weakly convex
surface in $(\mathbb{R}^3_+,\delta)$ is used to prove a comparison similar to \eqref{derivative of angle comparison}. See Remark \ref{angle comparison}. 
  \end{remark}

\begin{proof}
  We give some key steps for the case $\partial M \cap P_{\pm}$ are both
  disks. We consider the functional
  \begin{equation}
    I (E) = | \partial^{\ast} E \cap \ensuremath{\operatorname{int}}M| +
    2\mathcal{H}^3 (E) - \int_{\partial^{\ast} E \cap \partial M} \cos
    \bar{\gamma},
  \end{equation}
  where $\bar{\gamma}$ is the angle between the unit normal pointing outward
  of $(\bar{M}, \bar{g})$ and $\tfrac{\partial}{\partial x^3}$. We find a
  minimiser $E$ to this functional whose boundary $\Sigma := \partial E \cap
  \ensuremath{\operatorname{int}}M$ which is of mean curvature $- 2$ and
  satisfies
  \[ \langle X, N \rangle - \cos \bar{\gamma} = 0 \text{ along } \partial
     \Sigma . \]
  The stability gives
  \begin{equation}
    Q (f, f) := - \int_{\Sigma} (f \Delta f + (|A|^2
    +\ensuremath{\operatorname{Ric}}(N)) f^2) + \int_{\partial \Sigma} f
    (\tfrac{\partial f}{\partial \nu} - q f) \geq 0
  \end{equation}
  where $q$ is the same as \eqref{q}. By {\cite{witten-connectedness-1999}} or
  {\cite{andersson-rigidity-2008}}, we have that
  \[ |A|^2 +\ensuremath{\operatorname{Ric}} (N) = \tfrac{R_g + 6}{2} - K +
     \tfrac{|A^0 |^2}{2}, \]
  where $A^0$ is the traceless part of $A$; and from \eqref{boundary sy},
  similar to the steps obtaining \eqref{taking constant} and \eqref{after
  gauss-bonnet}, we obtain
\begin{align}
2 \pi \chi (\Sigma) = & \int_{\Sigma} K + \int_{\partial \Sigma} \kappa
\\
\geq & \int_{\partial \Sigma} \tfrac{H_{\partial M}}{\sin
\bar{\gamma}} + 2 \cot \bar{\gamma} + \tfrac{1}{\sin^2 \bar{\gamma}}
\partial_{\eta} \cos \bar{\gamma} + \tfrac{1}{2} \int_{\Sigma} R_g + 6 +
|A^0 |^2 \\
\geq & \int_{\partial \Sigma} \tfrac{\bar{H}_{\partial M}}{\sin
\bar{\gamma}} + 2 \cot \bar{\gamma} + \tfrac{1}{\sin^2 \bar{\gamma}}
\partial_{\bar{\eta}} \cos \bar{\gamma} .
\end{align}
  Let $\ell_s$ be the curve $\{x \in \partial M : x^3 = s\}$, we denote by
  $\kappa_s$ the the geodesic curvature of $\ell_s$ in the plane $\{x^3 =
  s\}$. We claim that
  \begin{equation}
    \tfrac{\bar{H}_{\partial M}}{\sin \bar{\gamma}} + 2 \cot \bar{\gamma} +
    \tfrac{1}{\sin^2 \bar{\gamma}} \partial_{\bar{\eta}} \cos \bar{\gamma} =
    \kappa_s . \label{hyperbolic decomposition}
  \end{equation}
  We prove this claim \text{{\itshape{differently}}} from Lemma \ref{lemma
  lower bound over separating curve}. We observe the model $(\bar{M},
  \bar{g})$ is foliated by parallel disks which is of mean curvature $- 2$
  (computed with respect to the unit normal pointing to the same direction as
  $\tfrac{\partial}{\partial x^3}$) with $\langle X, N \rangle - \cos
  \bar{\gamma} = 0$ on $\partial \Sigma$ for every leaf and hence stable. The
  variational vector field $Y$ of the foliation can be easily written down. We
  easily see that $\langle Y, N \rangle$ is a constant. Using $f = \langle Y,
  N \rangle$, $H = - 2$ and the rewrites \eqref{boundary sy} in \eqref{angle
  first variation} and \eqref{background angle first variation} leads to our
  claim.
  
  The bound $\int_{\Sigma} \kappa_s \geq 2 \pi$ is easy following the
  same lines as in Lemma \ref{lem_curve_inte}. The rest is fairly similar to
  the Euclidean case. We refer to Section \ref{slab case}.
  
  The conical and spherical cases are similar to the Euclidean case in Section
  \ref{conical case} and Section \ref{sec:spherical} (see {\cite{li-polyhedron-2020}} and
  {\cite{chai-dihedral-2022-arxiv}}). We would like to mention, in proving the
  hyperbolic analog of Lemma \ref{limit of mean}, it is preferable to use directly the
  metric $\bar{g}$, which is comparatively simpler than
  {\cite{chai-dihedral-2022-arxiv}} where we used the Euclidean metric (which
  is conformal to $\bar{g}$). Of course, they lead to the same bound.
\end{proof}

\subsection{Spaces with $\mathbb S^1$-symmetry}

Let $\tilde{x} = (x^2, x^3)$ represent an arbitrary point in $\mathbb{R}^2$,
$t_{\pm}$ be two positive real numbers with $t_- < t_+$, $P_0 = (z_0, t_-)$,
$P_1 = (z_1, t_-)$, $P_2 = (z_2 , t_+)$ and $P_3 = (z_3, t_+)$ be four points
on $\mathbb{R}^2$ where $z_0 < z_1$ and $z_2  < z_3$. We connect $P_0$ with
$P_1$ and $P_2$ with $P_3$ by straight segments. We connect $P_0$ with $P_2$
and $P_1$ with $P_3$ by smooth embedded curves. Then these curves bound a
region $W$ in the plane $W$. We use $P_i P_{i + 1}$ or $P_{i + 1} P_i$ to
denote the segments from $P_i$ to $P_{i + 1}$ (we set $P_4 : = P_0$) where $i$
ranges from $0$ to $3$ despite that not all $P_i$ are straight.

We consider the model $\bar{M} =\mathbb{S}^1 \times W$, and we use $\theta$ to
represent the circle factor. Let $\bar{g}$ be the standard metric $(\mathrm{d}
\theta)^2 + (\mathrm{d} x^2)^2 + (\mathrm{d} x^3)^2$ on $\bar{M}$ or the
hyperbolic metric
\[ \tfrac{1}{(x^3)^2} ((\mathrm{d} \theta)^2 + (\mathrm{d} x^2)^2 +
   (\mathrm{d} x^3)^2) . \]
Let $\tau$ be either the number $- 1$ or 0. Let $\bar{\gamma}$ be the angles
formed by the vector field $\bar{X}$ and $\tfrac{\partial}{\partial x^3}$ at
$\mathbb{S}^1 \times (\{x^3 = t\} \cap \partial M)$, $t_- < t < t_+$. It
extends to $\mathbb{S}^1 \times P_i$ by continuity. For simplicity, we set $M$
to be $\bar{M}$. We assume that $\mathbb{S}^1 \times P_0 P_2$ and
$\mathbb{S}^1 \times P_1 P_3$ are weakly convex in $M$ with the flat metric. Again the weak convexity is used to prove the comparison relation \eqref{derivative of angle comparison}.

\begin{theorem}
  \label{circle factor}Let $g$ be another smooth metric on $M$. If $g$ has
  scalar curvature satisfying $R_g \geq 6 \tau$, $H_{\partial
  M} \geq \bar{H}_{\partial M}$ on all $\partial M$, $\sigma \geq
  \bar{\sigma}$ on $\mathbb{S}^1 \times P_0 P_2$ and $\mathbb{S}^1 \times P_1
  P_3$, and the dihedral angles formed by the four pieces of $\partial M$ are
  no greater than the model. Then $g$ is of constant sectional curvature $\tau$.
\end{theorem}

\begin{proof}
  We consider the functional
  \begin{equation}
    I (E) = | \partial^{\ast} E \cap \ensuremath{\operatorname{int}}M| -
    \int_{\partial^{\ast} E \cap \partial M} \cos \bar{\gamma} .
  \end{equation}
  We find a minimiser $E$ to this functional whose boundary $\Sigma :=
  \partial E \cap \ensuremath{\operatorname{int}}M$ which is of mean curvature
  $2 \tau$ and satisfies
  \[ \langle X, N \rangle - \cos \bar{\gamma} = 0 \text{ along } \partial
     \Sigma \]
  The stability gives
  \begin{equation}
    Q (f, f) := - \int_{\Sigma} (f \Delta f + (|A|^2
    +\ensuremath{\operatorname{Ric}}(N)) f^2) + \int_{\partial \Sigma} f
    (\tfrac{\partial f}{\partial \nu} - q f) \geq 0
  \end{equation}
  where $q$ is the same as \eqref{q}. Similar to the steps obtaining
  \eqref{taking constant} and \eqref{after gauss-bonnet}, we obtain
\begin{align}
2 \pi \chi (\Sigma) = & \int_{\Sigma} K + \int_{\partial \Sigma} \kappa
\\
\geq & \int_{\partial \Sigma} \tfrac{H_{\partial M}}{\sin
\bar{\gamma}} - 2 \tau \cot \bar{\gamma} + \tfrac{1}{\sin^2 \bar{\gamma}}
\partial_{\eta} \cos \bar{\gamma} + \tfrac{1}{2} \int_{\Sigma} R_g + 6
\tau + |A^0 |^2 \\
\geq & \int_{\partial \Sigma} \tfrac{\bar{H}_{\partial M}}{\sin
\bar{\gamma}} - 2 \tau \cot \bar{\gamma} + \tfrac{1}{\sin^2 \bar{\gamma}}
\partial_{\bar{\eta}} \cos \bar{\gamma} .
\end{align}
  Similar to \eqref{hyperbolic decomposition}, we obtained that
  \[ \tfrac{\bar{H}_{\partial M}}{\sin \bar{\gamma}} - 2 \tau \cot
     \bar{\gamma} + \tfrac{1}{\sin^2 \bar{\gamma}} \partial_{\bar{\eta}} \cos
     \bar{\gamma} = 0. \]
  We know that $\Sigma$ is an annulus, so $\chi (\Sigma) = 0$. From here, we
  can trace back the equalities as Section \ref{slab case}, and we omit the
  details.
\end{proof}

\subsection{Unsolved cases}


One interesting question is when $\partial M$ consists more than two
pieces forming dihedral angles and that its geometry at $p_{\pm}$ are not disks. In Figure \ref{fig:multiple-pieces-of-pm}, we draw an example of three pieces with different geometries at $p_{\pm}$.

\begin{figure}[ht]
    \centering
	\begingroup
	\def\svgwidth{0.3\columnwidth}
\begingroup%
  \makeatletter%
  \providecommand\color[2][]{%
    \errmessage{(Inkscape) Color is used for the text in Inkscape, but the package 'color.sty' is not loaded}%
    \renewcommand\color[2][]{}%
  }%
  \providecommand\transparent[1]{%
    \errmessage{(Inkscape) Transparency is used (non-zero) for the text in Inkscape, but the package 'transparent.sty' is not loaded}%
    \renewcommand\transparent[1]{}%
  }%
  \providecommand\rotatebox[2]{#2}%
  \newcommand*\fsize{\dimexpr\f@size pt\relax}%
  \newcommand*\lineheight[1]{\fontsize{\fsize}{#1\fsize}\selectfont}%
  \ifx\svgwidth\undefined%
    \setlength{\unitlength}{211.28138505bp}%
    \ifx\svgscale\undefined%
      \relax%
    \else%
      \setlength{\unitlength}{\unitlength * \real{\svgscale}}%
    \fi%
  \else%
    \setlength{\unitlength}{\svgwidth}%
  \fi%
  \global\let\svgwidth\undefined%
  \global\let\svgscale\undefined%
  \makeatother%
  \begin{picture}(1,1.54945171)%
    \lineheight{1}%
    \setlength\tabcolsep{0pt}%
    \put(0,0){\includegraphics[width=\unitlength,page=1]{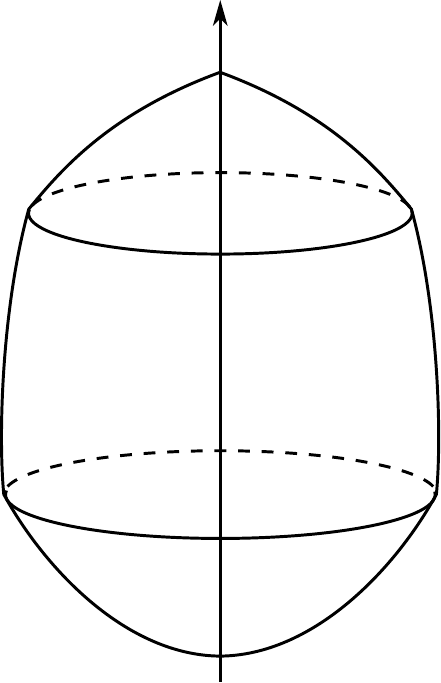}}%
    \put(0.32216138,1.44661517){\color[rgb]{0,0,0}\makebox(0,0)[lt]{\lineheight{0}\smash{\begin{tabular}[t]{l}$x_3$\end{tabular}}}}%
    \put(0,0){\includegraphics[width=\unitlength,page=2]{multiple-pieces-of-pm.pdf}}%
    \put(0.51645331,0.10510106){\color[rgb]{0,0,0}\makebox(0,0)[lt]{\lineheight{0}\smash{\begin{tabular}[t]{l}$p_-$\end{tabular}}}}%
    \put(0.51645331,1.41141644){\color[rgb]{0,0,0}\makebox(0,0)[lt]{\lineheight{0}\smash{\begin{tabular}[t]{l}$p_+$\end{tabular}}}}%
  \end{picture}%
\endgroup%

	\endgroup

    \caption{Multiple pieces of $\partial M$}
    \label{fig:multiple-pieces-of-pm}
\end{figure}

In Theorem \ref{circle factor}, we have assumed in the model that $P_0$,
$P_1$ are of positive distance from each other. It is worthwhile to look into
the case $P_0 = P_1$ similar to the conical case in Theorem \ref{main}, and an
even more difficult case is that when $\partial M$ is smooth at $\mathbb{S}^1
\times \{P_0 = P_1 \}$. We can assume that the curve connecting $P_0$ and
$P_2$ is only piecewise smooth.

Lastly, we would like to remark that our proof relied on the Gauss-Bonnet theorem and the regularity result {\cite[Theorem
1.3]{de-philippis-regularity-2015}} of minimal capillary surface, it is desirable to prove Theorem \ref{main} in higher dimensions.

\appendix
\section{Rotationally symmetric surface in $\mathbb R^3$}

Let $\vec{x} : [0, \varepsilon) \times \mathbb{S}^1 \to \mathbb{R}^3$ be a
rotationally symmetric surface $S$ given by
\[ \vec{x} (\rho, \theta) = (\psi (\rho) \cos \theta, \psi (\rho) \sin \theta,
   - \rho) . \]
Then the tangent vectors are
\[ \vec{x}_{\rho} = (\psi' \cos \theta, \psi' \sin \theta, - 1),
   \vec{x}_{\theta} = (- \psi \sin \theta, \psi \cos \theta) . \]
The unit normal to $S$ is
\begin{equation}
  X = (\cos \theta, \sin \theta, \psi') / \sqrt{1 + (\psi')^2} .
\end{equation}
The angle $\bar{\gamma}$ is given by
\begin{equation}
  \cos \bar{\gamma} = \tfrac{\psi'}{\sqrt{1 + (\psi')^2}} .
\end{equation}
And
\begin{equation}
  \bar{\eta} = \vec{x}_{\rho} / \sqrt{1 + (\psi')^2} .
\end{equation}
Let $A_S$ denote the second fundamental form, then
\begin{equation}
  A_S (\vec{x}_{\rho}, \vec{x}_{\rho}) = - \tfrac{\psi''}{\sqrt{1 +
  (\psi')^2}}, A_S (\vec{x}_{\theta}, \vec{x}_{\theta}) = \tfrac{1}{\sqrt{1 +
  (\psi')^2}} \psi . \label{std second fundamental form}
\end{equation}

Let $\Sigma_{\rho} = \{(\psi (\rho) \hat{x}, - \rho) : \hat{x} \in D\}$. We
use the subscript $\rho$ on every geometric quantities on $\Sigma_{\rho}$, for
example, $\gamma_{\rho}$ denotes the contact angles of $\Sigma_{\rho}$ with
$\partial M$ under the metric $g$ and $\bar{\gamma}_{\rho}$ denotes the
contact angles of $\Sigma_{\rho}$ with $\partial M$ under the flat metric.

Let
\[ \Sigma_{\rho, \rho^2 u} = \{(\psi (\rho + \rho^2 u) \hat{x},  - (\rho +
   \rho^2 u)) : \hat{x} \in D\}, u = u (\cdot, \rho) \in C^{2, \alpha} (D) \]
be a small perturbation of $\Sigma_{\rho}$, we use the subscript $\rho, \rho^2
u$ on every geometric quantities on $\Sigma_{\rho, \rho^2 u}$ except that
$\bar{\gamma}_{\rho, \rho^2 u}$ denotes the value of $\bar{\gamma}_{\rho}$ at
$\Sigma_{\rho, \rho^2 u}$. From the definition of $\bar{\gamma}_{\rho, \rho^2
u}$,
\[ \cos \bar{\gamma}_{\rho, \rho^2 u} = \tfrac{\psi' (\rho + \rho^2
   u)}{\sqrt{1 + (\psi' (\rho + \rho^2 u))^2}}, \hat{x} \in \partial D. \]

It is easy to see that for each $\hat{x} \in \partial D$,
\[ \cos \bar{\gamma}_{\rho, \rho^2 u} - \cos \bar{\gamma}_{\rho} =
   \tfrac{1}{\sqrt{1 + (\psi' (\rho))^2}} \tfrac{1}{1 + (\psi' (\rho))^2}
   \psi'' (\rho) \rho^2 u (\hat{x}, \rho) + O (\rho^3) . \]
It follows from \eqref{std second fundamental form} that
\begin{equation}
  \cos \bar{\gamma}_{\rho, \rho^2 u} - \cos \bar{\gamma}_{\rho} = -
  \bar{A}_{\partial M} (\bar{\eta}, \bar{\eta}) \rho^2 u + O (\rho^3) .
  \label{std angle difference}
\end{equation}
Note that we also have \eqref{background angle difference near pole}. In fact,
\[ \partial_{\bar{\eta}} \cos \bar{\gamma} = \bar{A}_{\partial M} (\bar{\eta},
   \bar{\eta}) \sin^2 \bar{\gamma} \]
on the boundary level set $\{(\psi (\rho) \hat{x},  - \rho) : \hat{x} \in
\partial D\}$.

\section{Mean curvatures and dihedral angles of cones in $\mathbb R^3$}%
\label{appdsec:meanAngleCone}
Suppose the surface $\Sigma$ in $\mathbb{R}^3$ is given by
	\begin{align*}
		\vec{x}(\rho,\theta)={}& (a_1\rho \cos \theta, a_2\rho\sin \theta,-c\rho)\quad\text{ for } \rho \in (0,+\infty), \theta \in [0,2\pi]
	\end{align*}
	for some $a_1,a_2,c>0$. The tangent vectors are
        \begin{align}
		\vec{x}_{\rho}:={} & \frac{\partial \vec{x}}{\partial \rho}=(a_1 \cos \theta,a_2\sin \theta,-c),\\
		\vec{x}_{\theta}:={} & \frac{\partial \vec{x}}{\partial \theta}=(-a_1\rho\sin \theta,a_2 \rho \cos \theta,0).
          \end{align}
So the metric $\sigma$ is given by          
\begin{align}
\sigma&:={} \begin{bmatrix}
                 \vec{x}_{\rho} \cdot \vec{x}_{\rho}
                & \vec{x}_{\rho} \cdot \vec{x}_{\theta} \\
                \vec{x}_{\theta} \cdot \vec{x}_{\rho}&
                 \vec{x}_{\theta} \cdot \vec{x}_{\theta}
		\end{bmatrix} \\
		&=\begin{bmatrix}
			a_1^2\cos^2 \theta+a_2^2\sin^2 \theta+c^2 
			& (a_2^2-a_1^2)\rho \sin \theta \cos \theta\\
			(a_2^2-a_1^2)\rho \sin \theta \cos \theta& 
			\rho^2(a_1^2\sin ^2\rho+a_2^2\cos ^2\theta)
		\end{bmatrix}.
\end{align}
The unit normal $X$ pointing outside the cone is given by         
\[
  X=\frac{(a_2c \cos \theta,a_1c \sin \theta, a_1a_2)}{\sqrt{a_1^2a_2^2+c^2(a_2^2\cos \theta+a_1^2\sin ^2\theta)}}.
  \]
Let $A_\Sigma$ denote the second fundamental form, we see that the only nonzero component of $A_\Sigma$ is 
\[
  A_\Sigma (\vec{x}_\theta, \vec{x}_\theta) =-\vec{x}_{\theta\theta}\cdot X =\frac{- a_1 a_2 c \rho}{\sqrt{a_1^2 a_2^2 + c^2  (a_2^2 \cos \theta + a_1^2
\sin^2 \theta)}}.
  \]
Hence, the mean curvature of $\Sigma$ is given by
\begin{equation}
	H_\Sigma=\mathrm{tr}({\sigma}^{-1}{A_\Sigma})=
		\frac{a_1a_2c (a_1^2\cos^2 \theta+a_2^2\sin^2 \theta+c^2)}{\rho(a_1^2a_2^2+c^2(a_2^2\cos ^2\theta+a_1^2\sin ^2\theta))^{\frac{3}{2}}}.
	\label{append:mean}
\end{equation}
	The contact angle between $\Sigma$ and plane $\left\{ z=-c \right\}$ is $\arccos (X \cdot e_3)$ and we know that
	\begin{equation}
 X\cdot e_3= \frac{a_1a_2}{\sqrt{a_1^2a_2^2+c^2(a_2^2\cos ^2\theta+a_1^2\sin ^2\theta)}}.\label{append:angle}
	\end{equation}
\bibliographystyle{alpha} 
\bibliography{refs} 
\end{document}